\documentclass[letterpaper]{amsart}

\newtheorem{theorem}{Theorem}[section]
\newtheorem{lemma}[theorem]{Lemma}

\newtheorem{proposition}[theorem]{Proposition}
\newtheorem{corollary}[theorem]{Corollary}

\theoremstyle{definition}
\newtheorem{definition}[theorem]{Definition}
\newtheorem{example}[theorem]{Example}

\theoremstyle{remark}
\newtheorem{remark}[theorem]{Remark}

\usepackage{graphicx}
\usepackage{color}
\usepackage{amsmath}
\usepackage{amsfonts}
\usepackage{amssymb}
\usepackage{epsfig}
\usepackage{rotating}
\usepackage{graphics}
\newcommand{\be}{\begin{equation}}

\newcommand{\ee}{\end{equation}}

\newcommand{\al}{\alpha}

\newcommand{\bet}{\beta}

\newcommand{\Om}{\Omega}

\newcommand{\om}{\omega}

\newcommand{\la}{\lambda}

\newcommand{\dz}{\wedge}

\newcommand{\ba}{\begin{array}}

\newcommand{\ea}{\end{array}}

\newcommand{\beq}{\begin{eqnarray}}

\newcommand{\eeq}{\end{eqnarray}}

\newtheorem{lm}{lemma}

\newtheorem{thee}{theorem}

\newtheorem{proo}{proposition}

\newtheorem{co}{corollary}

\newtheorem{rem}{remark}

\newtheorem{deff}{definition}

\newcommand{\bd}{\begin{deff}}

\newcommand{\ed}{\end{deff}}

\newcommand{\bl}{\begin{lm}}

\newcommand{\el}{\end{lm}}

\newcommand{\bp}{\begin{proo}}

\newcommand{\ep}{\end{proo}}

\newcommand{\bt}{\begin{thee}}

\newcommand{\et}{\end{thee}}

\newcommand{\bc}{\begin{co}}

\newcommand{\ec}{\end{co}}

\newcommand{\brm}{\begin{rem}}

\newcommand{\erm}{\end{rem}}

\newcommand{\der}{{\rm d}}

\usepackage{t1enc}
\def\frak{\mathfrak}

\newcommand{\newc}{\newcommand}

\let\ccdot\cdot
\def\cdot{\hbox to 2.5pt{\hss$\ccdot$\hss}}

\newc{\aR}{\mbox{\boldmath{$ R$}}}
\newc{\aS}{\mbox{\boldmath{$ S$}}}
\newc{\aT}{\mbox{\boldmath{$ T$}}}
\newc{\aW}{\mbox{\boldmath{$ W$}}}

\newc{\aK}{\mbox{\boldmath{$ K$}}}
\newc{\aL}{\mbox{\boldmath{$ L$}}}
\newcommand{\bbC}{\mathbb{C}}
\newcommand{\bbR}{\mathbb{R}}

\usepackage{amssymb}
\usepackage{amscd}

\newcommand{\bma}{\begin{pmatrix}}
\newcommand{\ema}{\end{pmatrix}}
\newcommand{\tc}{(\Om_2-\bar{\Om}_2)}
\newcommand{\ap}{(\Om_2+\bar{\Om}_2)}
\newcommand{\tj}{\theta^1}
\newcommand{\td}{\theta^2}
\newcommand{\ttr}{\theta^3}

\newc{\obstrn}[2]{B^{#1}_{#2}}

\newcommand{\rpl}                         % +) or <+
{\mbox{$
\begin{picture}(12.7,8)(-.5,-1)
\put(0,0.2){$+$}
\put(4.2,2.8){\oval(8,8)[r]}
\end{picture}$}}

\newcommand{\lpl}                         % (+ or +>
{\mbox{$
\begin{picture}(12.7,8)(-.5,-1)
\put(2,0.2){$+$}
\put(6.2,2.8){\oval(8,8)[l]}
\end{picture}$}}

\usepackage{ifthen}

\newc{\tensor}[1]{#1}
\newc{\Mvariable}[1]{\mbox{#1}}
\newc{\down}[1]{{}_{#1}}
\newc{\up}[1]{{}^{#1}}

\newc{\JulyStrut}{\rule{0mm}{6mm}}
\newc{\midtenPan}{\mbox{\sf S}}
\newc{\midten}{\mbox{\sf T}}
\newc{\midtenEi}{\mbox{\sf U}}
\newc{\ATen}{\mbox{\sf E}}
\newc{\BTen}{\mbox{\sf F}}
\newc{\CTen}{\mbox{\sf G}}

\def\sideremark#1{\ifvmode\leavevmode\fi\vadjust{\vbox to0pt{\vss% the remark
 \hbox to 0pt{\hskip\hsize\hskip1em%                          will appear only
 \vbox{\hsize3cm\tiny\raggedright\pretolerance10000%          on the side
 \noindent #1\hfill}\hss}\vbox to8pt{\vfil}\vss}}}%
\newcommand{\vper}{{\mathcal V}^\perp}
\newcommand{\vv}{{\mathcal V}}
\newcommand{\bgw}{{\textstyle \bigwedge}}

\numberwithin{equation}{section}

\newcounter{romenumi}
\newcommand{\labelromenumi}{(\roman{romenumi})}

\begin{document}
\title{Intrinsic geometry of oriented congruences in three dimensions}
\vskip 1.truecm
\author{C. Denson Hill} \address{Department of Mathematics, Stony
  Brook University, Stony Brook, N.Y. 11794, USA}
\email{dhill@math.sunysb.edu}  
\author{Pawe\l~ Nurowski} \address{Instytut Fizyki Teoretycznej,
Uniwersytet Warszawski, ul. Hoza 69, Warszawa, Poland}
\email{nurowski@fuw.edu.pl}
\thanks{This research was supported by
the Polish grant 1P03B 07529}

\date{\today}

\begin{abstract}
Starting from the classical notion of an oriented congruence (i.e. a foliation by oriented curves) in $\bbR^3$, we abstract the notion of an \emph{oriented congruence structure}. This is a 3-dimensional CR manifold $(M,H, J)$ with a preferred splitting of the tangent space $TM={\mathcal V}\oplus H$. We find all local invariants of such structures using Cartan's equivalence method refining Cartan's classification of 3-dimensional CR structures. We use these invariants and perform Fefferman like constructions, to obtain interesting Lorentzian metrics in four dimensions, which include explicit Ricci-flat and Einstein metrics, as well as not conformally Einstein Bach-flat metrics.  
\end{abstract}
\maketitle
%*************
\tableofcontents
\newcommand{\bbS}{\mathbb{S}}
\newcommand{\sog}{\mathbf{SO}}
\newcommand{\slg}{\mathbf{SL}}
\newcommand{\og}{\mathbf{O}}
\newcommand{\soa}{\frak{so}}
\newcommand{\sla}{\frak{sl}}
\newcommand{\sua}{\frak{su}}
\newcommand{\dr}{\mathrm{d}}
\newcommand{\sug}{\mathbf{SU}}
\newcommand{\gat}{\tilde{\gamma}}
\newcommand{\Gat}{\tilde{\Gamma}}
\newcommand{\thet}{\tilde{\theta}}
\newcommand{\Thet}{\tilde{T}}
\newcommand{\rt}{\tilde{r}}
\newcommand{\st}{\sqrt{3}}
\newcommand{\kat}{\tilde{\kappa}}
\newcommand{\kz}{{K^{{~}^{\hskip-3.1mm\circ}}}}
\newcommand{\bv}{{\bf v}}
\newcommand{\di}{{\rm div}}
\newcommand{\curl}{{\rm curl}}
%*************
\section{Introduction}
We study the local differential geometry of oriented 
congruences in 3-dimensional manifolds. This geometry turns out to 
be very closely related to local 3-dimensional \emph{CR geometry}. The latter 
can be traced back to Elie Cartan's 1932 papers \cite{cartan}, in which he used his \emph{equivalence method} to determine the \emph{full set of local invariants} of locally embedded 3-dimensional strictly pseudoconvex 
CR manifolds.

This paper should be regarded as an extension and refinement of Cartan's work. This is because a 3-dimensional manifold with an oriented congruence on it is an abstract 3-dimensional CR manifold with an additional structure: \emph{a preferred splitting} (see Section \ref{orjco}). This leads to a notion of local equivalence of such structures, which is more strict that than of Cartan. Hence the (coarse) CR equivalence classes of Cartan split into a \emph{fine structure}; as a result we produce many \emph{new local invariants}, corresponding to many 
more \emph{nonequivalent} structures than in Cartan's situation.

From this perspective, our paper may be also placed in the realm of 
\emph{special geometries}, i.e. geometries with an additional structure. These kind of geometries, such as, for example, special Riemannian geometries (hermitian, K\"ahler, $G_2$, etc.), find applications in mathematical physcis (e.g. string theory). The starting point of this paper also comes from physics: a congruence 
in $\bbR^3$ (i.e. a foliation of $\bbR^3$ by curves) is a notion that appears in hydrodynamics (velocity flow), Newtonian gravity and 
electrodynamics (field strength lines). These branches of physics have 
distinguished the two main invariants of such foliations, which are related to the classical notions of \emph{twist} and \emph{shear}. One of the byproducts of our 
analysis is also a refinement of these physical concepts.      

Contemporary physicists, because of the dimension of spacetime, 
have been much more interested in congruences in \emph{four} dimensions. Such 
congruences live in \emph{Lorentzian} manifolds, and as such, may 
be \emph{timelike}, \emph{spacelike} or \emph{null}. It turns out that 
the \emph{null} congruences in spacetimes, which are tangent to \emph{unparametrized 
geodesics without shear}, locally define a 3-manifold, which has a CR structure on it. One of the outcomes of this paper is that we found connections between  
properties of \emph{four} dimensional spacetimes admitting null and shearfree 
congruences, with their corresponding \emph{three} dimensional CR manifolds, and \emph{our new} invariants 
of the classical congruences in \emph{three} dimensions. In Sections \ref{cab} and 
\ref{cabb}, in particular, we use these \emph{three} dimensional invariants, to construct interesting families of Lorentzian metrics with shearfree congruences in \emph{four} dimensions (including metrics which are Ricci flat or Einstein, Bach flat but not conformal to Einstein, etc.).

Throughout the paper we will always have a nondegenerate (not
neccessarily Riemannian) metric $g_{ij}$ and its inverse
$g^{ij}$. This enables us to freely raise and lower indices at our
convenience. We use the Einstein summation convention.
We also denote by
$\omega_1\omega_2=\tfrac{1}{2}(\omega_1\otimes\omega_2+\omega_2\otimes\omega_1)$
the symmetrized tensor product of two 1-forms $\omega_1$
and $\omega_2$. In this paper we shall be working in the smooth
category; i.e., everything will be assumed to be $C^\infty$, without
mentioning it explicitly in what follows.

A large part of the paper is based on lengthy calculations, which are required by our main tool, namely \emph{Cartan's equivalence method}. These 
calculations were checked by the symbolic calculation program Mathematica. 
The structure of the paper is reflected in the table of contents.

We have been
inspired by our contact with Andrzej Trautman, Jacek Tafel and 
Jerzy Lewandowski, whom we thank warmly. We also thank the Mathematisches Institute der Humboldt-Universit\"at zu Berlin, and Ilka Agricola and J\"urgen Leiterer, in particular, for their kind hospitality during the preparation of this paper.

\section{Classical twist and shear}\label{cts}
In a simply connected domain $U$ of Euclidean space $\bbR^3$, equipped
with the flat metric
$g_{ij}=\delta_{ij}$, we
consider a smooth foliation by uniformly oriented curves. Let $\bv$ be
a vector field $\bv=v^i\nabla_i$ tangent to the foliation, consistent with the
orientation.

%Further properties of $\bv$ can be discussed by analyzing the first
%derivatives of $\bv$ in all possible directions.
We denote the total symmetrization by round brackets on the indices,
the total antisymetrization by square brackets on the indices, and use
$\epsilon_{ijk}=\epsilon_{[ijk]}$, $\epsilon_{123}=1$. We have the
following classical decomposition
\be
\nabla_iv_j=\alpha_{ij}+\sigma_{ij}+\tfrac{1}{3}\theta g_{ij},\label{deco}
\ee
where
$$\alpha_{ij}=\nabla_{[i}v_{j]}=\tfrac{1}{2}\epsilon_{ijk}(\curl~\bv)^k,$$
$$\theta=g^{ij}\nabla_iv_j=\di~\bv,$$
$$\sigma_{ij}=\nabla_{(i}v_{j)}-\tfrac{1}{3}\theta g_{ij}.$$
The decomposition (\ref{deco}) defines three functions, depending on the
choice of $\bv$, which can be used to characterize the foliation.
One of these functions is
the divergence $\theta$, also called the
{\it expansion of the vector field} $\bv$. It merely characterizes the
vector field $\bv$, hence it is {\it not} interesting as far the
properties of the foliation is concerned. The second function is
$$\alpha=|\alpha_{ij}|=\sqrt{g^{ik}g^{jl}\alpha_{ij}\alpha_{kl}},$$
the norm of the antisymmetric part $\alpha_{ij}$, called the
{\it twist of the vector field} $\bv$.

Vanishing of twist,  the {\it twist-free} condition $\alpha= 0$, is
equivalent to $\curl ~\bv= 0$. Although this condition is
$\bv$-dependent, it has a clear geometric meaning for the foliation.
Indeed,  a  vector field $\bv$ with
vanishing twist may be
represented by a gradient: $\bv=\nabla f$ for some function
$f:U\to\bbR$. In such a case the level
surfaces of the function $f$ define a foliation of $U$ with
2-dimensional leaves orthogonal to $\bv$. This can be rephrased by
saying that the distribution $\vper$ of 2-planes,
perpendicular to $\bv$, is integrable.

The third function obtained from the decomposition (\ref{deco}) is
$$\sigma=|\sigma_{ij}|=\sqrt{g^{ik}g^{jl}\sigma_{ij}\sigma_{kl}},$$ the norm of
the trace-free symmetric part $\sigma_{ij}$, called the
{\it shear of the vector field} $\bv$.

%The notion of {\it shear} comes from elasticity theory.
Regardless of whether or not $\vper$ is integrable, the condition of
{\it vanishing shear} $\sigma= 0$ is equivalent to
$\nabla_{(i}v_{j)}=\tfrac{1}{3}\theta g_{ij}$. Recalling that the Lie
derivative
%${\mathcal L}_\bv g_{ij}$ is
${\mathcal L}_\bv g_{ij}=\nabla_{(i}v_{j)}$, we see
that the shear-free condition for $\bv$ is the condition that this Lie
derivative be proportional to the metric. Thus
$\sigma=0$ if and only if ${\mathcal L}_\bv g_{ij}=h g_{ij}.$ This
condition again is $\bv$ dependent. However, it implies the following
geometric property of the foliation: the metric $g_{|\vper}$ induced
by $g_{ij}$ on the distribution $\vper$ is
conformally preserved when Lie transported along $\bv$. To say it
differently we introduce a complex structure $J$ on each 2-plane of
$\vper$. This is possible since each such plane is
equipped with a metric $g_{|\vper}$ and the orientation
induced by the orientation of $\bv$. Knowing this, we define $J$ on
each 2-plane as a rotation by
$\tfrac{\pi}{2}$, using the right hand rule. Now we can rephrase the statement about conformal
preservation of the metric $g_{|\vper}$ during Lie
transport along $\bv$, by saying that it is equivalent to the
constancy of $J$ under the Lie
transport along $\bv$.

The above notions of expansion, twist and shear are the classical
notions of elasticity theory. As we have seen, they are not invariants
of the foliation by curves, because they depend on the choice of the
vector field $\bv$. Nonetheless they do carry some invariant
information. One of the main purposes of this paper is to find all of the local
invariants of the intrinsic geometry associated with such foliations. 
With this classical motivation we now pass to the subject proper of this paper.

\section{Oriented congruences}\label{orjco}
Consider a smooth oriented real 3-dimensional manifold $M$ equipped with a
Riemannian metric $g$. Assume that $M$ is smoothly foliated by
uniformly oriented curves. Such a
foliation is called an {\it oriented congruence}. Note that we are
{\it not} assuming that the curves in the congruence are geodesics for
the metric $g$.

Our first observation is that $M$ has the structure of a smooth
abstract {\it CR manifold}. To see this we introduce the oriented line
bundle ${\vv}$, a subbundle of $TM$, consisting of the tangent lines to the
foliation. Using the metric we also have $\vper$, the
2-plane subbundle of $TM$ consisting of the planes orthogonal to the
congruence. These 2-planes are oriented by the right hand rule and are
equipped with the induced metric $g_{|\vper}$. Hence
$\vper$ is endowed with the complex structure operator $J$ as we
explained in the previous section. The pair $(\vper,J)$, by the very
definition, equips $M$
with the structure of an abstract 3-dimensional CR manifold. This CR
manifold has an additional structure consisting in the prefered
splitting $TM=\vper\oplus\vv$. It also defines an equivalence class $[g]$ of
{\it adapted} Riemannian metrics $g'$ in which $g'(\vv,\vper)=0$ and such that $g'_{|\vper}$
is hermitian for $J$. Thus, an oriented congruence in $(M,g)$ defines
a whole class of Riemannian manifolds $(M,[g])$ which are adapted to it.

Conversely, given an oriented abstract 3-dimensional CR manifold $(M,H,J)$ with a
distinguished line subbundle $\vv$ such that $\vv\cap H=\{0\}$, we may
reconstruct the oriented congruence. The curves of this congruence
consist of the trajectories of $\vv$. They are
oriented by the right hand rule applied in such a way that it agrees
with the orientation of $H$ determined by $J$.
Here $J: H\to H$ and $J^2=-$id. Since T$M=H\oplus\vv$ we recover also the equivalence class $[g]$
of adapted Riemannian metrics $g'$ in which $g(\vv,H)=0$ and such that
$g'_{|H}$ is hermitian for $J$.

We summarize with: let $M$ be an oriented 3-dimensional
manifold, then
\begin{proposition}
There is a one to one correspondence between oriented
congruences on $M$ with a distinguished orthogonal distribution
$\vv^\perp$, and CR structures $(H,J)$ on $M$ with a
distinguished line
subbundle $\vv$ such that $TM=H\oplus\vv$.
\end{proposition}

We now pass to the dual formulation. Given a CR structure $(H,J)$ with a
prefered splitting $TM=H\oplus\vv$, we define $H^0$ to be the anihilator
of $H$ and $\vv^0$ to be the anihilator of $\vv$. Note that $H^0$ is a
{\it real line} subbundle of $T^*M$ and $\vv^0$ is a 2-plane subbundle of
$T^*M$. This $H^0$ is known as the {\it characteristic bundle}
associated with the CR structure. $\vv^0$ is equipped with the complex
structure $J^*$, the adjoint of $J$ with respect to the natural
duality pairing. The complexification $\bbC \vv^0$ splits into $\bbC
\vv^0=\vv^0_+\oplus\vv^0_-$,  where $\vv^0_\pm$ are the $\mp i$ eigenspaces of
$J^*$. Both spaces $\vv^0_\pm$ are {\it complex line} subbundles of
the complexification $\bbC T^*M$ of the cotangent bundle. $\vv^0_-$ is
the complex conjugate of $\vv^0_+$, $\overline{\vv}^0_\pm=\vv^0_{\mp}$.

The reason for passing to the dual formulation is that we want to
apply Cartan's method of equivalence to determine the local
invariants of an oriented congruence in $M$. For this we need a local
nonzero section $\lambda$ of $H^0$ and a local nonzero section $\mu$ of
$\vv^0_+$. Then $\lambda\dz\mu\dz\bar{\mu}\neq 0$. Any other local section
$\lambda'$ of $H^0$ and any other local section $\mu'$ of $\vv^0_+$ are
related to $\lambda$ and $\mu$ by
$\lambda'=f \lambda$ and $\mu'=h\mu$, for some real function $f$
and some complex function $h$. This motivates the following
definition:
\begin{definition}\label{d3}
A {\it structure} $(M,[\lambda,\mu])$
{\it of an oriented congruence} on a 3-dimensional manifold
$M$ is an equivalence class of pairs of 1-forms $[\lambda,\mu]$ on $M$
satisfying the following conditions:
\begin{itemize}
\item[{\it (i)}] $\lambda$ is real, $\mu$ is complex
\item[{\it (ii)}] $\lambda\dz\mu\dz\bar{\mu}\neq 0$ at each point of $M$
\item[{\it (iii)}] two pairs $(\lambda,\mu)$ and $(\lambda',\mu')$ are equivalent iff
  there exist nonvanishing functions $f$
  (real) and $h$ (complex) on $M$ such that
\be
\lambda'=f\lambda,\quad\quad\quad \mu'=h\mu.\label{ec}
\ee
\end{itemize}
We say that two such structures
$(M,[\lambda,\mu])$ and  $(M',[\lambda',\mu'])$ are (locally)
{\it equivalent}
iff there exists a (local) diffeomorphism $\phi: M\to M'$ such that
\be
\phi^*(\lambda')=f\lambda,\quad\quad\quad \phi^*(\mu')=h\mu\label{efi}
\ee
for some nonvanishing functions $f$ (real) and $h$ (complex) on $M$.
If such a diffeomorphism is from $M$ to $M$ it is called an {\it
  automorphism} of $(M,[\lambda,\mu])$. The full set of automorphisms
is called the {\it group of automorphisms} of $(M,[\lambda,\mu])$.
A vector field $X$ on $M$ is
called a {\it symmetry} of $(M,[\lambda,\mu])$
iff
$${\mathcal L}_X\lambda=f\lambda,\quad\quad\quad {\mathcal
  L}_X\mu=h\mu.$$
Here the functions $f$ (real) and $h$ (complex) are not required to be
  nonvanishing; they may even vanish identically. Observe, that if $X$
  and $Y$ are two symmetries of $(M,[\lambda,\mu])$ then their
  commutator $[X,Y]$ is also a symmetry. Thus, we may speak about the
  {\it Lie algebra of symmetries}.
\end{definition}

\begin{remark}\label{carrem}
Note that E. Cartan \cite{cartan} would define a 3-dimensional CR
manifold as a structure $(M,[\lambda,\mu])$ as above, with the
exception that condition {\it (iii)} is weakend to
\begin{itemize}
\item[${\it (iii)}_{CR}$] two pairs $(\lambda,\mu)$ and $(\lambda',\mu')$
  are equivalent iff
  there exist nonvanishing functions $f$
  (real) and $h$ (complex) and a complex function $p$ on $M$ such that
$$\lambda'=f\lambda,\quad\quad\quad \mu'=h\mu+p \lambda.$$
\end{itemize}
In this sense our structure of an oriented congruence
 $(M,[\lambda,\mu])$ is a CR manifold on which there is an
additional structure. In particular the diffeomorphisms $\phi$ that provide an
 equivalence of our structures are special cases of CR
 diffeomorphisms,
which for CR structures defined {\it a la} Cartan by ${\it (iii)_{CR}}$ are $\phi:M\to M'$
 such that
$\phi^*(\lambda')=f\lambda$, $\phi^*(\mu')=h\mu+p\lambda$. In terms of
 the nowadays definition of a CR manifold as a triple $(M,H,J)$, this
 last Cartan condition is equivalent to the CR map requirement:
$\der\phi\circ J=J \circ\der\phi$ and similarly for $\phi^{-1}$.
\end{remark}
\begin{remark}
Two CR structures which are not equivalent in the sense of
Cartan \cite{cartan} are also not equivalent, in our sense, as oriented congruences;
but not vice versa. On the other hand, every symmetry of an oriented congruence $(M,[\lambda,\mu])$
is a CR symmetry of the CR structure determined by
$[\lambda,\mu]$ via ${\it (iii)}_{CR}$; and not vice versa.
\end{remark}
We omit the proof of the following easy proposition.
\begin{proposition}\label{adda}
A given structure $(M,[\lambda,\mu])$
determines a CR structure $(M,H,J)$ with the preferred splitting
$TM=H\oplus\vv$, where $H$ is the annihilator
of $Span_\bbR (\lambda)$ and $\bbC\vv$ is the annihilator of
$Span_\bbC(\mu)\oplus Span_\bbC(\bar{\mu})$. The class of adapted
Riemannian metrics $[g]$ is parametrized by two arbitrary nonvanishing
functions $f$ (real) and $h$ (complex) and given by
$$g=f^2 \lambda^2+2|h|^2\mu\bar{\mu}.$$
\end{proposition}
\section{Elements of Cartan's equivalence method}

Here we outline the procedure we will follow in applying Cartan's
method to our particular situation.

\subsection{Cartan invariants}
Consider two structures $(M,[\lambda,\mu])$ and
$(M',[\lambda',\mu'])$. Our aim is to determine whether they are
equivalent or not, according to Definition \ref{d3}, equation (\ref{efi}).
This question is not easy to answer, since it is equivalent to the
problem of the existence of a solution $\phi$ for a system (\ref{efi})
of linear
first order PDEs in which the right hand side is undetermined.
Elie Cartan associates with the forms
$(\lambda,\mu,\bar{\mu})$ and $(\lambda',\mu',\bar{\mu}')$,
representing the structures, two
systems of {\it ordered coframes} $\{\Om_i\}$ and $\{\Om'_i\}$
on manifolds $P$ and $P'$ of the same
dimension, say $n\geq 3$, which are fiber bundles over $M$. Then he shows that equations like
(\ref{efi}) for
$\phi$ have a solution {\it if and only if} a simpler system
\be
\Phi^*\Om'_i=\Om_i,\quad\quad i=1,2,...,n\label{eq:roo1}
\ee
 of differential equations for a
diffeomorphism $\Phi:P\to P'$ has a solution. Note that derivatives of
$\Phi$ still occur in (\ref{eq:roo1}), since $\Phi^*$ is the pullback
of forms from $P'$ to $P$.

One famous example is
his original solution to the equivalence problem for 3-dimensional
strictly pseudoconvex CR structures. There $P$ and $P'$ are
8-dimensional, and his procedure produces two systems of eight
linearly independent 1-forms $\{\Omega_i\}$ and  $\{\Omega_i'\}$.

In our situation, provided $n<\infty$, and if we are able to find $n$
well defined linearly independent 1-forms $\{\Omega_i\}$ on $P$, then
$(P,\{\Omega_i\})$ provides the {\it full system of local invariants} for
the original structure $(M,[\lambda,\mu])$. In particular, using
$(P,\{\Omega_i\})$ one introduces the {\it scalar invariants}, which are the
coefficients $\{K_I\}$ in the decomposition of $\{\der\Omega_i\}$ with
respect to the {\it invariant} basis of 2-forms
$\{\Omega_i\dz\Omega_j\}$.

Now in order to determine if two structures
$(M,[\lambda,\mu])$ and $(M',[\lambda',\mu'])$ are equivalent, it is
enough to have $n$ functionally independent $\{K_I\}$. Then the
condition (\ref{eq:roo1}) becomes
\be
\Phi^* {K'}_I=K_I,\quad\quad I=1,2,...,n.\label{k}
\ee
The advantage of this condition, as compared to (\ref{eq:roo1}), is
that (\ref{k}), being the pull back of functions, does not involve
derivatives of $\Phi$. In this case the existence of $\Phi$ becomes a
question involving the implicit function theorem, and the whole
problem reduces to checking whether a certain Jacobian is non-degenerate.

We remark that an immediate application of
the invariants obtained by Cartan's
equivalence method is to use them to find {\it all} the
homogeneous examples of the particular structure under consideration. 
The procedure of
enumerating these examples is straightforward and algorithmic once
the Cartan invariants have been determined. In our
situation the homogeneous examples will often have local symmetry groups of
dimension {\it three}. The 3-dimensional Lie groups are classified
according to the Bianchi classification of 3-dimensional Lie
algebras \cite{bianchi}. Since we will use this classification in subsequent
sections, 
we recall it below.
\subsection{Bianchi classification of 3-dimensional Lie algebras}
In this section $X_1,X_2,X_3$ denote a basis of a 3-dimensional Lie
algebra $\mathfrak g$ with Lie bracket $[\cdot,\cdot]$. All the
nonequivalent Lie algebras fall into {\it Bianchi types} $I$, $II$,
$VI_0$, $VII_0$, $VIII$, $IX$, $V$, $IV$, $VI_h$, $VII_h$. Apart from
types $VI_h$ and $VII_h$, there is always precisely one Lie algebra
corresponding to a given type. For each value of the real parameter
$h<0$ there is also precisely one Lie algebra of type $VI_h$. Likewise for each value of the parameter $h>0$ there is precisely one Lie algebra of type $VII_h$. 
The commutation relations for each Bianchi type are given in
the following table.\\

\begin{tabular}{|c|c|c|c|c|c|c|}
\hline 
Bianchi type:&$I$&$II$&$VI_0$&$VII_0$&$VIII$&$IX$\\
\hline
$[X_1,X_2]~=$&$0$&$0$&$0$&$0$&$-X_3$&$X_3$\\
$[X_3,X_1]~=$&$0$&$0$&$-X_2$&$X_2$&$X_2$&$X_2$\\
$[X_2,X_3]~=$&$0$&$X_1$&$X_1$&$X_1$&$X_1$&$X_1$\\
\hline
\hline
Bianchi type:&$V$&$IV$&$VI_h$&$VII_h$&&\\
\hline
$[X_1,X_2]~=$&$0$&$0$&$0$&$0$&&\\
$[X_3,X_1]~=$&$X_1$&$X_1$&$-X_2+h X_1$&$X_2+hX_1$&&\\
$[X_2,X_3]~=$&$-X_2$&$X_1-X_2$&$X_1-hX_2$&$X_1-hX_2$&&\\
\hline
\end{tabular}\\

Note that Bianchi type I corresponds to the abelian Lie group, type II
corresponds to the Heisenberg group; types VIII and IX correspond to
the simple groups: $\sog(1,2)$, $\slg(2,\bbR)$ for type
VIII, and $\sog(3)$, $\sug(2)$ for type IX.  
\section{Basic relative invariants of an oriented congruence}\label{sec5}
We make preparations to apply the Cartan method of equivalence for
finding all local invariants of the structure of
an oriented congruence $(M,[\lambda,\mu])$ on a 3-manifold $M$.

Given a structure $(M,[\lambda,\mu])$ we take representatives
$\lambda$ and $\mu$ of 1-forms from the class $[\lambda,\mu]$. Since
$(\lambda,\mu,\bar{\mu})$ is a basis of 1-forms on $M$ we can express
the differentials $\der\lambda$ and $\der\mu$ in terms of the
corresponding basis of 2-forms
$(\mu\dz\bar{\mu},\mu\dz\lambda,\bar{\mu}\dz\lambda)$. We have
\beq
\der\lambda&=&i
a\mu\dz\bar{\mu}+b\mu\dz\lambda+\bar{b}\bar{\mu}\dz\lambda\nonumber\\
\der\mu&=&p\mu\dz\bar{\mu}+q\mu\dz\lambda+s\bar{\mu}\dz\lambda\label{str}\\
\der\bar{\mu}&=&-\bar{p}\mu\dz\bar{\mu}+\bar{s}\mu\dz\lambda+\bar{q}\bar{\mu}\dz\lambda,\nonumber
\eeq
where $a$ is a real valued function and  $b,p,q,s$ are complex valued
functions on $M$. Given any function $u$ on $M$ we define
first order linear partial differential operators acting on $u$ by
$$\der u=u_\la\lambda+u_\mu\mu+u_{\bar{\mu}}\bar{\mu}.$$
Note that $u_\la$ is a real vector field acting on $u$, $u_\mu$ is a
complex vector field of type (1,0) acting on $u$ and $u_{\bar{\mu}}$ is a
complex vector field of type (0,1) acting on $u$. The commutators of
these operators, when acting on $u$ are
\beq
u_{\bar{\mu}\mu}-u_{\mu\bar{\mu}}&=&-ia u_\la-p u_\mu+\bar{p}u_{\bar{\mu}}\nonumber\\
u_{\la\mu}-u_{\mu\la}&=&-bu_\la-q
u_\mu-\bar{s}u_{\bar{\mu}}\label{com}\\
u_{\la\bar{\mu}}-u_{\bar{\mu}\la}&=&-\bar{b}u_\la-s u_\mu-\bar{q}u_{\bar{\mu}}.\nonumber
\eeq
A function $u$ on a CR manifold $(M,[\la,\mu])$ is called a {\it CR
  function} if
\be\der u\dz\la\dz\mu\equiv 0.\label{tcr1}\ee
In terms of the
differential operators above this is the same as
\be u_{\bar{\mu}}\equiv
0.\label{tcr2}\ee
Thus $u_{\bar{\mu}}$ is just the {\it tangential Cauchy-Riemann
  operator} acting on $u$. The equation (\ref{tcr1}) or (\ref{tcr2})
is called the {\it tangential Cauchy-Riemann equation}.

It is easy to see that each of the following two conditions
\be
\der\lambda\dz\lambda=0,\quad\quad \der\mu\dz\mu=0,\label{twistshear}
\ee
is independent of the choice of the respresentatives $(\lambda,\mu)$
from the class $[\lambda,\mu]$. Thus the identical vanishing or not of
either the
coefficient $a$, or the coefficent $s$, is an invariant property of the structure
$(M,[\lambda,\mu])$. Using Cartan's terminology the functions $a$ and
$s$ are the {\it basic relative invariants} of $(M,[\lambda,\mu])$. By
definition they
correspond to the identical vanishing or not of the {\it twist} (the function $a$) and of the
{\it shear} (the function $s$) of the oriented congruence represented by
$(M,[\lambda,\mu])$.

They are invariant versions of the classical $\bv$-dependent notions of twist
$\alpha$ and shear $\sigma$ we considered in Section \ref{cts}. Given
an oriented congruence with vanishing twist $a$ in $M=\bbR^3$ we
can always find a vector field $\bv$ tangent to the congruence such
that the twist $\alpha$ for this vector field is zero. We also have an
analogous statement for $s$ and $\sigma$. Conversely, every vector
field $\bv$ in $\bbR^3$ which
has vanishing twist $\alpha$ (or shear $\sigma$) defines an oriented
congruence with vanishing twist $a$ (or shear $s$).

We note that the twist $a$ is just the {\it Levi form} of the CR
structure and that the shear $s$ is now complex; its meaning will be
explained further in Section \ref{vtns}.

In what follows we will often use the following
(see e.g. \cite{lnt})
\begin{lemma}\label{le}
Let $\mu$ be a smooth complex valued 1-form defined locally in
$\bbR^3$ such that $\mu\dz\bar{\mu}\neq 0$. Then
$$\der \mu\dz\mu \equiv 0\quad{\rm
if~ and~ only~ if}\quad \mu=h\der\zeta$$
where $\zeta$ is a smooth complex function such that
$\der\zeta\dz\der\bar{\zeta}\neq 0$, and $h$ is a smooth nonvanishing complex
function.
\end{lemma}
\begin{proof}
Consider an open set
$U\in\bbR^3$ in which we have $\mu$ such that
$\der\mu\dz\mu=0$ and $\mu\dz\bar{\mu}\neq 0$. We define {\it real} 1-forms
$\theta^1={\rm Re}(\mu)$ and $\theta^2={\rm
  Im}(\mu)$. They satisfy $\theta^1\dz\theta^2\neq 0$ in $U$. Since
$U\subset\bbR^3$ we trivially have
$\der\theta^1\dz\theta^1\dz\theta^2\equiv 0$ and
$\der\theta^2\dz\theta^1\dz\theta^2\equiv 0$. Now
the real Fr\"obenius theorem implies that there exists a coordinate
chart $(x,y,u)$ in $U$ such that $\theta^1=t_{11}\der x+t_{12}\der y$
and $\theta^2=t_{21}\der x+t_{22}\der y$, with some {\it real}
functions $t_{ij}$ in $U$ such that $t_{11}t_{22}-t_{12}t_{21}\neq
0$. Thus in the coordinates $(x,y,u)$ the form
$\mu=\theta^1+i\theta^2$
can be written as $\mu=c_1\der x+ c_2 \der
y$, where now $c_1$, $c_2$ are {\it complex} functions such that
$c_1\bar{c}_2-\bar{c}_1c_2\neq 0$ on $U$, so neither $c_1$ nor $c_2$
can be zero. The $\der\mu\dz\mu\equiv 0$
condition for $\mu$ written in this representation is simply
$c_2^2\der(\tfrac{c_1}{c_2})\dz\der x\dz\der y\equiv 0$. Thus the
partial derivative $(\frac{c_1}{c_2})_u\equiv 0$, which means that
the ratio $\frac{c_1}{c_2}$ does not depend on $u$. This ratio defines
a nonvanishing {\it complex} function $F(x,y)=\frac{c_1}{c_2}$ of only
{\it two} real variables $x$ and $y$. Returning to
$\mu$ we see that it is of the form $\mu=c_2\big(\der y+F(x,y)\der
x\big)$. Consider the real bilinear symmetric
form $G=2\mu\bar{\mu}=|c_2|^2\Big(\der y^2+2\big(F(x,y)+\bar{F}(x,y)\big)\der
x\der y+|F(x,y)|^2\der
x^2\Big)$.
%\beq
%G&=&2\mu\bar{\mu}\nonumber\\
%&=&|c_2|^2\Big(\der y^2+2\big(F(x,y)+\bar{F}(x,y)\big)\der
%x\der y+F(x,y)\bar{F}(x,y)\der
%x^2\Big).\nonumber
%\eeq
Invoking the
classical theorem on the existence of isothermal coordinates we are
able to find an open set $U'\subset
U$ with new coordinates $(\xi,\eta,u)$ in which $G=h^2(\der
\xi^2+\der \eta^2)$, where $h=h(\xi,\eta,u)$ is a real function in
$U'$. This means that in these coordinates $\mu=h\der(
\xi+i\eta)=h\der\zeta$. The proof in the other direction is obvious.
\end{proof}
\section{Vanishing twist and shear} \label{vts} Let us assume that the
structure $(M,[\lambda,\mu])$ satisfies both conditions
(\ref{twistshear}); i.e., that $a\equiv 0$ and $s\equiv 0$. It is easy
to see that all such structures have no local invariants, meaning that
all of them are locally equivalent. Indeed, if
$\der\lambda\dz\lambda\equiv 0$ then the real Fr\"obenius theorem
guarantees that locally $\lambda=f\der u$. Similarly, if
$\der \mu\dz\mu\equiv 0$, then the Lemma \ref{le} assures that
$\mu=h\der\zeta$. Since $\der\zeta\dz\lambda\dz\mu\equiv 0$, we see
that the function $\zeta$ is a {\it holomorphic} coordinate. Recalling the fact that
$\lambda\dz\mu\dz\bar{\mu}\neq 0$, we conclude that
if $a\equiv 0$ and $s\equiv 0$ then the CR manifold $M$ with the
prefered splitting is locally equivalent to $\bbR\times \bbC$,
with local coordinates $(u,\zeta)$, such that $u$ is
real. 
In these coordinates the structure may be represented by $\lambda=\der u$ and
$\mu=\der\zeta$. The local group of automorphisms for such structures
is infinite dimensional and given in terms of two functions $U=U(u)$
and $Z=Z(\zeta)$ such that $U$ is real, $U_u\neq 0$, $Z$ is
holomorphic and $Z_\zeta\neq 0$. The automorphism transformations are
then $\tilde{u}=U(u)$, $\tilde{\zeta}=Z(\zeta)$. Note that from the
point of view of Cartan's method this is the {\it involutive} case in which
$n=\infty$. There are no local invariants in this situation.

\section{Nonvanishing twist and vanishing shear} \label{ntvs}
\subsection{The relative invariants $K_1$ and $K_2$}\label{ncrn}
Next let us
assume that the structure $(M,[\lambda,\mu])$ has some twist,
$a\neq 0$, but has identically vanishing shear, $s\equiv 0$. Let us interpret this in terms of the corresponding CR
structure with the prefered splitting. The nonvanishing twist condition $\der
\lambda\dz\lambda\neq 0$ is the condition that the CR
structure has nonvanishing Levi form. This means that the CR manifold
is {\it strictly pseudoconvex} and hence is not locally equivalent to
$\bbR\times\bbC$. The no shear condition,
$\der\mu\dz\mu\equiv 0$, by the Lemma \ref{le}, means that the class
$[\mu]$ may be represented by a 1-form $\mu=\der\zeta$ with a complex
function $\zeta$ on $M$ satisfying $\der\zeta\dz\der\bar{\zeta}\neq
0$. Note that this function trivially satisfies the {\it tangential
  Cauchy-Riemann equation} $\der\zeta\dz\lambda\dz\mu=0$ for this CR
structure, and hence is a {\it CR function}. If $Z$ is {\it any} holomorphic function with nonvanishing derivative,
then $Z=Z(\zeta)$ is again a CR function with $\der
Z\dz\der\bar{Z}\neq 0$. This gives us a distinguished class of
genuinely complex CR functions $Z=Z(\zeta)$, which we denote by
$[\zeta]$. Conversely if we have a
{\it strictly pseudoconvex} 3-dimensional CR structure $(M,H,J)$ with
a distinguished class $[\zeta]$ of CR functions $Z=Z(\zeta)$, such that
$\der\zeta\dz\der\bar{\zeta}\neq 0$ and $Z'\neq 0$, then this CR
structure defines a representative $(\lambda,\mu=\der Z)$, with
$\lambda$ being a nonvanishing section of the characteristic bundle
$H^0$. This in turn defines a structure $(M,[\lambda,\mu])$ of an
oriented congruence which has $a\neq 0$ and $s\equiv 0$.

Summarizing we have
\begin{proposition}
All local structures of an oriented congruence $(M,[\lambda,\mu])$ with
nonvanishing twist, $a\neq 0$, and vanishing shear, $s\equiv 0$, are in a
one to one correspondence with local CR structures $(M,H,J)$ having
nonvanishing Levi form and possessing a distinguished class $[\zeta]$
of genuinely complex CR functions on $M$.
\end{proposition}

Note that the proposition remains true if we drop the nonvanishing
twist condition on the left and drop the nonvanishing Levi form
condition on the right.

We now pass to the determination of the local invariants of
$(M,[\lambda,\mu])$ with nonvanishing twist and vanishing shear. We
take a representative $(\lambda,\mu)$. Because of our
assumptions the formulae (\ref{str}) become
\beq
\der\lambda&=&i
a\mu\dz\bar{\mu}+b\mu\dz\lambda+\bar{b}\bar{\mu}\dz\lambda\nonumber\\
\der\mu&=&p\mu\dz\bar{\mu}+q\mu\dz\lambda\label{str1}\\
\der\bar{\mu}&=&-\bar{p}\mu\dz\bar{\mu}+\bar{q}\bar{\mu}\dz\lambda.\nonumber
\eeq
For example if we were to choose $\mu$ as $\mu=\der\zeta$, where $\zeta$
is a particular representative of the distinguished class $[\zeta]$ of
CR functions, then $\der\mu$ would identically vanish, so
$p\equiv 0$ and $q\equiv 0$. Although this choice of $\mu$ is very
convenient and quite simplifies the determination of the invariants,
we will work in the most general representation (\ref{str1}) of
$[\lambda,\mu]$ to get the formulae for the invariants in their full
generality.

Given a choice $(\lambda,\mu)$ as in (\ref{str1}) we take the most
general representatives
\be
\omega=f\lambda,\quad\quad\omega_1=h\mu,\quad\quad
\bar{\omega}_1=\bar{h}\bar{\mu},\label{theta}
\ee
of the class $[\lambda,\mu]$. Here $f\neq 0$ (real) and $h\neq 0$
(complex) are arbitrary functions. Then we reexpress the
differentials $\der\omega$, $\der\omega_1$ and $\der\bar{\omega}_1$ in
terms of the general basis $(\omega,\omega_1,\bar{\omega}_1)$. We
have:
\beq
\der\om&=&i\frac{fa}{|h|^2}~\om_1\dz\bar{\om}_1~+~[~\der\log f+\frac{b}{h}\om_1+\frac{\bar{b}}{\bar{h}}\bar{\om}_1~]\dz\om\label{eq1}\\
\der\om_1&=&[~\der\log
  h-\frac{p}{\bar{h}}\bar{\om}_1-\frac{q}{f}\om~]\dz\om_1\\
\der\bar{\om}_1&=&[~\der\log
  \bar{h}-\frac{\bar{p}}{h}\om_1-\frac{\bar{q}}{f}\om~]\dz\bar{\om}_1
\eeq
Since $a\neq 0$ we can easily achieve
\be
\der\om\dz\om=i\om_1\dz\bar{\om}_1\dz\om\label{n0}
\ee
by taking
\be
f=\frac{|h|^2}{a}.\label{f}
\ee
Thus condition (\ref{n0}) `fixes the gauge' in the
choice of $f$.

Introducing the real functions $\rho>0$ and $\phi$ via $h=\rho{\rm
  e}^{i\phi}$ and maintaining the condition (\ref{n0}) we may
rewrite equation (\ref{eq1}) in the form
$$\der\om=i\om_1\dz\bar{\om}_1+(\Omega+\bar{\Omega})\dz\om,$$
where the real valued 1-form $\Omega+\bar{\Omega}$ is \be
\Omega+\bar{\Omega}= 2\der\log\rho+(b-(\log
a)_\mu)\mu+(\bar{b}-(\log
a)_{\bar{\mu}})\bar{\mu}+t\lambda.\label{omp}\ee The real function
$t$ appearing in $\Omega+\bar{\Omega}$ can be
  determined {\it algebraically} from the condition that
\be
(\der\om_1+\der\bar{\om}_1)\dz(\om_1-\bar{\om}_1)=-\om_1\dz\bar{\om}_1\dz(\Omega+\bar{\Omega}).\label{n1}
\ee
If this condition is imposed then
\be
t=-q-\bar{q}.\label{t}
\ee
Now, if $t$ is as in (\ref{t}) and $f$ is as in (\ref{f}) we define
$\Omega-\bar{\Omega}$ to be an imaginary 1-form such that
\be
(\der\om_1+\der\bar{\om}_1)\dz(\om_1+\bar{\om}_1)=
\om_1\dz\bar{\om}_1\dz(\Omega-\bar{\Omega}).\label{n2}
\ee
This determines $\Omega-\bar{\Omega}$ to be
$$\Omega-\bar{\Omega}=2i\der\phi+(\bar{q}-q)\la+z\mu-\bar{z}\bar{\mu},$$
where $z$ is a still undetermined function. The condition that fixes $z$
in an algebraic fashion is the requirement that
\be
\der\om_1=\Omega\dz\om_1,\quad\quad\quad
\der\bar{\om}_1=\bar{\Omega}\dz\bar{\om}_1.\label{n3}
\ee
If this is imposed we have
\be
z=2\bar{p}+b-(\log a)_\mu,\quad\quad\quad \bar{z}=2p+\bar{b}-(\log a)_{\bar{\mu}}.
\ee
Thus given a structure $(M,[\lambda,\mu])$ with nonvanishing twist and
vanishing shear, the four normalization conditions (\ref{n0}),
(\ref{n1}), (\ref{n2}), (\ref{n3}) uniquely specify a 5-dimensional
manifold $P$, which is locally $M\times\bbC$, and a well
defined coframe $(\omega,\omega_1,\bar{\omega}_1,\Omega,\bar{\Omega})$
on it such that
\beq
\om&=&\frac{\rho^2}{a}\lambda\nonumber\\
\om_1&=&\rho {\rm e}^{i\phi}\mu\nonumber\\
\bar{\om}_1&=&\rho {\rm e}^{-i\phi}\bar{\mu}\label{n77}\\
\Omega&=&\der\log\rho+i\der\phi+(\bar{p}+b-(\log
a)_\mu)\mu-p\bar{\mu}-q\lambda\nonumber\\
\bar{\Omega}&=&\der\log\rho-i\der\phi-\bar{p}\mu+(p+\bar{b}-(\log
a)_{\bar{\mu}})\bar{\mu}-\bar{q}\lambda.\nonumber
\eeq
Here the complex coordinate along the factor $\bbC$ in $M\times\bbC$
is $h=\rho{\rm e}^{i\phi}$.
The coframe $(\omega,\omega_1,\bar{\omega}_1,\Omega,\bar{\Omega})$
satisfies
\beq
\der\om&=&i\om_1\dz\bar{\om}_1+(\Omega+\bar{\Omega})\dz\om\nonumber\\
\der\om_1&=&\Om\dz\om_1\nonumber\\
\der\bar{\om}_1&=&\bar{\Om}\dz\bar{\om}_1\label{syste}\\
\der\Omega&=&K_1\om_1\dz\bar{\om}_1+K_2\om_1\dz\om\nonumber\\
\der\bar{\Omega}&=&-K_1\om_1\dz\bar{\om}_1+\overline{K}_2\bar{\om}_1\dz\om,\nonumber
\eeq where
\be
K_1=\frac{1}{\rho^2}k_1,\quad\quad K_2=\frac{{\rm
    e}^{-i\phi}}{\rho^3}k_2,\label{k22}
\ee
are functions on $P$ with $k_1$ and $k_2$ given by \beq
k_1&=&{\rm Re}\Big((\log a)_{\mu\bar{\mu}}-(\log a)_\mu p-iqa-b_{\bar{\mu}}+bp-2\bar{p}_{\bar{\mu}}+2|p|^2\Big)\nonumber\\
k_2&=&a_{\mu\la}-ab_\la+i(\log
a)_\mu(b_{\bar{\mu}}-\bar{b}_\mu-bp+\bar{b}\bar{p})-2a_\mu
q-aq_\mu-(a\bar{q})_\mu-ab\bar{q}.\nonumber \eeq Note that the
functions $k_1$ and $k_2$ are actually defined on $M$. Note also
that $k_1$ is {\it real} as a consequence of the commutation
relations (\ref{com}). The functions $K_1$ and $K_2$ are the {\it
relative invariants} of the structure $(M,[\lambda,\mu])$, and
(\ref{syste}) are the {\it structural equations} for
$(M,[\lambda,\mu])$.
%
%Note that due to the equation $\der^2\Omega\equiv 0$ the function
%$K_I$ is {\it real}. It can be put to the form
%$$
%.K_I&=&\rho^{-2}\Big((\log a)_{1\bar{1}}-b_{\bar{1}}-(\log a)_1
%p+bp-p_1-\bar{p}_{\bar{1}}+2|p|^2-i a q\Big)
%$$
\begin{theorem}
A given structure $(M,[\lambda,\mu])$ of an oriented congruence with
nonvanishing twist, $a\neq 0$, and vanishing shear, $s\equiv 0$,
uniquely defines a 5-dimensional manifold $P$, 1-forms
$\om,\om_1,\bar{\om}_1,\Om,\bar{\Om}$
and functions  $K_1,K_2,\overline{K}_2$ on $P$
such that
\begin{itemize}
\item[-] $\om,\om_1,\bar{\om}_1$ are as in {\rm (\ref{theta})},
\item[-]
  $\om\dz\om_1\dz\bar{\om}_1\dz\Om\dz\bar{\Om}\neq
  0$ at each point of $P$,
\item[-] the forms and functions $K_1$ (real), $K_2$ (complex)
are uniquely determined by the requirement that on $P$
they satisfy equations {\rm (\ref{syste})}.
\end{itemize}
In particular the identical vanishing, or not, of either $k_1$ or $k_2$ are
invariant conditions. Also the sign of $k_1$ is an invariant, if $k_1\neq 0$.
\end{theorem}
\subsection{Description in terms of the Cartan connection}\label{cconse}
The above theorem, stated in modern language, means the
following. The manifold $P$ is a Cartan bundle $H_2\to P\to M$, with $H_2$
a 2-dimensional abelian subgroup of a certain 5-dimensional Lie group
$G_5$. The group $G_5$ is a subgroup of ${\bf SU}(2,1)$; i.e.,
the 8-dimensional Lie group which
preserves the $(2,1)$-signature hermitian form
$$
h(Z,Z)=
\bma
Z^1,&Z^2,&Z^3
\ema
\hat{h}
\bma
\bar{Z}^1\\
\bar{Z}^2\\
\bar{Z}^3
\ema, ~~~~~~
\hat{h}=
\bma
0&0&2i\\
0&1&0\\
-2i&0&0
\ema.
$$
The forms
$\om,\om_1,\bar{\om}_1,\Om,\bar{\Om}$
in the theorem can be collected into a matrix of 1-forms
$$
\tilde{\omega}=
\begin{pmatrix}
\frac{1}{3}(2\Om+\bar{\Om})&0&0\\
&&\\
\om_1&\frac{1}{3}(\bar{\Om}-\Om)&0\\
&&\\
2\om&2i\bar{\om}_1&-\frac{1}{3}(2\bar{\Om}+\Om),
\end{pmatrix}
$$
satisfying
$$
\tilde{\om}\hat{h}+\hat{h}\tilde{\om}^\dagger=0.
$$
The Lie algebra ${\frak g}_5$ of the group $G_5$ is then
$${\frak g}_5=\{
\begin{pmatrix}
\frac{1}{3}(2z_2+\bar{z}_2)&0&0\\
&&\\
z_1&\frac{1}{3}(\bar{z}_2-z_2)&0\\
&&\\
2x&2i\bar{z}_1&-\frac{1}{3}(2\bar{z}_2+z_2)
\end{pmatrix},~x\in\bbR,~z_1,z_2\in\bbC\},$$ and as such is a real
5-dimensional Lie algebra parametrized by the parameters
$x,{\rm Re}(z_1),{\rm Im}(z_1),{\rm Re}(z_2),{\rm Im}(z_2)$. It is naturally contained in
  ${\frak su}(2,1)$. The subgroup $H_2$ corresponds to the subalgebra
  ${\frak h}_2\subset{\frak g}_5$ given by $x=0,z_1=0$. Now, $\tilde{\om}$ can be interpreted as
a Cartan connection on $P$ \cite{Kobayashi} having values in the Lie
algebra ${\frak g}_5\subset{\frak su}(2,1)$. It follows from
equations (\ref{syste}) that the curvature $R$ of
this connection is
$$
R=\der \tilde{\om}+\tilde{\om}\dz\tilde{\om}=
\begin{pmatrix}
R_1&0&0\\
0&R_2&0\\
0&0&-R_1-R_2
\end{pmatrix},
$$
where
\beq
R_1&=&-\tfrac{2}{3}K_2\om\dz\om_1-\tfrac{1}{3}\overline{K}_2\om\dz\bar{\om}_1+\tfrac{1}{3}K_1\om_1\dz\bar{\om}_1\nonumber\\
R_2&=&\tfrac{1}{3}K_2\om\dz\om_1-\tfrac{1}{3}\overline{K}_2\om\dz\bar{\om}_1-\tfrac{2}{3}K_1\om_1\dz\bar{\om}_1\nonumber
\eeq
It yields all the invariant information about the corresponding
structure $(M,[\lambda,\mu])$, very much in the same way as the Riemann
curvature yields all the information about a Riemannian structure.\\

\subsection{Conformal Lorentzian metrics}\label{cconse1}
Using the matrix elements $\tilde{\om}^i_{~j}$ of the Cartan connection $\tilde{\om}$
it is convenient to consider the bilinear form
$$
G=-i\tilde{\om}^3_{~j}\tilde{\om}^j_{~1}.
$$
This form, when written explicitly in terms of
$\om,\om_1,\bar{\om}_1,\Om,\bar{\Om}$,
is given by
$$
G=2\om_1\bar{\om}_1+\frac{2}{3i}\om(\Om-\bar{\Om}).
$$
Introducing the basis of vector fields
$X,X_1,\bar{X}_1,Y,\bar{Y}$, the respective duals of $
\om,\om_1,\bar{\om}_1$, $\Om,\bar{\Om}$,
one sees that $G$ is a form of signature $(+++-0)$ with the
degenerate direction tangent to the vector field $Y+\bar{Y}=\rho\partial_\rho$.
We may think of the Cartan bundle $P$ as being foliated by 1-dimensional
leaves tangent to this vector field. Now
equations (\ref{syste}) guarantee that the Lie derivative
$$
{\mathcal L}_{(Y+\bar{Y})}~G=2~G,
$$
so that the bilinear form $G$ is preserved up to a scale when Lie
transported along the leaves of the foliation. Therefore the
4-dimensional leaf space $N=P/\hspace{-0.15cm}\sim$ of the
foliation is naturally
equipped with a conformal class of Lorentzian metrics $[g]$, the
class to which the bilinear form $G$ naturally descends. The
Lorentzian metrics
\be
g=2\om_1\bar{\om}_1+\frac{2}{3i}\om(\Om-\bar{\Om})\label{fef}
\ee
on $N$ are the analogs of the Fefferman
metrics \cite{Fef} known in CR manifold theory.

We note that $N$ is a circle bundle above $M$ with the fiber
coordinate $\phi$.

Interestingly metrics (\ref{fef}) belong to a larger conformal family,
which is also well defined on $N$. It turns out that if we start
with a bilinear form
$$
G_t=2\om_1\bar{\om}_1~+~2ti~\om(\bar{\Om}-\Om)
$$
where $t$ is {\it any function} on $P$ constant along the $Y+\bar{Y}$
direction, then it also well projects to a conformal Lorentzian class
$[g_t]$ on $N$ with representatives
\be
g_t=2\om_1\bar{\om}_1~+~2ti~\om(\bar{\Om}-\Om)\label{feog}
\ee
parametrized by $t$.
 To
see this it is enough to look at the explicit expressions for the
forms $(\om_1,\bar{\om}_1,\om,\Om,\bar{\Om})$ in (\ref{n77}) and to
note that $G_t$ is of the form $G_t=\rho^2(...)$, where the dotted
terms do not depend on the coordinate $\rho$ which is aligned with
$Y+\bar{Y}$ on $P$. 

Although $t$ may be an arbitrary function on $N$, in what follows 
we will only be interested in the case when $t$ is a {\it constant} 
parameter.

We return to metrics $g_t$ in Section \ref{bbaa}, where we
discuss their conformal curvature $F_t$ and provide some example of
the Lorentzian metrics satisfying the so called Bach condition.

\subsection{Basic examples}
\noindent
\begin{example}\label{hypc}
Note that the assumption that $K_1$ and $K_2$ are {\it
  constant} on $P$
is compatible with (\ref{syste}) iff $K_1=K_2=0$. In such case the
curvature $R$ of the Cartan connection $\tilde{\om}$
%associated
%with the CR-structure $(N,[(\la,\mu)])$
vanishes, and it follows that there is only one, modulo local
equivalence, $[\lambda,\mu]$ structure with this property. It
coincides with the CR structure of the {\it Heisenberg group}
$$M=\{~(z,w)\in\bbC^2: {\rm Im}(w)=|z|^2~\}$$ with the preferred
splitting $\vv$ generated by the vector field $\bv=\partial_u$,
$u={\rm Re}(w)$. We call this the {\it standard splitting} on the Heisenberg group.
The resulting oriented congruence has
the
maximal possible group of symmetries isomorphic to the group $G_5$.\\
\end{example}

\begin{example}\label{rigid1}
We recall that a 3-dimensional CR manifold $M$ embedded in $\bbC^2$ via
$$M=\{~(z,w=u+iv)\in\bbC^2~:~v=\tfrac12 H(z)~\},$$
where $H$ is a real-valued fuction of  the variable $z\in\bbC$, is called {\it rigid}. It
can be given a structure of an oriented congruence by choosing the splitting to be spanned
by the vector field $\partial_u$. As in the above special case of the Heisenberg group
we call this preferred splitting on $M$ the {\it standard splitting} on a rigid CR structure. Intrinsically
this CR-manifold with the preferred splitting may be described in
terms of the forms $\lambda$ and $\mu$ given by
\be
\lambda=\der u+\tfrac{i}{2}(H_{\bar{z}}\der\bar{z}-H_z\der z),\quad\quad\mu=\der
z.\label{rigid}
\ee
Via (\ref{ec}), these forms define a structure
$(M,[\lambda,\mu])$ of an oriented congruence on $M$. In the following we assume that
$$H_{z\bar{z}}\neq 0$$
at every point of $M$. It means that $M$ is strictly pseudoconvex.
\end{example}

\begin{definition}
A structure $(M,[\lambda,\mu])$ of an oriented congruence with
vanishing shear and nonvanishing twist on a manifold $M$ is called
(locally) \emph{flat} iff (locally) it has vanishing curvature $R$
for its Cartan connection $\tilde{\omega}$. The necessary and
sufficient conditions for that are $K_1\equiv 0$ and $K_2\equiv 0$.
\end{definition}

A short calculation leads to the following proposition.

\begin{proposition}
Let $(M,[\lambda,\mu])$ be a structure of an oriented congruence
associated with the rigid CR-manifold $M$ via the forms $\lambda$ and $\mu$
of (\ref{rigid}). Then for any real-valued function $H=H(z)$ such that
$H_{z\bar{z}}\neq 0$ this structure
has vanishing shear and non-vanishing twist. Its
relative invariant $K_2$ is identically vanishing, $K_2\equiv 0$; the
relative invariant $K_1$ is given by
$K_1=\tfrac{1}{\rho^2}[\log(H_{z\bar{z}})]_{z\bar{z}}$.
When it vanishes the structure is flat.
\end{proposition}

\begin{example}\label{epsi}
We remark that the Heisenberg group CR structure may have various
splittings that endow $M$ with nonequivalent structures of an
oriented congruence. To see this we perturb the standard splitting
on the Heisenberg group given by the vector field $\partial_u$. This
is accomplished by choosing a 2-parameter family of CR-functions on
$M$ given by
\be
\zeta_{\epsilon_1\epsilon_2}=\epsilon_1
z+\epsilon_2(u+i|z|^2),\label{zee}
\ee
and defining the structure of
an oriented congruence on $M$ via (\ref{ec}) with the forms
$$\lambda=\der u+i (z\der\bar{z}-\bar{z}\der z),\quad\quad
\mu_{\epsilon_1\epsilon_2}=\der\zeta_{\epsilon_1\epsilon_2}.$$ Note
that since $\lambda$ is a section of the characteristic bundle $H^0$
of the Heisenberg group CR-structure, and
$\mu_{\epsilon_1\epsilon_2}$ is the differential of a CR-function,
the structure $(M,[\lambda,\mu_{\epsilon_1\epsilon_2}])$ is {\it
  twisting} and
{\it without shear} for all values of the real parameters
$\epsilon_1$ and $\epsilon_2$. The real vector field $\bv$ which
gives the splitting on $M$ is given by
$$\bv=\partial_u+\frac{\epsilon_2}{\epsilon_1}~[~\frac{i\epsilon_1+2\epsilon_2
    z}{-i\epsilon_1+\epsilon_2(\bar{z}-z)}\partial_z+\frac{-i\epsilon_1+2\epsilon_2 \bar{z}}{i\epsilon_1+\epsilon_2(z-\bar{z})}\partial_{\bar{z}}~],$$
if $\epsilon_1\neq 0$, and
$$\bv=i(z\partial_z-\bar{z}\partial_{\bar{z}})$$
otherwise.
A short calculation shows that the relative invariants $K_{1\epsilon_1\epsilon_2}$ and $K_{2\epsilon_1\epsilon_2}$ for this
2-parameter family of structures are
$$K_{1\epsilon_1\epsilon_2}=\frac{8\epsilon_2^2}{\rho^2|2\epsilon_2z+i\epsilon_1|^4},\quad\quad
K_{2\epsilon_1\epsilon_2}\equiv 0.$$ This proves that the structures
with $\epsilon_2=0$ and $\epsilon_2\neq 0$ are {\it not} locally
equivalent. To analyse if the structures with $\epsilon_2\neq 0$ are
equivalent or not we need to apply further the Cartan equivalence
method. We will perform it in a more general setting than this
example.
\subsection{The case $K_1\neq 0$, $K_2\equiv 0$}\label{K1nK2} Let
$(M,[\lambda,\mu])$ be an arbitrary structure of an oriented
congruence which has nonvanishing twist, vanishing shear, and in
addition has the relative invariants $K_1$ and $K_2$ such that
$$K_1\neq 0\quad{\rm and}\quad K_2\equiv 0.$$ Given such a structure, using the
system (\ref{syste}) and the assumption $K_2\equiv 0$, we observe that the
corresponding structural form $\Omega$ has closed real part,
\be \der(\Omega+\bar{\Omega})\equiv 0. \label{bbu}\ee The
assumption that $K_1\neq 0$ enables us to make a further reduction of
the Cartan system (\ref{syste}) defining the invariants. Indeed
since $K_1=\frac{1}{\rho^2}k_1\neq 0$,
we may restrict ourselves to a (possibly double-sheeted)
hypersurface $N_0$ in $P$ on which
$$K_1=\pm1,$$
where the sign is determined by the sign of the function $k_1$.
Recall that this sign is an invariant of the structure.

Locally $N_0$ is a  circle bundle over
$M$ defined by the condition
$$\rho^2=|k_1|.$$
Now the system (\ref{syste}) when pullbacked to $N_0$ locally
reduces to \beq
\der\om&=&i\om_1\dz\bar{\om}_1+2\der A\dz\om\nonumber\\
\der\om_1&=&\der A\dz\om_1+i\Sigma\dz\omega_1\nonumber\\
\der\bar{\om}_1&=&\der A\dz\bar{\om}_1-i\Sigma\dz\bar{\omega}_1\label{systen}\\
\der\Sigma&=&\mp i\om_1\dz\bar{\om}_1.\nonumber \eeq
Here the real
1-form $\Sigma$ is the pullback of the form
$\tfrac{1}{2i}(\Omega-\bar{\Omega})$ from $P$ to $N_0$. According to our choice of $\Sigma$, the {\it minus} sign in
(\ref{systen}) corresponds to $K_1=+1$. The
differential $\der A$ of the real function $A$ on $N_0$ is determined by the condition that $2\der A$
is locally equal to the pullback of the $\Omega+\bar{\Omega}$ from $P$
to $N_0$. Note that this pullback must be closed due to (\ref{bbu}).
Looking at the explicit expression for $\Omega+\bar{\Omega}$ in
(\ref{omp}), (\ref{t}) and the integrability conditions for (\ref{systen}) we find that locally we have
\be
2\der A=A_1\omega_1+\bar{A}_1\bar{\omega}_1,\label{da}
\ee
with
\be
A_1=\frac{{\rm
e}^{-i\phi}}{\sqrt{|k_1|}}((\log\frac{|k_1|}{a})_{\mu}+b).\label{a1}
\ee
The
function $A_1$ gives a new relative invariant
for the structures $(M,[\lambda,\mu])$ with $K_1\neq 0$ and
$K_2\equiv 0$. It follows from the construction that two such
structures $(M,[\lambda,\mu])$ and $(M',[\lambda',\mu'])$ are
(locally) equivalent if there exists a (local) diffeomorphism of
the corresponding manifolds $N_0$ and ${N_0}'$ which transforms the
corresponding forms $(\omega,\omega_1,\bar{\omega}_1,\Sigma)$ to
$(\omega',\omega_1',\bar{\omega}_1',\Sigma')$. This in turn implies
that the relative invariant $A_1$ must be transformed to
$A_1'$.

\begin{remark}\label{reflatc}
We note that among
all the structures with $K_1\neq 0$ and $K_2\equiv 0$ the simplest have $A_1\equiv 0$.
Modulo local equivalence there are only two such structures, corresponding to the $\mp$ sign in (\ref{systen}) with
$A_1\equiv0$. These are the {\it `flat cases' for the subtree} in which $K_1\neq 0$
and $K_2\equiv 0$.
\end{remark}

The function $A$ defining the relative invariant $A_1$ is defined only up to the
addition of a constant, $A\to A+t$.
Given a family of functions $A(t)=A+t$ we consider the family of bilinear
forms
$G_A(t)$ on $N_0$ defined by
$$G_{A(t)}={\rm e}^{-2(A+t)}\om_1\om_2.$$
The forms $G_{A(t)}$ are clearly degenerate on $N_0$. Denoting by
$(X,X_1,\bar{X}_1,Y)$ the dual vector fields to the basis of 1-forms $(\om,\om_1,\bar{\om}_1,\Sigma)$ on
$N_0$, we see that the signature of $G_{A(t)}$ is $(+,+,0,0)$ with the degenerate
directions aligned with the real vector fields $X$ and $Y$. Next we observe that the
system (\ref{systen})
implies that $[X,Y]\equiv 0$, hence the distribution spanned by $X$
and $Y$ is integrable. Thus $N_0$ is foliated by real
2-dimensional leaves. Locally the leaf space $S$ of this foliation is a
2-dimensional real manifold, which is a Riemann surface, since
the pullback to $S$ of the 1-form $\om_1$ gives a basis for the $(1,0)$ forms. Now
the
formula (\ref{da}) implies that $X(A)=Y(A)\equiv 0$. Using this and
the system (\ref{systen}), a calculation shows that
$${\mathcal L}_X G_{A(t)}\equiv 0,\quad\quad {\mathcal L}_Y G_{A(t)}\equiv
0.$$
This means that the bilinear forms $G_{A(t)}$ descend to Riemannian
homothetic metrics $g_{A(t)}$ on the Riemann surface $S$. We have the following
theorem.
\begin{theorem}\label{ries}
The Riemann surface $S$ naturally associated with the structure of
an
oriented congruence having $K_1\neq 0$, $K_2\equiv 0$ possesses
Riemannian homothetic metrics $g_{A(t)}$ whose Gaussian curvatures $\kappa(t)$ are related
to the relative invariant $A_1$ via:
$$\kappa(t)=\mp {\rm e}^{2(A+t)},\quad\quad
{\rm i.e.}\quad\quad 2\der A=\der\log\kappa.$$
\end{theorem}
\noindent {\bf Example \ref{epsi} (continued)} Calculating $A_1$ for the
structures $(M,[\lambda,\mu_{\epsilon_1\epsilon_2}])$ of Example
\ref{epsi}, assuming that $\epsilon _2\neq 0$, we easily find that
for all $\epsilon_1$, and $\epsilon_2\neq 0$, we have $A_1\equiv 0$.
Thus for all nonzero values of $\epsilon_2$, and all
values of $\epsilon_1$, the structures {\it are} locally equivalent.
Hence the apparent 2-parameter family of the structures
$(M,[\lambda,\mu_{\epsilon_1\epsilon_2}])$ includes {\it only two
nonequivalent} cases; isomorphic to those with
$(\epsilon_1,\epsilon_2)=(1,0)$, and e.g. to those with
$(\epsilon_1,\epsilon_2)=(0,1)$. The first case is the flat case $K_1\equiv 0$, $K_2\equiv 0$,
corresponding to the Heisenberg group with the standard splitting.
The second case is considerably different, being one of the `flat cases'
for the subtree $K_1\neq 0$ and $K_2\equiv 0$, corresponding to $A_1\equiv 0$
and the {\it minus} sign in (\ref{systen}).
In particular the $(0,1)$ case has only
a 4-dimensional symmetry group, as opposed to the 5-dimensional
symmetry group of the $(1,0)$ case.

We would like to point out that if we were to choose a more
complicated CR function than the $\zeta_{\epsilon_1\epsilon_2}$ of
(\ref{zee}), for example
$$
\zeta=\epsilon_1 z+\epsilon_2(u+i|z|^2)^m,
$$
with $m\neq 0$ and $m\neq 1$, we would produce an oriented
congruence $(M,[\der u+i (z\der\bar{z}-\bar{z}\der z),\der\zeta])$, still twisting and without
shear, again based on the Heisenberg group, but not equivalent to
either of the two structures above. The reason for this is that the
condition $m\neq 0, m\neq 1$ makes $(M,[\der u+i (z\der\bar{z}-\bar{z}\der z),\der\zeta])$ have
the relative invariant $K_2$ nonvanishing.
\end{example}

We now give a local representation for an arbitrary structure
$(M,[\lambda,\mu])$ with vanishing shear, nonvanishing twist, and
with $K_1\neq 0$, $K_2\equiv 0$. This can be done by integration of the
system (\ref{systen}). Interestingly this integration can be performed
explicitly, leading to the following theorem.

\begin{theorem}\label{coct}
If $(M,[\lambda,\mu])$ is a structure of an oriented congruence
with vanishing shear, nonvanishing twist, and
with the relative invariants $K_1\neq 0$, $K_2\equiv 0$ then there
exists a coordinate system $(u,z,\bar{z})$ on $M$ such that the forms
$\lambda$ and $\mu$ representing the structure can be chosen to be
$$\lambda=\der u+\tfrac{i}{2}(H_{\bar{z}}\der \bar{z}-H_z\der z),\quad\quad \mu=\der z,$$
where the real functions $A=A(z)$ and $H=H(z)$ satisfy
the system of PDEs
\begin{eqnarray}
h_{z\bar{z}}&=&\mp {\rm e}^{2A}{\rm e}^{-h}\label{se1}\\
H_{z\bar{z}}&=&{\rm e}^{-h}\label{se2}
\end{eqnarray}
with a real function $h=h(z)$.
The structure corresponding to such $\lambda$ and $\mu$ satisfies the
system
\begin{eqnarray*}
\der\om&=&i\om_1\dz\bar{\om}_1+2\der A\dz\om\\
\der\om_1&=&\der A\dz\om_1+i\Sigma\dz\omega_1\\
\der\bar{\om}_1&=&\der A\dz\bar{\om}_1-i\Sigma\dz\bar{\omega}_1\\
\der\Sigma&=&\mp i\om_1\dz\bar{\om}_1
\end{eqnarray*}
with forms
$$\omega={\rm e}^{2A}\lambda,\quad\quad \omega_1={\rm e}^{A}{\rm e}^{-h/2}{\rm e}^{i
  \phi}\mu,
\quad\quad \bar{\omega}_1={\rm e}^{A}{\rm e}^{-h/2}{\rm e}^{-i \phi}\bar{\mu}$$
$$\Sigma=\der\phi+\tfrac{i}{2}(h_{\bar{z}}\der\bar{z}-h_z\der z).$$
The relative invariant $A_1$ of this structure is given by $$A_1=2{\rm e}^{-A}{\rm e}^{h/2}{\rm
  e}^{-i \phi}A_z.$$
\end{theorem}

Note that the system of PDEs (\ref{se1})-(\ref{se2}) is underdetermined.
To see that it always has solutions, choose a real function $H=H(z)$ on the complex plane.
Define the real function $h=h(z)$ via equation (\ref{se2}), insert it into equation (\ref{se1}) and solve this
{\it real} PDE for a real function $A=A(z)$. Since the function $H$ can be chosen arbitrarily, returning to
Example \ref{rigid1}, we see that this theorem characterizes the oriented congruences
which are locally equivalent to those defined on rigid CR manifolds
with the standard splitting.

\begin{corollary}\label{kolo}
Every structure $(M,[\lambda,\mu])$ of an oriented congruence
with vanishing shear, nonvanishing twist, and
with the relative invariants $K_1\neq 0$, $K_2\equiv 0$ admits
\emph{one} symmetry.
\end{corollary}
\begin{proof}
To proof this it is enough to check that in the local representation (\ref{se1})-(\ref{se2}) the symmetry
is generated by $X_0=\partial_u$.
\end{proof}

Starting with a structure $(M,[\lambda,\mu])$ having $K_1\neq 0$ and $K_2\equiv
0$ we constructed its associated circle bundle $\bbS^1\to N_0\to M$
equipped with the invariant forms
$(\omega,\om_1,\bar{\om}_1,\Sigma)$. Using the dual basis
$(X,X_1,\bar{X}_1,Y)$ and the system (\ref{systen}) we see that the
symmetry $X_0$ lifts to a vector field $\tilde{X}={\rm e}^{2A}X$
with the property that
$${\mathcal L}_{\tilde{X}}\Sigma=0,\quad\quad {\mathcal
L}_{\tilde{X}}\om_1=2\tilde{X}(A)\om_1.$$
We now introduce a quotient 3-dimensional manifold $M_\Sigma$ whose points are the integral curves
of $\tilde{X}$.
Then the forms $\Sigma$ and $\omega_1$ descend from $N_0$ to a class of forms
$[\Sigma,\om_1]$ on $M_\Sigma$ given up to the transformations
$\Sigma\to\Sigma$, $\om_1\to h\om_1$. Thus they can be used to
define a structure of an oriented
congruence $(M_\Sigma,[\Sigma,\omega_1])$. This structure
naturally associated with $(M,[\lambda,\mu])$ may be locally represented by
the coordinates $(\phi,z,\bar{z})$ of Theorem \ref{coct} with the
representatives $\Sigma$ and $\omega_1$ given by
$$\Sigma=\der\phi+\tfrac{i}{2}(h_{\bar{z}}\der\bar{z}-h_z\der
z),\quad\quad\om_1=\der z.$$
Here the real function $h=h(z)$ is related to the original structure
$(M,[\lambda,\mu])$ via equations (\ref{se1})-(\ref{se2}). In
particular $(M_\Sigma,[\Sigma,\om_1])$ is again based on a rigid CR
structure with the standard splitting.\\

Now we use Theorem \ref{coct} to describe all the structures with $K_1\neq
0$ and $K_2\equiv 0$ which
have a 4-dimensional transitive symmetry group. It turns out that they must be
equivalent to those with $\der A\equiv 0$. This is because the
existence of a
4-dimensional transitive symmetry group implies that $A_1$ must be a
constant. But since $A$ and $h$ depend only on $z$ and $\bar{z}$, and
$A_1$ has nontrivial ${\rm e}^{i\phi}$ dependence, it is possible
iff $A_z\equiv 0$; hence $A_1\equiv 0$. Thus according to Remark \ref{reflatc}
there are only two such structures.
One of them, the one with the {\it upper} sign in (\ref{systen}),
is equivalent to the structure $(\epsilon_1,\epsilon_2)=(0,1)$ of Example \ref{epsi}. To find the second one
we use Theorem \ref{coct} and integrate equations
(\ref{se1})-(\ref{se2}) for $A=0$. Modulo equivalence we get two solutions
$$h_\mp=2\log (1\mp\tfrac12 z\bar{z}),\quad\quad H_\mp=\mp 2\log
(1\mp\tfrac12 z\bar{z}),\quad\quad A=0$$
which lead to the two nonequivalent `flat models' with $K_1=\pm 1$, $A_1\equiv 0$. These are generated by the forms
\be
\lambda_\mp=\der u+\tfrac{i}{2}\frac{z\der\bar{z}-\bar{z}\der z}{1\mp\tfrac12 z\bar{z}},\quad\quad \mu=\der z.\label{taub}
\ee
Obviously the structure corresponding to the upper sign is isomorphic to the structure
$(\epsilon_1,\epsilon_2)=(0,1)$ of Example \ref{epsi}. Interestingly, in either of the two nonequivalent cases
the forms $(\lambda,\mu)$ can be used to intrinsically define a {\it flat}
CR structure (in the sense of Cartan's paper \cite{cartan}) on $M$ parametrized by
$(u,z,\bar{z})$. Another feature of these two nonequivalent
structures is that their Riemann surface $S_{\mp}$ described by
Theorem \ref{ries} is equipped with metrics $g_{A(t)}$ which may be
represented by
$$g_\mp=\frac{2\der z\der\bar{z}}{(1\mp\tfrac12 z\bar{z})^2}.$$
Thus these Riemann surfaces are either locally homothetic to the Poincar\'e disc
(in the upper sign case) or to the 2-dimensional
sphere $\bbS^2$ (in the lower sign case). This leads to the following definition.
\begin{definition}
The two structures of an oriented congruence $(M,[\lambda_\mp,\mu])$
generated by the forms $\lambda_\mp,\mu$ of (\ref{taub})
are called the \emph{Poincar\'e disc structure} (in the upper sign case) and the
\emph{spherical structure} (in the lower sign case).
\end{definition}
\noindent
We further note that the
natural structures $(M_\Sigma,[\Sigma_\mp,\om_1])$ associated with the structures (\ref{taub})
are locally isomorphic to the original structures $(M,[\lambda_\mp,\mu])$. Finally we note that the forms
$\lambda_+, \mu$ are
identical with the forms which appear in the celebrated vacuum Taub-NUT solution of
the Lorentzian Einstein field equations 
(see formulae (\ref{kke})-(\ref{kke1}) 
with $K-1=m=a=0$ and with the coordinate $z$ replaced by $2/z$). We summarize the considerations of this paragraph in the following Theorem.
\begin{theorem}\label{flt}
All structures $(M,[\lambda,\mu])$ of an oriented congruence with vanishing shear, nonvanishing twist, having the
relative invariants $K_1\neq 0$, $K_2\equiv 0$ and possessing a 4-dimensional transitive
symmetry group are locally isomorphic to either the Poincar\'e disc structure $(M,[\lambda_-,\mu])$ or
the spherical structure $(M,[\lambda_+,\mu])$,
i.e. they are isomorphic to one of the 'flat models' for the $K_1\neq 0$ and $K_2\equiv 0$
case.
\end{theorem}

We now pass to the determination of all local invariants for the structures with $A_1\neq
0$. Let $(M,[\lambda,\mu])$ be such a structure with the
corresponding circle bundle $N_0$ and the system of invariants
(\ref{systen}). Looking at the explicit form (\ref{a1}) of the relative invariant $A_1$,
we see that we may always choose a section of the bundle $N_0$ such that $A_1$ is {\it
real} and {\it positive}. Locally this corresponds to the choice of $\phi$ as a
function on the manifold $M$ such that
\be
\frac{{\rm
e}^{-i\phi}}{\sqrt{|k_1|}}((\log\frac{|k_1|}{a})_{\mu}+b)=\frac{{\rm
e}^{i\phi}}{\sqrt{|k_1|}}((\log\frac{|k_1|}{a})_{\bar{\mu}}+\bar{b})>0.
\label{a1real}
\ee
If $\phi$ satisfies (\ref{a1real}) then
$$A_1>0,$$
and all the structural objects defined by the system (\ref{systen})
may be uniquely pullbacked to $M$. As the result of this pullback
the real 1-form $\Sigma$ becomes dependent on the pullbacked forms
$(\om,\om_1,\bar{\om}_1)$. Since these three 1-forms constitute a
coframe on $M$ we may write
$\Sigma=B_0\om+B_1\om_1+\bar{B}_1\bar{\om}_1$
where $B_0$ (real) and $B_1$ (complex) are functions on $M$. Now using the fact that these structures admit a
symmetry (Corollary \ref{kolo}), we get $B_0\equiv 0$. Hence
$$\Sigma=B_1\om_1+\bar{B}_1\bar{\om}_1.$$
With
this notation the pullbacked system (\ref{systen}) becomes
\begin{eqnarray}
\der\om&=&i\om_1\dz\bar{\om}_1+2A_1(\om_1+\bar{\om}_1)\dz\om\nonumber\\
\der\om_1&=& -(A_1+i\bar{B}_1)\om_1\dz\bar{\om}_1\label{systek}\\
\der\bar{\om}_1&=&(A_1-iB_1)\om_1\dz\bar{\om}_1,\nonumber
\end{eqnarray}
with the fourth equation given by
\be
\der (B_1\om_1+\bar{B}_1\bar{\om}_1)=\mp
i\om_1\dz\bar{\om}_1.\label{systekk}
\ee
\begin{remark}
Note that since on $N_0$ the complex function
$A_1$ was constrained by $\der
(A_1\om_1+\bar{A}_1\bar{\om}_1)=0$, because of (\ref{da}),  
the equations (\ref{systek})-(\ref{systekk}) should be supplemented by
the equation
$\der[A_1(\om_1+\bar{\om}_1)]=0$ for $A_1>0$. This however is equivalent to 
$$\der
  A_1\dz(\om_1+\bar{\om}_1)=[iA_1(B_1+\bar{B}_1)]\om_1\dz\bar{\om}_1,$$ 
and turns out to follow from the integrability conditions for (\ref{systek})-(\ref{systekk}).
\end{remark}

Writing these integrability conditions explicitly we have:
\begin{eqnarray}
\der
A_1&=&[a_{11}+\tfrac{i}{2}A_1(B_1+\bar{B}_1)]\om_1+[a_{11}-\tfrac{i}{2}A_1(B_1+\bar{B}_1)]\bar{\om}_1\nonumber\\
\der B_1&=&B_{11}\om_1+[b_{12}+\tfrac12 A_1(\bar{B}_1-B_1)+i(\pm\tfrac{1}{2}-|B_1|^2)]\bar{\om}_1\label{itg}\\
\der\bar{B}_1&=&[b_{12}-\tfrac12
A_1(\bar{B}_1-B_1)-i(\pm\tfrac{1}{2}-|B_1|^2)]\om_1+\bar{B}_{11}\bar{\om}_1,\nonumber
\end{eqnarray}
where the real functions $a_{11}$, $b_{12}$ are the scalar invariants of
the next higher order than $A_1$ and $B_1$.
\begin{theorem}
The functions $A_1>0$ and $B_1$ (complex) constitute the full system of basic scalar
invariants for the structures $(M,[\lambda,\mu])$ with $K_1\neq 0$, $K_2\equiv 0$ and $A_1\neq 0$.
It follows from the construction that two such
structures $(M,[\lambda,\mu])$ and $(M',[\lambda',\mu'])$ are
(locally) equivalent iff there exists a (local) diffeomorphism between 
$M$ and $M'$ which transforms the
corresponding forms $(\omega,\omega_1,\bar{\omega}_1)$ to
$(\omega',\omega_1',\bar{\omega}_1')$. This in particular implies
that the invariants $A_1$ and $B_1$ must be transformed to
$A_1'$ and $B_1'$.
\end{theorem}
The system (\ref{systek})-(\ref{itg}) and the above theorem can be
used to find all structures with $K_1\neq 0$ and $K_2\equiv 0$
having a
strictly 3-dimensional transitive symmetry group. These are the
structures described by the system (\ref{systek})-(\ref{itg}) with
{\it constant} basic invariants $A_1>0$, $B_1$. It follows that it
is possible only if $B_1=i\tau$, $A_1=\frac{\pm 1-2\tau^2}{2\tau}>
0$ and $\tau\neq 0$ is a real parameter. This leads to {\it only} two quite different cases, which are described
by Propositions \ref{pr3s1} and \ref{pr3s2}.
\begin{proposition}\label{pr3s1}
\emph{(i)} All locally nonequivalent structures $(M,[\lambda,\mu])$ of oriented
congruences having vanishing shear, nonvanishing twist, $K_1\neq 0$, $K_2\equiv 0$, and
\emph{possessing a strictly 3-dimensional transitive group} $G_h$ \emph{of symmetries of Bianchi type} $VI_h$, $h\leq 0$, may
be locally represented by
$$\lambda = y^b \der u - y^{-1}\der x,\quad\quad \mu = y^{-1}(\der x + i \der
y).$$
Here $(u,z,\bar{z})$ with $z=x+i y$ are coordinates on $M$ and
$$b=-2(1\mp 2\tau^2).$$
The real parameter $\tau$ is related to the invariants $B_1$ and $A_1$ via
$$B_1=i\tau,\quad\quad A_1=-\frac{\mp 1+2\tau^2}{2\tau}>0
,$$
and as such enumerates nonequivalent structures.\\
\emph{(ii)} Regardless of the values of $\tau$ the structures corresponding to the upper and lower signs in
the expressions above are nonequivalent.
In the case of the lower signs the real parameter $\tau< 0$. In the case of the upper
signs $\tau<-\frac{1}{\sqrt{2}}$ or $0<\tau<\frac12$ or $\frac12<\tau<\frac{1}{\sqrt{2}}$.\\
\emph{(iii)} The structures are locally \emph{CR equivalent} to the Heisenberg group
CR structure only in the case of the upper signs with
$\tau=\frac{\sqrt{3}}{2\sqrt{2}}$.\\
\emph{(iv)} The symmetry group is of Bianchi type $VI_h$, with
the parameter $h\leq 0$ related to $\tau$ via
$$h=-\Big(\frac{3\mp 4\tau^2}{1\mp4 \tau^2}\Big)^2.$$
In the lower sign case the possible values of $h$ are $-9< h<-1$, and for each value of $h$ we
always have \emph{one} structure with the symmetry group
$G_h$. In the upper sign case $h$ may assume all values $h\leq 0$, $h\neq -1$.
In this case, we always have
\begin{itemize}
\item[-]\emph{two} nonequivalent structures with symmetry group $G_h$ with
$h<-9$;
\item[-]\emph{one} structure with symmetry group $G_h$ with
$-9\leq h<-1$; if the parameter $\tau$ is
$\tau=\frac{\sqrt{3}}{2\sqrt{2}}$ then $h=-9$ and the structure is based
on the Heisenberg group with a particular \emph{nonstandard}
splitting;
\item[-]\emph{two} nonequivalent structures with symmetry group $G_h$ with
$-1<h<0$;
\item[-]\emph{one} structure with symmetry group of Bianchi type $VI_0$.
\end{itemize}
\end{proposition}
%\begin{remark}
%We remark that in the above Proposition the
%limiting value $\tau\to 0$ corresponds to the flat model, i.e the
%Heisenberg group with the standard splitting. In the lower sign
%case the limiting values $\tau\to\pm\frac{1}{\sqrt{2}}$ correspond
%to the Poincar\'e disc structure; the limiting value $\tau\to\frac12$
%is, in a certain sense, covered by Proposition \ref{pr3s2}.
%\end{remark}
\begin{proposition}\label{pr3s2}
Modulo local equivalence there exists only \emph{one} structure $(M,[\lambda,\mu])$ of an oriented
congruence having vanishing shear, nonvanishing twist, $K_1\neq 0$, $K_2\equiv 0$, and
\emph{possessing a strictly 3-dimensional transitive group of symmetries of Bianchi type} $IV$. Locally it
may be represented by the forms
$$\lambda=y^{-1}(\der u+\log y\der x),\quad\quad\mu=y^{-1}(\der
x+i\der y).$$
Here $(u,z,\bar{z})$ with $z=x+i y$ are coordinates on $M$. The structure has
the basic local invariants  $A_1=\frac12$ and $B_1=\frac{i}{2}$.
\end{proposition}

Summarizing we have the following theorem.
\begin{theorem}\label{33s}
All locally nonequivalent structures $(M,[\lambda,\mu])$ of oriented
congruences having vanishing shear, nonvanishing twist, $K_1\neq 0$, $K_2\equiv 0$, and
\emph{possessing a strictly 3-dimensional transitive group of
symmetries} are locally equivalent to one of the structures defined in
Propositions \ref{pr3s1} and \ref{pr3s2}.
\end{theorem}
\begin{remark}
Example \ref{hypc}, Theorem \ref{flt} and Theorem \ref{33s} describe
{\it all} locally nonequivalent {\it homogeneous} structures of an oriented congruence
having vanishing shear, nonvanishing twist and with the invariant $K_2\equiv
0$. They may have
\begin{itemize}
\item[-] maximal symmetry group of dimension 5, and then
they are locally isomorphic to the Heisenberg group with the
standard splitting.
\item[-] symmetry group of exact dimension 4, and then they are locally
isomorphic to one of the two nonequivalent structures of Theorem \ref{flt}.
\item[-] symmetry group of exact dimension 3 which must be of either Bianchi type $VI_h$ or $IV$; in this case
they are given by Propositions \ref{pr3s1} and \ref{pr3s2}.
\end{itemize}
\end{remark}
\subsection{The case $K_2\neq 0$}\label{k2niez}
Looking at the explicit expresion for $K_2$ in (\ref{k22}) we see that
in this case we may fix both $\rho$ and $\phi$ by the requirement that
\be
K_2=1.\label{k21}
\ee
Indeed this normalization forces $\rho$ and $\phi$ to be
$$\rho=|k_2|^{\tfrac13},\quad\quad \phi={\rm Arg}(k_2).$$
This provides an embedding of $M$ into
$P$. Using it (technically speaking, 
by inserting $\rho$ and $\phi$ in the definitions of the
invariant coframe (\ref{n77})) we pullback the forms
$(\om_1,\bar{\om}_1,\om,\Omega,\bar{\Omega})$ on $P$ to $M$. Also $K_1$
is pullbacked to $M$, so that
$$K_1=\frac{k_1}{|k_2|^{\tfrac23}}.$$
Since $M$
is 3-dimensional the pullbacked forms are no longer linearly
independent, 
and the pulback of the derived form $\Omega$ decomposes onto the invariant coframe
$(\om_1,\bar{\om}_1,\om)$ on $M$. We denote the coefficients of this 
decomposition by $(Z_1,Z_2,Z_0)$ so that:
\begin{eqnarray*}
&&\Omega=Z_1\om_1+Z_2\bar{\om}_1+Z_0\om\\
&&\bar{\Omega}=\bar{Z}_2\om_1+\bar{Z}_1\bar{\om}_1+\bar{Z}_0\om.
\end{eqnarray*}
These coefficients constitute the basic scalar invariants of the
structures under consideration. They satisfy the following differential
system:
\begin{eqnarray}
&&\der\om=i\om_1\dz\bar{\om}_1+(Z_1+\bar{Z}_2)\om_1\dz\om+(Z_2+\bar{Z}_1)\bar{\om}_1\dz\om\nonumber\\
&&\der\om_1=-Z_2\om_1\dz\bar{\om}_1-Z_0\om_1\dz\om\label{szy}\\
&&\der\bar{\om}_1=\bar{Z}_2\om_1\dz\bar{\om}_1-\bar{Z}_0\bar{\om}_1\dz\om\nonumber
\end{eqnarray}
with
\begin{eqnarray}
&&\der[Z_1\om_1+Z_2\bar{\om}_1+Z_0\om]=K_1\om_1\dz\bar{\om}_1+\om_1\dz\om\nonumber\\
&&\der[\bar{Z}_2\om_1+\bar{Z}_1\bar{\om}_1+\bar{Z}_0\om]=-K_1\om_1\dz\bar{\om}_1+\bar{\om}_1\dz\om.\nonumber
\end{eqnarray}
Instead of considering the last two equations above, it is
convenient to replace them by the integrability conditions for the
system (\ref{szy}). These are:
\begin{eqnarray}
\der Z_1=Z_{11}\om_1+(-K_1+i
Z_0-Z_1Z_2+Z_2\bar{Z}_2+Z_{21})\bar{\om}_1+(Z_0\bar{Z}_2+Z_{01}-1)\om\nonumber\\
\der \bar{Z}_1=(-K_1-i
\bar{Z}_0-\bar{Z}_1\bar{Z}_2+Z_2\bar{Z}_2+\bar{Z}_{21})\om_1+\bar{Z}_{11}\bar{\om}_1+(\bar{Z}_0Z_2+\bar{Z}_{01}-1)\om\nonumber\\
\der
Z_2=Z_{21}\om_1+Z_{22}\bar{\om}_1+(Z_{02}+Z_0\bar{Z}_1+Z_0Z_2-\bar{Z}_0Z_2)\om\nonumber\\\der
\bar{Z}_2=\bar{Z}_{22}\om_1+\bar{Z}_{21}\bar{\om}_1+(\bar{Z}_{02}+\bar{Z}_0Z_1+\bar{Z}_0\bar{Z}_2-Z_0\bar{Z}_2)\om\label{szy1}\\
\der Z_0=Z_{01}\om_1+Z_{02}\bar{\om}_1+Z_{00}\om\nonumber\\
\der
\bar{Z}_0=\bar{Z}_{02}\om_1+\bar{Z}_{01}\bar{\om}_1+\bar{Z}_{00}\om\nonumber\\
\der K_1=K_{11}\om_1+\bar{K}_{11}\bar{\om}_1+K_{10}\om,\nonumber
\end{eqnarray}
where, in addition to the basic scalar invariants $Z_0,Z_1,Z_2,K_1$,
we have 
introduced the scalar invariants of the next higher order:
$Z_{00},Z_{01},Z_{02},Z_{11},Z_{21},Z_{22}$ (complex) and
  $K_{10}$ (real). Note that if the basic scalar invariants
$Z_0,Z_1,Z_2,K_1$ were constants, all the higher order invariants such
as $Z_{00},Z_{01},Z_{02},Z_{11},Z_{21},Z_{22}, K_{10}$ would be
identically vanishing.   
   
\begin{theorem}\label{337}
All locally nonequivalent structures $(M,[\lambda,\mu])$ of oriented
congruences having vanishing shear, nonvanishing twist, and with
$K_2\neq 0$ are described by the invariant system (\ref{szy}) with the
integrabilty conditions (\ref{szy1}).
\end{theorem}

Now we pass to the determination of all nonequivalent structures with
$K_2\neq 0$ which have a strictly 3-dimensional transitive group of
symmetries. They correspond to the structures of Theorem \ref{337}
with all the scalar invariants being constants. It turns out that
there are two families of such structures. The first family is
described by the following invariant system:
\begin{eqnarray*}
\der\om_1={\rm
  e}^{i\al}[-(2\sin\al)^{-1/3}\om_1\dz\bar{\om}_1-(2\sin\al)^{1/3}\om_1\dz\om],\\
\der\bar{\om}_1={\rm
  e}^{-i\al}[(2\sin\al)^{-1/3}\om_1\dz\bar{\om}_1-(2\sin\al)^{1/3}\bar{\om}_1\dz\om],\\
\der \om=i\om_1\dz\bar{\om}_1+(2\sin\al)^{-1/3}({\rm
  e}^{i\al}\om_1\dz\om+{\rm e}^{-i\al}\bar{\om}_1\dz\om).
\end{eqnarray*}
All the nonvanishing scalar invariants here are:
$$K_1=(2\sin\al)^{-2/3}$$
and
$$ Z_1=i(2\sin\al)^{2/3},\quad Z_2={\rm
  e}^{i\al}(2\sin\al)^{-1/3},\quad Z_0={\rm
  e}^{i\al}(2\sin\al)^{1/3}.$$
%\edz{Here the Bianchi type and the local representation is missing}
Two different values $\al$ and $\al'$ of the parameter yield 
different respective 
  quadruples $(K_1,Z_0,Z_1,Z_2)$ and $(K_1',Z_0',Z_1',Z_2')$, and
  hence correspond to nonequivalent structures.\\

The second family of nonequivalent structures with a strictly
3-dimensional group of symmetries corresponds to the following
invariant system:
\begin{eqnarray}
&&\der \om=i\om_1\dz\bar{\om}_1+i\beta^{-1} \om\dz (\om_1-\bar{\om}_1)\nonumber\\
&&\der \om_1=-i(\beta\om+\beta^{-1}\bar{\om}_1)\dz \om_1\label{b89}\\
&&\der \bar{\om}_1=i(\beta\om+\beta^{-1}\om_1)\dz \bar{\om}_1.\nonumber
\end{eqnarray}
The nonvanishing scalar invariants here are:
\be
K_1=(\beta^3+3)\beta^{-2},\quad Z_1=-2i \beta^{-1},\quad Z_2=-i
\beta^{-1},\quad Z_0=-i\beta.
\label{ok1}\ee
The corresponding structures of an oriented congruence 
are parametrized by a real parameter $\beta\neq 0$. This means that
each $\beta\neq 0$ defines a distinct structure. 

A further
analysis of this system shows that the congruence structures described
by it 
have a transitive symmetry group of Bianchi type $VII_0$ (iff
$\beta=-2^{\frac{1}{3}}$), Bianchi type $VIII$ (iff $\beta>-2^{\frac{1}{3}}$),
and of Bianchi type $IX$ (iff $\beta<-2^{\frac{1}{3}}$).

If we parametrize the 3-dimensional manifold $M$ by
$(u,z,\bar{z})$, the structures $(M,\lambda,\mu)$ corresponding to the system (\ref{b89}) 
may be locally represented by:
\begin{eqnarray}
&&\la=\der u +\frac{2\beta{\rm
    e}^{-i\beta u}+i\bar{z}}{\beta(z\bar{z}-2\beta^2(2+\beta^3))}\der
z+\frac{2\beta{\rm
    e}^{i\beta u}-iz}{\beta(z\bar{z}-2\beta^2(2+\beta^3))}\der
\bar{z}\label{cast}\\&&
\mu=-\frac{2\beta^2{\rm e}^{-i\beta u}}{z\bar{z}-2\beta^2(2+\beta^3)}\der
z,\quad\quad 
\bar{\mu}=-\frac{2\beta^2{\rm e}^{i\beta u}}{z\bar{z}-2\beta^2(2+\beta^3)}\der\bar{z}.\nonumber
\end{eqnarray}

Note that the above $(\la,\mu)$ can be also used to define a \emph{CR structure}
on $M$. Despite the fact that the 3-dimensional CR structures are associated
with this $(\la,\mu)$ by fairly more general transformations,
$(\la,\mu)\to(f\la,h\mu+p\la)$, than the \emph{oriented congruence  
structures}, which are defined by the \emph{restricted}
$(\la,\mu)\to(f\la,h\mu)$ transformations, each $s\neq 0$ in
(\ref{cast}) defines also a \emph{distinct} 
CR structure in the sense of Cartan.

Three particular values of $\beta\neq 0$ in (\ref{cast}) are worthy of
mention. These are:
$$\beta=\beta_B=-2^{\frac{1}{3}},$$ 
when the local symmetry
  group (both the CR and the oriented congruence symmetry) changes the
  structure from Bianchi type $IX$, with $\beta<\beta_B$; through Bianchi type
  $VII_0$, with $\beta=\beta_B$; to Bianchi type $VIII$, with $\beta>\beta_B$. 

Next is: 
$$\beta=\beta_H=-1,$$ 
when the lowest order \emph{Cartan} invariant of the
  CR structure associated with $\la_{\beta_H}$ and $\mu_{\beta_H}$ is identically
  vanishing \cite{tafel}; in this case the CR structure becomes locally equivalent
  to the Heisenberg group CR structure, and the 3-dimensional
  transitive CR
  symmetry group of
Bianchi type $IX$ is extendable, from the local $\sog(3)$ group, to the
8-dimensional local CR symmetry group $\sug(2,1)$.

The third distinguished $\beta$ is: 
$$\beta=\beta_{K}=-3^{\frac{1}{3}}.$$ 
Note that for $\beta=\beta_K$ \emph{our} invariant $K_1$ of the
congruence structure $(\la_\beta,\mu_\beta)$ vanishes, $K_1\equiv 0$, as in
(\ref{ok1}).  This case is of some importance, since it will be shown in Section \ref{bbaa} that the
congruence structures with $K_1\equiv 0$ and $K_2\neq 0$ have very nice
properties.

\section{Vanishing twist and nonvanishing shear} \label{vtns} Now we assume the
opposite of Section \ref{ntvs}, namely that $(M,[\lambda,\mu])$
has some shear, $s\neq 0$, but has identically vanishing twist,
$a\equiv 0$. As in Section \ref{vts} the no twist condition
$\der\lambda\dz\lambda\equiv 0$ yields $\lambda=f\der t$ for
some real function $t$ on $M$. Thus in this case we again have a
foliation of $M$ by the level surfaces $t=const$. Each leaf 
$\mathcal C$ of
this foliation is a 2-dimensional real submanifold which is equipped with
a complex structure $J$ determined by the requirement that its
holomorphic vector bundle $H^{1,0}=\{X-i JX, X\in \Gamma(T{\mathcal
  C})\}$ coincides with the anihilator of $Span_\bbC(\lambda)\oplus
Span_\bbC(\bar{\mu})$. But the simple situation of $M$ being locally
equivalent to $\bbR\times\bbC$ is no longer true. If $s\neq 0$ the
manifold $M$ gets equipped with the structure of a fibre
bundle ${\mathcal C}\to M\to V$, with fibres ${\mathcal C}$ being
1-dimensional complex manifolds -- the leaves of the foliation given
by $t=const$, and with the base $V$ being 1-dimensional, and
parametrized by $t$. This can be rephrased by saying that
we have a 1-parameter family of complex curves ${\mathcal C}(t)$,
with complex structure tensors $J_{{\mathcal C}(t)}$, which are {\it
  not} 
invariant under Lie
transport along the vector field $\partial_t$. Recall that
having a complex structure in a real 2-dimensional vector space is
equivalent to having a conformal metric and an orientation in the
space. Thus the condition of having $s\neq 0$ means that, under Lie transport along
$\partial_t$, the metrics on
the 2-planes tangent to the surfaces $t=const$ change in a fashion more
general than conformal. This means that small circles on these
two planes do not go to small circles when Lie transported along
$\partial_t$. They may, for example, be distorted into small ellipses, which
intuitively means that the congruence generated by $\partial_t$ has
shear. This explains the name of the complex parameter $s$, as was
promised in Section \ref{sec5}.

We now pass to a more explicit description of this situation. We start
with an arbitrary structure $(M,[\lambda,\mu])$ with
$\der\lambda\dz\lambda=0$. This guarantees that the 2-dimensional
distribution anihilating $\lambda$ defines a foliation in $M$, and $M$
is additionally
equipped with a transversal congruence of curves. Note that a
foliation of a 3-space by 2-surfaces equipped with a congruence {\it locally} 
can either be described in terms of coordinates $(t,x,y)$ such that the tangent
vector to the congruence is $\partial_t$ (in such case the surfaces
are in general curved
for each value of the parameter $t$), or in terms of coordinates
$(u,z,\bar{z})$ such that locally the surfaces are 2-planes (in such 
case the congruence is tangent to a vector field with a more
complicated representation 
$X=\partial_u+S\partial_z+\bar{S}\partial_{\bar{z}}$. Regardless of the descriptions the leaves of the foliation are given by the level
surfaces of the real parameters $t=const$ (in the first case, as in
the begining of this Section) or $u=const$ (as it will be used in this
Section from now on). Having this in mind and recalling the allowed
transformations (\ref{ec}) we conclude that our
$(M,[\lambda,\mu])$ with $\der\lambda\dz\lambda=0$ may be represented
by a pair of 1-forms  
$$
\lambda=\der u,\quad\quad\quad\mu=\der z +H\der\bar{z}+G\der u,
$$
where $H=H(u,z,\bar{z})$ and $G=G(u,z,\bar{z})$ are complex-valued
functions on $M$, with coordinates $(u,z,\bar{z})$, such that
$|H|<1$. 
The foliation has leaves tangent to the vector fields $\partial_z$,
$\partial_{\bar{z}}$. Each leaf is equipped with a complex
structure, which may be described by saying that its $T^{(1,0)}$ space
is spanned by the vector field 
\be
Z=\partial_z-\bar{H}\partial_{\bar{z}};
\ee
consequently the $T^{(0,1)}$ space is spanned by the complex conjugate
vector field
$$
\bar{Z}=\partial_{\bar{z}}-H\partial_z.
$$
The congruence on $M$ which gives the preferred splitting is tangent
to the real vector field
\be
X=\partial_u+\tfrac{\bar{G}H-G}{1-H\bar{H}}\partial_z+\tfrac{G\bar{H}-\bar{G}}{1-H\bar{H}}\partial_{\bar{z}}.
\ee
Thus we have the following proposition.
\begin{proposition}\label{hg}
All structures $(M,[\lambda,\mu])$ with vanishing twist,
$a\equiv 0$, may be locally represented by 
\be
\lambda=\der u,\quad\quad\quad\mu=\der z +H\der\bar{z}+G\der
u,\label{aneze}
\ee
where $H=H(u,z,\bar{z})$ and $G=G(u,z,\bar{z})$ are complex-valued
functions on $M$, with coordinates $(u,z,\bar{z})$, such that
$|H|<1$. They have
nonvanishing shear $s\neq 0$ iff
$$H_u-GH_z+HG_z-G_{\bar{z}}\neq 0.$$
\end{proposition}

The following two cases are of particular interest:
\begin{itemize}
\item $H\equiv 0$. In this case all surfaces $u=const$ are equipped
  with the standard complex structure. The coordinate $z$ is the
  holomorphic coordinate for it, but the congruence is tangent to a
  complicated real vector field
  $X=\partial_u-G\partial_z-\bar{G}\partial_{\bar{z}}.$
\item $G\equiv 0$. Here each surface $u=const$ has its own complex
  structure $J$, for which $z$ is not a holomorphic coordinate; $J$
  is determined by specifying a complex function $H$. A nice feature
  of this case is that the congruence is now tangent to the very
  simple vector field $X=\partial_u$, which enables us to identify
  coordinates $t$ and $u$.
\end{itemize} 
Note that in Proposition \ref{hg} we made an assumption about the
modulus of the function $H$. The modulus equal to one is excluded
because it violates the condition that the forms $\lambda$, $\mu$,
$\bar{\mu}$ are independent. We excluded also the $H>1$ case, since 
because of the coordinate transformation $z\to\bar{z}$
followed by $H\to 1/H$, such structures are
in one to one equivalence with those having $|H|<1$. We now turn to
the question about nonequivalent structures among those covered by Proposition
\ref{hg}.
\subsection{The invariant $T_0$ and the relative
  invariants $T_1$, $K_0$, $K_1$}
To answer this we have to go back
to the begining of Section \ref{sec5} and again
perform the Cartan analysis on the system (\ref{str}), but now with 
$a\equiv 0$, $s\neq 0$. In this case the formulae (\ref{str}) become
\beq
\der\lambda&=&
b\mu\dz\lambda+\bar{b}\bar{\mu}\dz\lambda\nonumber\\
\der\mu&=&p\mu\dz\bar{\mu}+q\mu\dz\lambda+s\bar{\mu}\dz\lambda\label{st1}\\
\der\bar{\mu}&=&-\bar{p}\mu\dz\bar{\mu}+\bar{s}\mu\dz\lambda+\bar{q}\bar{\mu}\dz\lambda.\nonumber
\eeq
It is convenient to write the complex shear function $s$ as
$$s=|s|{\rm e}^{i\psi}.$$ Now for a chosen pair $(\lambda,\mu)$
representing the structure, using (\ref{st1}), we find that the
differentials of the Cartan frame  
\be
(\omega,\omega_1,\bar{\omega}_1)=(f\lambda,\rho {\rm e}^{i\phi}\mu,
\rho {\rm e}^{-i\phi}\bar{\mu})\label{cfra}\ee are:
\begin{eqnarray*}
\der\om&=&\der\log f\dz\om + \frac{b}{\rho}{\rm e}^{-i\phi}\om_1\dz\om + \frac{\bar{b}}{\rho}{\rm e}^{i\phi}\bar{\om}_1\dz\om\\
\der\om_1&=&i\der\phi\dz\om_1 + \der\log\rho\dz\om_1
+\frac{p}{\rho}{\rm e}^{i\phi}\om_1\dz\bar{\om}_1
+\frac{q}{f}\om_1\dz\om +  \frac{|s|}{f}{\rm e}^{i(2 \phi +
  \psi)}\bar{\om}_1\dz\om\\
\der\bar{\om}_1&=&-i\der\phi\dz\bar{\om}_1 + \der\log\rho\dz\bar{\om}_1 -\frac{\bar{p}}{\rho}{\rm e}^{-i\phi}\om_1\dz\bar{\om}_1 +  \frac{|s|}{f}{\rm e}^{-i(2 \phi + \psi)}\om_1\dz\om+\frac{\bar{q}}{f}\bar{\om}_1\dz\om. 
\end{eqnarray*}
Because of  $s\neq 0$, we can gauge the structure so that 
\be
\der\om_1\dz\om_1=\om_1\dz\bar{\om}_1\dz\om.\label{n01}
\ee
This requirement defines $f$ modulo sign to be $f=\pm |s|$.  Writing
$f$ as
$$f={\rm e}^{i\epsilon\pi}|s|,$$
where $\epsilon=0,1$, and still requiring the normalization
(\ref{n01}), we get
$$\phi=-\tfrac12\psi+\epsilon\tfrac{\pi}{2}.$$
Thus the functions $f$ and $\phi$ are fixed modulo $\epsilon$. 
 
After this normalization we introduce a {\it real} 1-form
$\Omega$ such that
\be
(\der\om_1-\der\bar{\om}_1)\dz(\om_1+\bar{\om}_1)=
2\Om\dz\om_1\dz\bar{\om}_1.\label{n02}
\ee
This equation defines $\Omega$ to be
$$
\Om=\der\log\rho+z\om_1+\bar{z}\bar{\om}_1+(1-{\rm e}^{i\epsilon\pi}\tfrac{q+\bar{q}}{2|s|})\om,
$$
where $z$ is an auxiliary complex parameter. The condition that fixes
$z$ in an algebraic fashion is:
\be
\der\om_1\dz\om=\Om\dz\om_1\dz\om,\quad\quad\quad\der\bar{\om}_1\dz\om=\Om\dz\bar{\om}_1\dz\om.\label{n03}
\ee
It uniquely specifies $z$ to be
$$z=\tfrac{(i\psi_\mu-2\bar{p})}{2\rho}{\rm
    e}^{\tfrac{i}{2}(\psi-\epsilon \pi)},\quad\quad\quad
    \bar{z}=\tfrac{(-i\psi_{\bar{\mu}}-2p)}{2\rho}{\rm
    e}^{-\tfrac{i}{2}(\psi-\epsilon\pi)}.$$
Thus given a structure $(M,[\lambda,\mu])$ with vanishing twist and
nonvanishing shear, the three normalization conditions (\ref{n01}),
(\ref{n02}), (\ref{n03}) uniquely specify a 4-dimensional
manifold $P$, which is locally $M\times\bbR_+$, and a well
defined coframe $(\omega,\omega_1,\bar{\omega}_1,\Omega)$
on it such that
\beq
\om&=&{\rm e}^{i\epsilon\pi}|s|\lambda\nonumber\\
\om_1&=&\rho {\rm e}^{-\tfrac{i}{2}(\psi-\epsilon\pi)}\mu\nonumber\\
\bar{\om}_1&=&\rho {\rm e}^{\tfrac{i}{2}(\psi-\epsilon\pi)}\bar{\mu}\label{n771}\\
\Om&=&\der\log\rho+\tfrac{(i\psi_\mu-2\bar{p})}{2\rho}{\rm
    e}^{\tfrac{i}{2}(\psi-\epsilon \pi)}\om_1+\tfrac{(-i\psi_{\bar{\mu}}-2p)}{2\rho}{\rm
    e}^{-\tfrac{i}{2}(\psi-\epsilon\pi)}\bar{\om}_1+\nonumber\\
&&(1-{\rm e}^{i\epsilon\pi}\tfrac{q+\bar{q}}{2|s|})\om.
\nonumber
\eeq
Here the positive coordinate along the factor $\bbR_+$ in the
fibration $\bbR_+\to P\to M$
is $\rho$. The coframe $(\omega,\omega_1,\bar{\omega}_1,\Omega)$
satisfies
\beq
\der\om&=&T_1\om_1\dz\om+\bar{T}_1\bar{\om}_1\dz\om\nonumber\\
\der\om_1&=&\Om\dz\om_1+(\om_1+\bar{\om}_1)\dz\om+iT_0\om_1\dz\om\nonumber\\
\der\bar{\om}_1&=&\Om\dz\bar{\om}_1+(\om_1+\bar{\om}_1)\dz\om-iT_0\bar{\om}_1\dz\om\label{systey}\\
\der\Omega&=&iK_0\om_1\dz\bar{\om}_1+K_1\om_1\dz\om+\bar{K}_1\bar{\om}_1\dz\om\nonumber
\eeq where
\be
T_0=\tfrac{\psi_\lambda+i(\bar{q}-q)}{2|s|}{\rm e}^{i\epsilon\pi},\quad\quad T_1=\frac{t_1}{\rho},\quad\quad K_0=\frac{k_0}{2\rho^2},\quad\quad
K_1=\frac{k_1}{2\rho}\label{to}
\ee
and
\begin{eqnarray}
t_1&=&(b|s|+|s|_\mu)\frac{e^{\tfrac{i}{2}(\psi-\epsilon\pi)}}{|s|}\nonumber\\
k_0&=&-\psi_{\mu\bar{\mu}}-\psi_{\bar{\mu}\mu}+p\psi_\mu+\bar{p}\psi_{\bar{\mu}}+2i(p_\mu-\bar{p}_{\bar{\mu}})\label{wwn}\\
k_1&=&2(t_1-\bar{t}_1)+\nonumber\\
&&{\rm e}^{\tfrac{i}{2}\epsilon\pi}[(b \bar{q}-b q -q_\mu+\bar{q}_\mu+i q
  \psi_\mu-i\psi_{\mu\lambda}){\rm e}^{\tfrac{i}{2}\psi}+
i\psi_{\bar{\mu}}|s|{\rm e}^{-\tfrac{i}{2}\psi}]|s|^{-1}.\nonumber
\end{eqnarray}
Note that functions $T_0$, $T_1$, $K_0$ and $K_1$ are invariants of
the structure on the {\it bundle} $\bbR_+\to P\to M$, with the fiber
coordinate $\rho$. They are defined modulo the parameter
$\epsilon=0,1$. Thus two structures which differ only by the
value of $\epsilon$ are equivalent. \\

If we want to look for the invariants on the {\it original} manifold
$M$ we must examine 
the fiber coordinate dependence of the structural functions 
$T_0$, $T_1$, $K_0$ and $K_1$. Since the last three functions $T_1$,
$K_0$, $K_1$ have a nontrivial $\rho$ dependence they do not project to
invariant functions on $M$. However, since in all these cases this
dependence is just 
{\it scaling} by $\rho$ we 
conclude that they lead to the {\it relative} invariants on $M$. Thus
the vanishing or not of any of the functions $t_1,k_1$ (complex), $k_0$
(real) is an invariant
property of the structure on $M$. The situation is quite different for the
real function $T_0$. Although originally defined on $P$ it is 
{\it constant} along the fibers. Thus it projects to a well defined
invariant on the original manifold $M$. Thus $T_0$ is an invariant of the
structure on $M$. We summarize the above discussion in the following Theorem.

\begin{theorem}\label{th324}
A given structure $(M,[\lambda,\mu])$ of an oriented congruence with
vanishing twist, $a\equiv 0$, and nonvanishing shear, $s\neq 0$,
uniquely defines a 4-dimensional manifold $P$, 1-forms
$\om,\om_1,\bar{\om}_1,\Om$
and functions  $T_0, K_0$ (real) $T_1,K_1$ (complex) on $P$
such that
\begin{itemize}
\item[-] $\om,\om_1,\bar{\om}_1,\Om$ are as in {\rm (\ref{n771})},
\item[-]
  $\om\dz\om_1\dz\bar{\om}_1\dz\Om\neq
  0$ at each point of $P$,
\item[-] the forms and functions $T_0,T_1,K_0,K_1$ 
are uniquely determined by the requirement that on $P$
they satisfy equations {\rm (\ref{systey})}.
\end{itemize}
In particular $T_0$ is an invariant of the structure on $M$; 
the identical vanishing, or not, of either of the
functions $t_1$, $k_0$ or $k_1$ defined in (\ref{wwn}) is an 
invariant condition on $M$.
\end{theorem}

The structures covered by Theorem \ref{th324} admit symmetry groups of
{\it at most} four dimensions. Those for which the symmetry group is strictly
 4-dimensional have {\it all} the relative invariants $t_1$, $k_0$, $k_1$ 
equal to {\it zero} and {\it constant} invariant $T_0$. When
finding such structures it is enough to consider $T_0=\al= const\geq 0$
since, due to the fact that $T_0$ is defined modulo sign (${\rm
  e}^{i\epsilon\pi}=\pm 1$), each
 structure with $T_0=\al<0$ is equivalent to the one 
with $T_0=|\al|$. Inspecting all the possibilities we get the
 following theorem.
\begin{theorem}\label{4symtm}
All locally nonequivalent structures $(M,[\lambda,\mu])$ of oriented
congruences having vanishing twist, nonvanishing shear, and
\emph{possessing a strictly 4-dimensional transitive group} \emph{of
  symmetries} are parametrized by a real constant $\alpha\geq 0$ as
follows.
\begin{itemize}
\item if $0\leq\alpha<1$ they can be locally represented by
$$\lambda=\der u,\quad\quad \mu=\der x+{\rm
  e}^{2u\sqrt{1-\alpha^2}}(\alpha+i\sqrt{1-\alpha^2})\der y$$ 
\item if $\al=1$ they can be locally represented by 
$$\lambda=\der u,\quad\quad \mu=\der x+(i+2u)\der y$$ 
\item if $\al>1$ they can be locally represented by
\begin{eqnarray*}
\lambda&=&\der u,\\
\mu&=&[(i+\al)\cos(u\sqrt{\al^2-1})-i\sqrt{\al^2-1}\sin(u\sqrt{\al^2-1})]\der
x+\\
&&[(i+\al)\sin(u\sqrt{\al^2-1})+i\sqrt{\al^2-1}\cos(u\sqrt{\al^2-1})]\der
y.
\end{eqnarray*}
\end{itemize}
Here $(u,x,y)$ are coordinates on $M$.
The real parameter $\alpha\geq 0$ is just the invariant $T_0=\alpha$ 
and as such enumerates nonequivalent structures.
\end{theorem}
\subsection{Description in terms of the Cartan connection}
Equations (\ref{systey}) can be better understood in terms of the
matrix $\tilde{\om}$ of 1-forms defined by
$$
\tilde{\om}=
\begin{pmatrix}
2(\Om-\om)&0&0\\
&&\\
\om_1&\Om-\om&\om\\
&&\\
\bar{\om}_1&\om&\Om-\om,
\end{pmatrix}
$$
where the 1-forms $(\om_1,\bar{\om}_1,\om,\Om)$ are as in
(\ref{systey}) or as is (\ref{n771}). 

This matrix has values in the 4-dimensional Lie algebra
$\mathfrak{g}_4$ which is a semidirect product of two 2-dimensional
Abelian Lie algebras 
$$\mathfrak{h}_0=\{~ \begin{pmatrix}
2x&0&0\\
&&\\
0&x&y\\
&&\\
0&y&x
\end{pmatrix}~|~ x,y\in\bbR~ \}
$$
and 
$$\mathfrak{h}_1=\{~ \begin{pmatrix}
0&0&0\\
&&\\
u+i v&0&0\\
&&\\
u-i v&0&0
\end{pmatrix}~|~ u,v\in\bbR~ \},
$$
for which the commutator is the usual commutator of $3\times 3$
matrices. Thus \be\mathfrak{g}_4=\mathfrak{h}_0\oplus\mathfrak{h}_1,\label{sptil}\ee
as the direct sum of vector spaces $\mathfrak{h}_0$ and $\mathfrak{h}_1$, with the commutator between
$\mathfrak{h}_0$ and $\mathfrak{h}_1$ given by 
$$[\mathfrak{h}_0,\mathfrak{h}_1]\subset\mathfrak{h}_1.$$

It turns out that due to the relations (\ref{systey}),  
$\tilde{\om}$ is a Cartan connection on the principal fibre bundle $\bbR_+\to
P\to M$, which has as its structure group a 1-parameter Lie group
generated by the vector field $\rho\partial_\rho$ dual to $\Om$.

\begin{remark}\label{bif}
It is worthwile to note that the fibre bundle $\bbR_+\to P\to M$ has
some additional structure. Indeed, equations (\ref{systey}) guarantee
that $P$ is foliated by 2-dimensional leaves of the \emph{integrable} 
2-dimensional real distribution $\mathcal{D}$ 
anihilating forms $\om_1$ and $\bar{\om}_1$. Thus, locally, 
$P$ has also the structure of 
a fibre bundle over the leaf space $P/{\mathcal D}$. This is actually
a \emph{principal} fiber bundle $H_0\to P\to P/{\mathcal D}$, with the  
structure group $H_0$ having $\mathfrak{h}_0$ as its Lie algebra.

\end{remark}

Equations (\ref{systey}) imply that the curvature $R$ of
the Cartan connection $\tilde{\om}$ is
$$
R=\der \tilde{\om}+\tilde{\om}\dz\tilde{\om}=
\begin{pmatrix}
2R_1&0&0\\
R_3&R_1&R_2\\
\bar{R}_3&R_2&R_1
\end{pmatrix},
$$
where
\beq
R_1&=&i K_0\om_1\dz\bar{\om}_1+(K_1-T_1)\om_1\dz\om+
(\bar{K}_1-\bar{T}_1)\bar{\om}_1\dz\om\nonumber\\
R_2&=&T_1\om_1\dz\om+\bar{T}_1\bar{\om}_1\dz\om\nonumber\\
R_3&=&i T_0\om_1\dz\om.\nonumber
\eeq
In particular the absence of vertical $\Omega\dz$ terms in the curvature
confirms our interpretation of $\tilde{\om}$ as a
$\mathfrak{g}_4$-valued Cartan connection on $P$ over $M$.

The Cartan connection $\tilde{\om}$ yields all the invariant information about the corresponding
structures $(M,[\lambda,\mu])$ and can be used in an invariant
description of various examples of such structures. In particular, the
invariant decomposition (\ref{sptil}) may be used to distinguish two
large classes $(M,[\la,\mu])_0$ and $(M,[\la,\mu])_1$ 
of nonequivalent structures $(M,[\la,\mu])$. These are defined by the
requirement that the curvature $R$ of their Cartan connection
$\tilde{\om}$ has values in the respective parts $\mathfrak{h}_0$ for
$(M,[\la,\mu])_0$, and  $\mathfrak{h}_1$ for $(M,[\la,\mu])_1$. 

\subsubsection{Curvature $R\in\mathfrak{h}_0$}\label{koso1} The curvature $R$ of the
Cartan connection $\tilde{\om}$ resides in $\mathfrak{h}_0$ iff it is
of the form
$$
R=
\begin{pmatrix}
2R_1&0&0\\
0&R_1&R_2\\
0&R_2&R_1
\end{pmatrix}.
$$ 
An example of a
structure $(M,[\la,\mu])$ with such $R$ is given by the
following forms $(\om_1,\bar{\om}_1,\om,\Om)$:
\begin{eqnarray*}
&&\om_1={\rm e}^r(\der x+i {\rm e}^{2(u+f)}\der y),\\
&&\bar{\om}_1={\rm e}^r(\der x-i {\rm e}^{2(u+f)}\der y),\\
&&\om=\der u,\\
&&\Om=\der r+2 \der u+2f_x\der x,
\end{eqnarray*}
with a real function $f=f(x,y)$ of real variables $x$ and $y$. These
two variables, supplemented with the real $u$ and $r$, constitute a
coordinate system  
$(u,x,y,r)$ on  $R_+\to P\to M$. The triple $(u,x,y)$ parametrizes
$M$, and $r$ is related to the positive fiber coordinate $\rho$ via $\rho={\rm e}^r$. 

For each choice of a twice differentiable function $f=f(x,y)$ 
the forms $(\om_1,\bar{\om}_1,\om,\Om)$ satisfy the differential system (\ref{systey}) with
$$
K_1\equiv 0,\quad\quad T_1\equiv 0,\quad\quad T_0\equiv 0,
$$
and the relative invariant $K_0$ being
$$K_0=-{\rm e}^{-2(r+u+f)}f_{xy}.$$
A special case here is $f_{xy}\equiv 0$, in particular $f\equiv
0$. If this happens the corresponding structures
$(M,[\la,\mu])$ are all equivalent to the structure with 4-dimensional transitive
symmetry group having $\al=0$ in Theorem \ref{4symtm}. If $f_{xy}\neq
0$, then $K_0\neq 0$, and the corresponding structures have the curvature of the Cartan
connection $\tilde{\om}$ in the form 
$$
R=
-{\rm e}^{-2(r+u+f)}\begin{pmatrix}
2i \om_1\dz\bar{\om}_1&0&0\\
0&i\om_1\dz\bar{\om}_1&0\\
0&0&i\om_1\dz\bar{\om}_1
\end{pmatrix}f_{xy}.
$$ 
As such they are \emph{special cases} of structures with
$R\in\mathfrak{h}_0$. We will retutn to them in Section \ref{koso},
where we further analyze the case $K_0\neq 0$, $T_1=0$ and $K_1=0$.

\subsubsection{Curvature $R\in\mathfrak{h}_1$}\label{zaza}

The case of $R\in\mathfrak{h}_0$ is entirely characterized by the
requirement that all 
the relative invariants $t_1,k_0,k_1$ identically vanish. Examples of
such structures are structures with a 4-dimensional transitive group of
symmetries given in Theorem \ref{4symtm}. However these examples do
not exhaust the list of nonequivalent structures having
$R\in\mathfrak{h}_1$. 
To find them \emph{all} we proceed as follows.

We want to find all structures with
$$
R=
\begin{pmatrix}
0&0&0\\
R_3&0&0\\
\bar{R}_3&0&0
\end{pmatrix},
$$
i.e. those for which \emph{all} the relative invariants $T_1$, $K_1$,
$K_0$, as in (\ref{systey}), vanish: 
\be
T_1\equiv 0,\quad\quad K_0\equiv 0,\quad\quad 
K_1\equiv 0.\label{zaz}
\ee
Assuming (\ref{zaz}), equations (\ref{systey}) guarantee that real coordinates $u$ and
$r$ may be introduced on $P$ such that
$$\om=\der u,\quad\quad \Om=\der r.$$
Then, taking the exterior derivatives of both sides of equations
(\ref{systey}), we see that (\ref{zaz}) forces $T_0$ to be a real
function of $u$ only. Denoting this function by $\alpha=\alpha(u)$ we have
$$T_0=\alpha(u).$$ 
Integrating the system for such $T_0$, and denoting the
$u$-derivatives by primes, we get the following theorem.
\begin{theorem}\label{alu}
A structure $(M,[\lambda,\mu])$ of an oriented congruence with
vanishing twist, $a\equiv 0$, nonvanishing shear, $s\neq 0$, and having 
the curvature of its corresponding Cartan connection $\tilde{\om}$ of
the pure $\mathfrak{h}_1$ type, $R\in\mathfrak{h}_1$, can be
locally represented by
$$\lambda=\der u,\quad\quad \mu=\der
z-(\frac{\bar{h}'}{h}+\frac{\bar{h}}{h}-i\al\frac{\bar{h}}{h})\der\bar{z},$$
where the complex function $h=h(u)\neq 0$ satisfies a second order ODE:
\be
h''+2h'+(\al^2+i\al')h=0.\label{alul}\ee
Here the nonequivalent structures are
distinguished by the real invariant $T_0=\al(u)$.
\end{theorem}
Note that if $\al(u)=const$ we recover the structures from Theorem
\ref{4symtm}.\\

\subsection{The case $T_1\equiv 0$}

Now we pass to the general case $T_1\equiv 0$. To proceed we have to
distinguish two subcases:
\begin{itemize}
\item $K_1\equiv 0$
\item $K_1\neq 0$.
\end{itemize}
\subsubsection{The case $K_1\equiv 0$}\label{koso}
In this situation we have 
$$\der\Om=iK_0\om_1\dz\bar{\om}_1,$$
with $K_0$ given by (\ref{to})-(\ref{wwn}). Since $K_0$ is not
identically equal to zero, because this correponds to the case
$t_1\equiv 0$, $k_0\equiv 0$, $k_1\equiv 0$ already
studied, we use it to
fix $\rho$ by the
requirement 
\be 
K_0={\rm sign}(k_0)=\pm 1.\label{sn23}
\ee
We note that this sign is an invariant of the structures under
consideration. This implies that the structures with different signs
are nonequivalent.

After the normalization (\ref{sn23}) the forms
$(\om_1,\bar{\om}_1,\om,\Omega)$ are defined as forms on $M$. Performing
the standard Cartan analysis on the system (\ref{systey}), we verified
that after pullback to $M$ it reads: 
\begin{eqnarray}
&&\der \om=0,\nonumber\\
&&\der\om_1=(iB -A)\om_1\dz\bar{\om}_1+i T_0\om_1\dz\om+\bar{\om}_1\dz\om,\label{ssy}\\
&&\der\bar{\om}_1=(iB+A)\om_1\dz\bar{\om}_1+\om_1\dz\om-i T_0\bar{\om}_1\dz\om,\nonumber\\
&&\der[(A+iB)\om_1+(A-iB)\bar{\om}_1+\om]=\pm i\om_1\dz\bar{\om}_1.\nonumber
\end{eqnarray}
Here the real functions $A,B,T_0$ are the scalar invariants on $M$. They 
satisfy the following 
integrability conditions
\begin{eqnarray}
&&\der A=[A_1+\tfrac{i}{2}(B_1+\bar{B}_1\pm
  1)]\om_1+[A_1-\tfrac{i}{2}(B_1+\bar{B}_1\pm 1)]\bar{\om}_1+(A-B
T_0)\om\nonumber\\
&&\der B=B_1\om_1+\bar{B}_1\bar{\om}_1+(A T_0-B)\om\label{ssyp}\\
&&\der T_0\dz\om=0,\nonumber
\end{eqnarray}
with the functions $A_1$ (real) and $B_1$ (complex)
being the scalar invariants of the next higher order. In principle, we
could have written the explicit fotmulae for all these scalar
invariants in terms of the defining variables $b,q,p$ and $s$ of
(\ref{st1}). We refrain from doing this, because the formulae are
quite complicated, and not enlightening. 

We summarize these considerations in the following theorem.
\begin{theorem}\label{t0con}
All locally nonequivalent structures $(M,[\lambda,\mu])$ of oriented
congruences having vanishing twist, nonvanishing shear, with
$T_1\equiv 0$ and $K_1\equiv 0$, are described by the invariant forms
$(\omega,\omega_1,\bar{\omega}_1)$ satisfying the system
(\ref{ssy})-(\ref{ssyp}) 
on $M$.
\end{theorem}

Thus having a representative $(\lambda,\mu)$ of a structure with
vanishing twist, nonvanishing shear and with
$T_1\equiv 0$, we always can gauge it to the invariant
forms satisfying system (\ref{ssy})-(\ref{ssyp}). The other way around: given two 1-forms
$\om$ and $\omega_1$ satisfying the system (\ref{ssy})-(\ref{ssyp}), we may consider
them as a representative pair $(\lambda=\om ,\mu=\om_1)$ of a certain 
structure with
vanishing twist, nonvanishing shear and with
$T_1\equiv 0$.

The immediate consequence of the integrabilty conditions (\ref{ssyp})
is the {\it nonexistence} of structures (\ref{ssy}) 
with a strictly 3-dimensional
transitive group of symmetries. This is because, if such structures
existed, they would have {\it constant} invariants $A$, $B$ and
$T_0$. Thus, for such structures the right hand sides of all the equations
(\ref{ssyp}) would be zero. But
this is impossible, since in such a situation 
the second equation (\ref{ssyp}) implies $B_1\equiv 0$ which, when
compared with equating to zero the r.h.s of the first equation
(\ref{ssyp}), gives contradiction. 

A family of nonequivalent structures $(M,[\la,\mu])$ from this branch
of the classification is given in Section \ref{koso1}. Indeed,
consider the examples of this section for which $$f_{xy}\neq 0.$$ Since
this guarantees that $K_1\neq 0$, and since we have $T_1=0$ and $K_1=0$ (and, what is
less important for us here $T_0=0$) for them, we may perform the above
described normalization procedure on the invariant forms
$(\om_1,\bar{\om}_1,\om,\Om)$ defined in \ref{koso1}. A simple
calculation, based on the normalization
\be
-{\rm e}^{-2(r+u+f)}f_{xy}=\pm 1,\label{ipp}
\ee
leads to the reduction to $M$, where the invariant forms read:
\begin{eqnarray*}
&&\om=\der u,\\
&&\om_1={\rm e}^{-(u+f)}\big(\mp f_{xy}\big)^{\frac{1}{2}}(\der x+i{\rm e}^{2(u+f)}\der
  y),\\
&&\bar{\om}_1={\rm e}^{-(u+f)}\big(\mp f_{xy}\big)^{\frac{1}{2}}(\der x-i{\rm
    e}^{2(u+f)}\der y).
\end{eqnarray*}
They satisfy the system (\ref{ssy})-(\ref{ssyp}) with the functions $A$ and
$B$ given by:
\begin{eqnarray*}
&&A=\tfrac14\big(\mp f_{xy}\big)^{-\frac{3}{2}}\Big(2f_xf_{xy}+f_{xxy}\Big) {\rm e}^{u+f}\\
&&B=\tfrac14 \big(\mp
f_{xy}\big)^{-\frac{3}{2}}\Big(2f_yf_{xy}-f_{xyy}\Big){\rm e}^{-u-f}.
\end{eqnarray*}
These structures can thus be represented on $M$ by 
$$\la=\der u,\quad\quad\mu=\der x+i{\rm e}^{2\big(u+f(x,y)\big)}\der y.$$ 
The only scalar invariants for them are the functions $A$ and $B$ as
above, since as we already noticed, the scalar invariant $T_0$
identically vanishes, $T_0\equiv 0$.

Note in particular, that given a function $f=f(x,y)$, two structures $(M,[\la,\mu])$
with $\la$, $\mu$ as above, corresponding to two different signs of
$f_{xy}$ are nonequivalent. This is 
because the sign $\pm$ in (\ref{ipp}) is an invariant of such structures.   

\begin{remark}
The structures described above belong to a subclass of structures for which the
curvature $R$ is much more restricted than to $\mathfrak{h}_0$. Since,
in addition to $T_0\equiv 0$, we have here $T_1\equiv 0$, the curvature $R$
is actually contained in the diagonal 1-dimensional subalgebra of
$\mathfrak{h}_0$. Moreover, since also $K_1\equiv 0$, the curvature
$R$ does not involve $\om\dz$ terms. This means that in this example,
similarly as in all examples with $T_0\equiv T_1\equiv K_1\equiv
0$, the curvature of the Cartan connection $\tilde{\om}$ is
\emph{horizontal from the point of view of the principal 
fiber bundle} $H_0\to P\to P/{\mathcal D}$ discussed in
Remark \ref{bif}. Thus here, the Cartan connection $\tilde{\om}$ 
can be reinterpreted 
as a $\mathfrak{g}_4$-valued Cartan connection on the bundle $H_0\to P\to P/{\mathcal D}$ .
\end{remark}

\subsubsection{The case $K_1\neq 0$}
If $K_1\neq 0$ we can use definition (\ref{to}) to scale it in such a
way that it has values on the unit circle
$$K_1={\rm e}^{i\gamma}.$$
This fixes $\rho$ uniquely, and the system (\ref{systey}) is again reduced to
an invariant system on $M$. This reads (with new $A$ and $B$):
\begin{eqnarray}
&&\der \om=0,\nonumber\\
&&\nonumber\\
&&\der\om_1=(iB -A)\om_1\dz\bar{\om}_1+(1-C+i
  T_0)\om_1\dz\om+\bar{\om}_1\dz\om,\nonumber\\
&&\label{ssy1}\\
&&\der\bar{\om}_1=(iB+A)\om_1\dz\bar{\om}_1+\om_1\dz\om+(1-C-i
  T_0)\bar{\om}_1\dz\om,\nonumber\\
&&\nonumber\\
&&\der[(A+iB)\om_1+(A-iB)\bar{\om}_1+C\om]=\nonumber\\
&&iK_0\om_1\dz\bar{\om}_1+{\rm e}^{i\gamma}\om_1\dz\om+{\rm e}^{-i\gamma}\bar{\om}_1\dz\om.\nonumber
\end{eqnarray}
Here, all the real invariants are $T_0$, $A$, $B$, $C$, $\gamma$ and $K_0$ are well
defined functions on $M$. They are expressible in terms of the
original variables defining the structure and the functions
$k_0$, $k_1$ of (\ref{wwn}). In particular,
$$K_0=2\frac{k_0}{|k_1|^2}.$$
To discuss the integrabilty conditions for the system (\ref{ssy1}) 
we have to distinguish two cases:
\begin{itemize}
\item either $K_1={\rm e}^{i\gamma}\neq\pm 1$,
\item or $K_1={\rm e}^{i\gamma}\equiv\pm 1$.
\end{itemize}
In the first case:
\begin{eqnarray}
&&\der T_0=i({\rm e}^{i\gamma}\om_1-{\rm
    e}^{-i\gamma}\bar{\om}_2)+T_{00}\om\nonumber\\
&&\nonumber\\
&&\der A=\tfrac12[i(\tfrac{K_0}{2}+A_1)+A_2]\om_1+\tfrac12[-i(\tfrac{K_0}{2}+A_1)+A_2]\bar{\om}_1+A_0\om\nonumber\\
&&\nonumber\\
&&\der B=\tfrac12[-\tfrac{K_0}{2}+A_1+iB_1]\om_1+\tfrac12[-\tfrac{K_0}{2}+A_1-iB_1]\bar{\om}_1+B_0\om\nonumber\\
&&\label{ssy2}\\
&&\der C=[-2 A+A C+A_0+B T_0+i(B C-A T_0+B_0)+{\rm
    e}^{i\gamma}]\om_1+\nonumber\\
&&[-2 A+A C+A_0+B T_0-i(B C-A T_0+B_0)+{\rm e}^{-i\gamma}]\bar{\om}_1+C_0\om\nonumber\\
&&\nonumber\\
&&\der\gamma=[B+(A+\gamma_1)\cot\gamma+i\gamma_1]\om_1+[B+(A+\gamma_1)\cot\gamma-i\gamma_1]\bar{\om}_1+\gamma_0\om\nonumber\\
&&\nonumber\\
&&\der
    K_0=K_{01}\om_1+\bar{K}_{01}\bar{\om}_1+2[(A+\gamma_1)\csc\gamma+(1-C)K_0]\om,\nonumber
\end{eqnarray}
and in addition to the the basic scalar invariants 
$K_0$, $\gamma$, $A$, $B$, $C$, we have higher order scalar invariants
$A_0,A_1,A_2,B_0,B_1,C_0,\gamma_0,\gamma_1$ (all real) and $K_{01}$
(complex).

In the second case, when ${\rm e}^{i\gamma}\equiv\pm 1,$ one of the
integrabilty conditions is the vanishing of the scalar invariant $A$ of 
(\ref{ssy1}),
$$A\equiv 0.$$
The rest of the integrabilty conditions are 
\begin{eqnarray}
&&\der T_0=\pm i(\om_1-\bar{\om}_2)+T_{00}\om\nonumber\\
&&\nonumber\\
&&\der B=[-\tfrac{K_0}{2}+iB_1]\om_1+[-\tfrac{K_0}{2}-iB_1]\bar{\om}_1+B_0\om\nonumber\\
&&\label{ssy2p}\\
&&\der C=[B T_0+i(B C+B_0)\pm1]\om_1+\nonumber\\
&&[B T_0-i(B C+B_0)\pm1]\bar{\om}_1+C_0\om\nonumber\\
&&\nonumber\\
&&\der
    K_0=K_{01}\om_1+\bar{K}_{01}\bar{\om}_1+2[\mp B+(1-C)K_0]\om,\nonumber
\end{eqnarray}
with the new higher order scalar invariants $B_0,B_1,C_0$ (all real) and $K_{01}$
 (complex).
\begin{theorem}\label{t0con1}
All locally nonequivalent structures $(M,[\lambda,\mu])$ of oriented
congruences having vanishing twist, nonvanishing shear, with
$T_1\equiv 0$ and $K_1\neq 0$, are described by the invariant forms
$(\omega,\omega_1,\bar{\omega}_1)$ satisfying 
\begin{itemize}
\item either the system
(\ref{ssy1}), (\ref{ssy2}) 
on $M$, in which case $K_1={\rm e}^{i\gamma}\neq\pm1$,
\item or the system (\ref{ssy1}), (\ref{ssy2p}) on $M$, in which case
  $K_1\equiv\pm1$ and $A\equiv 0$.
\end{itemize}
\end{theorem}
As it is readily seen fom the integrabilty conditions (\ref{ssy2}),
(\ref{ssy2p}) neither of these cases admits structures with a strictly
3-dimensional transitive symmetry group (look at the equations for
$\der T_0$ in (\ref{ssy2}),
(\ref{ssy2p}), and observe that
$T_0=const$, which implies $\der T_0=0$, is forbidden!).
\subsection{The case $T_1\neq 0$}
To analyze this case we again start with the basic system
(\ref{systey}) and we assume that $t_1\neq 0$. This assumption enables
us to normalize $T_1$ so that its modulus is equal to one. Thus now we
require
$$|T_1|=1,$$
which uniquely fixes $\rho$ to be $$\rho=|t_1|.$$ After such
normalization all the forms become forms on $M$ and, depending on the
location of $T_1$ on the unit circle, we have to consider two cases:
\begin{itemize}
\item either $T_1={\rm e}^{i\delta}\neq \pm1$,
\item or $T_1=\pm1$.
\end{itemize}
We analyze the $T_1\neq \pm 1$ case first. 
Here we easily reduce the system (\ref{systey}) to the following system
on $M$:
\begin{eqnarray}
&&\der \om=({\rm e}^{i\delta}\om_1+{\rm e}^{-i\delta}\bar{\om}_1)\dz\om,\nonumber\\
&&\der\om_1=(iB -A)\om_1\dz\bar{\om}_1+(1-C+i
  T_0)\om_1\dz\om+\bar{\om}_1\dz\om,\label{ssy45}\\
&&\der\bar{\om}_1=(iB+A)\om_1\dz\bar{\om}_1+\om_1\dz\om+(1-C-i
  T_0)\bar{\om}_1\dz\om.\nonumber
\end{eqnarray}
It has the following integrability conditions:
\begin{eqnarray}
&&\der\delta=[\delta_1+i((B-\delta_1)\cot\delta-A)]\om_1+[\delta_1-i((B-\delta_1)\cot\delta-A)]\bar{\om}_1+\delta_0\om\nonumber\\
&&\label{ssy46}\\
&&\der T_0\dz\om=\nonumber\\
&&\{[B_0+BC-AT_0+2\sin\delta+i(2A-AC-BT_0-A_0+C_1)-\nonumber\\
&&{\rm
      e}^{i\bet}(T_0-iC)]\om_1+\nonumber\\
&&[B_0+BC-AT_0+2\sin\delta-i(2A-AC-BT_0-A_0+\bar{C}_1)-\nonumber\\
&&{\rm e}^{-i\bet}(T_0+iC)]\bar{\om}_1\}\dz\om.\nonumber
\end{eqnarray}
Here, the new scalar invariants are: $T_0,\delta,A,B,C$ (real), and the
higher order scalar invariants are: $\delta_0,\delta_1,B_0$ (real) and
$C_1$ (complex).

In the $T_1\equiv\pm1$ case the equations (\ref{ssy45}) are still
valid, provided that we put $$B\equiv 0.$$ This condition is implied
by $T_1\equiv\pm1$. Thus in this case the invariant forms satisfy
\begin{eqnarray}
&&\der \om=\pm(\om_1+\bar{\om}_1)\dz\om,\nonumber\\
&&\der\om_1=-A\om_1\dz\bar{\om}_1+(1-C+i
  T_0)\om_1\dz\om+\bar{\om}_1\dz\om,\label{ssy47}\\
&&\der\bar{\om}_1=A\om_1\dz\bar{\om}_1+\om_1\dz\om+(1-C-i
  T_0)\bar{\om}_1\dz\om.\nonumber
\end{eqnarray}
The integrability conditions for this system are:
\begin{eqnarray}
&&\der T_0=T_{00}\om+\nonumber\\
&&\big((\mp1-A)T_0+i(2A-AC-A_0+C_1\pm C)\big)\om_1+\label{ssy48}\\
&&\big((\mp1-A)T_0-i(2A-AC-A_0+\bar{C}_1\pm C)\big)\bar{\om}_1,\nonumber
\end{eqnarray}
with the invariant sign equal to $\pm1$, the new scalar invariants being: $T_0,A,C$ (real), and the
higher order scalar invariants being: $B_0$, $T_{00}$ (real) and
$C_1$ (complex).

We summarize with the following theorem.
\begin{theorem}\label{t0con2}
All locally nonequivalent structures $(M,[\lambda,\mu])$ of oriented
congruences having vanishing twist, nonvanishing shear, with
$T_1\neq0$, are described by the invariant forms
$(\omega,\omega_1,\bar{\omega}_1)$ satisfying 
\begin{itemize}
\item either the system
(\ref{ssy45}), (\ref{ssy46}) 
on $M$, in which case $T_1={\rm e}^{i\delta}\neq\pm1$,
\item or the system (\ref{ssy47}), (\ref{ssy48}) on $M$, in which case
  $T_1\equiv\pm1$.
\end{itemize}
\end{theorem}

We pass to the determination of the structures with strictly
3-dimensional transitive group of symmetries.

Using the system (\ref{ssy45}), (\ref{ssy46}) 
we easily establish that in the case $T_1\neq\pm 1$ the
structures are governed by the following system of invariant forms:
\begin{eqnarray}
&&\der \om=({\rm e}^{i\delta}\om_1+{\rm e}^{-i\delta}\bar{\om}_1)\dz\om,\nonumber\\
&&\der\om_1=-\frac{1-C-\cos2\delta}{1-C+\cos 2\delta}{\rm e}^{-i\delta}\om_1\dz\bar{\om}_1+(1-C+i\sin2\delta)\om_1\dz\om+\bar{\om}_1\dz\om,\label{ssy40}\\
&&\der\bar{\om}_1=\frac{1-C-\cos2\delta}{1-C+\cos 2\delta}{\rm e}^{i\delta}\om_1\dz\bar{\om}_1+\om_1\dz\om+(1-C-i\sin2\delta)\bar{\om}_1\dz\om.\nonumber
\end{eqnarray}

In a similar way, if $T_1\equiv \pm1$, using the system (\ref{ssy47}),
(\ref{ssy48}), we see that the structures with 3-dimensional symmetry
groups are governed by the
following system:
\begin{eqnarray}
&&\der \om=\pm(\om_1+\bar{\om}_1)\dz\om,\nonumber\\
&&\der\om_1=\pm\om_1\dz\bar{\om}_1+iT_0\om_1\dz\om+\bar{\om}_1\dz\om,\label{ssy49}\\
&&\der\bar{\om}_1=\mp\om_1\dz\bar{\om}_1+\om_1\dz\om-i
  T_0\bar{\om}_1\dz\om.\nonumber
\end{eqnarray}

\section{Nonvanishing twist and nonvanishing shear}
The Cartan procedure applied to this case is very similar to the one in
Section \ref{vtns} concerned with $a\equiv 0$ and $s\neq
0$. There, before the final reduction to three dimensions, the
procedure stopped at the intermediate 4-dimensional manifold $M\times
\bbR_+$ parametrized by the points of $M$ and the positive coordinate
$\rho$. In the present case, in addition to $s\neq 0$, we also have
$a\neq 0$, which enables us to make an immediate reduction to
three dimensions and thus to produce invariants on $M$. 
Explicitly this reduction is
achieved as follows.

We start with the general system (\ref{str}) of Section
\ref{sec5}. We have $$ a\neq 0,\quad\quad\quad s\neq 0$$  
and we again write the complex shear function $s$ as
$$s=|s|{\rm e}^{i\psi}.$$ Now, for a chosen pair $(\lambda,\mu)$
representing the structure, we impose the conditions 
\begin{eqnarray}
\der\omega\dz\omega&=&i\omega_1\dz\bar{\omega}_1\dz\omega\label{twii}\\
\der\omega_1\dz\omega_1&=&\omega_1\dz\bar{\omega}_1\dz\omega\label{sheaea}
\end{eqnarray}
on the Cartan frame 
$$\om=f\lambda,\quad\quad \om_1=\rho{\rm e}^{i\phi}\mu,\quad\quad \bar{\om}_1=\rho{\rm e}^{-i\phi}\bar{\mu}.$$ 
Note that (\ref{twii}) is possible because of $a\neq 0$ and 
(\ref{sheaea}) is possible because of $s\neq 0$. It is a
matter of straightforward calculation to show that these two
conditions uniquely specify the choice of $f$, $\rho$ and $\phi$. To
write the relevant formulae for $f$, $\rho$ and $\phi$ we
denote the sign of $a$ by ${\rm e}^{i\epsilon \pi}$, where
$\epsilon=0$ or $1$. Then having ${\rm e}^{i\epsilon \pi}={\rm
  sign}(a)$, these formualae are:
$$f={\rm e}^{i\epsilon \pi}|s|,\quad\quad\quad \rho{\rm
  e}^{i\phi}=\sqrt{|a|}\sqrt{|s|}{\rm e}^{-\tfrac{i}{2}(\psi-\epsilon \pi)}$$
and the forms $(\omega,\om_1,\bar{\om}_1)$ satisfy
\beq
\der\om&=&i\om_1\dz\bar{\om}_1+k_1\om_1\dz\om+\bar{k}_1\bar{\om}_1\dz\om\nonumber\\
\der\om_1&=&k_2\om_1\dz\bar{\om}_1+k_3\om_1\dz\om+\bar{\om}_1\dz\om\label{systeas}\\
\der\bar{\om}_1&=&-\bar{k}_2\om_1\dz\bar{\om}_1+\om_1\dz\om+\bar{k}_3\bar{\om}_1\dz\om.\nonumber
\eeq 
Here the complex functions $k_1$, $k_2$, $k_3$ are defined on
  $M$ and:
\begin{eqnarray*}
k_1&=&\frac{(b|s|+|s|_\mu)}{\sqrt{|a|}\sqrt{|s|^3}}e^{\tfrac{i}{2}(\psi-\epsilon\pi)}\\
k_2&=&\frac{-(\log|a|)_{\bar{\mu}}+2 p-(\log|s|)_{\bar{\mu}}+i\psi_{\bar{\mu}}}{2\sqrt{|a|}\sqrt{|s|}}e^{-\tfrac{i}{2}(\psi-\epsilon\pi)}\\
k_3&=&\frac{ib_{\bar{\mu}} - i\bar{b}_\mu -ibp +i\bar{b}\bar{p} + 
      {\rm e}^{-i\epsilon\pi}|a|(q - \bar{q}-(\log|s|)_\lambda + i\psi_\lambda)}{2|a||s|}
\end{eqnarray*}
These functions constitute the full system of invariants of
$(M,[\lambda,\mu])$ for $a\neq 0$, $s\neq 0$.
\begin{theorem}
A given structure $(M,[\lambda,\mu])$ of an oriented congruence with
nonvanishing twist, $a\neq 0$, and nonvanishing shear, $s\neq 0$,
uniquely defines the frame of invariant 1-forms $\om,\om_1,\bar{\om}_1$ and
invariant complex functions $k_1,k_2,k_3$ on $M$. The forms and the
functions are determined by the requirement that they satisfy the 
system (\ref{systeas}). Starting with an arbitrary
representative $(\lambda,\mu)$ of the structure $[\lambda,\mu]$, the
forms are given by 
$$\om={\rm e}^{i\epsilon \pi}|s|\lambda,\quad\quad\om_1=\sqrt{|a|}\sqrt{|s|}{\rm e}^{-\tfrac{i}{2}(\psi-\epsilon
    \pi)}\mu,\quad\quad \bar{\om}_1=\sqrt{|a|}\sqrt{|s|}{\rm
    e}^{\tfrac{i}{2}(\psi-\epsilon \pi)}\bar{\mu},$$ where the shear
  function is $s=|s|{\rm
  e}^{i\psi}$. Here ${\rm e}^{i\epsilon\pi}$, $\epsilon=0,1$, denotes
    the sign of the twist function $a$. The system (\ref{systeas}) 
encodes all the invariant information of the structure
    $(M,[\lambda,\mu])$.
\end{theorem}

We pass to the determination of all homogeneous examples with
$a\neq 0$, $s\neq 0$. Now the maximal dimension of a group
of transitive symmetries is three. The structures with 3-dimensional
groups of symmetries correspond to those satisfying system
(\ref{systeas}) with all the functions $k_1,k_2,k_3$ being
constants. Applying the exterior differential to the system
(\ref{systeas}) with $k_1,k_2,k_3$ constants we arrive at the
following Theorem.

\begin{theorem}
All homogeneous structures $(M,[\lambda,\mu])$ with nonvanishing
twist, $a\neq 0$, and nonvanishing shear, $s\neq 0$, have a 
strictly 3-dimensional symmetry group and fall into four main types
characterized by:
\begin{itemize}
\item[\emph{I:}] $k_3=1$. In this case there is a 2-real parameter
  family of nonequivalent structures distinguished by real constants
  $x$ and $y$ related to the invariants $k_1$ and $k_2$ via: 
$$k_1=x,\quad\quad k_2=iy.$$ 
\item[\emph{II:}] $k_3={\rm e}^{i\phi}$, $0<\phi<2\pi$. In
  this case there is a 2-real parameter
  family of nonequivalent structures distinguished by real constants
  $x,y$ which together with the parameter $\phi$ are constrained by the equation
  $$\cos\phi(1-2xy+\cos\phi)=0.$$ The invariants $k_1,k_2,k_3$ are then
  given by $$k_1=x(\cot\tfrac{\phi}{2}+i),\quad\quad k_2=-iy(\cot\tfrac{\phi}{2}+i),\quad\quad k_3=\cos\phi+i\sin\phi.$$  
\item[\emph{III:}] $k_3+\bar{k}_3=0$, $k_3\neq\pm i$.  In this case there is a 3-real parameter
  family of nonequivalent structures distinguished by real constants
  $y'\neq\pm 1$, $x$, $y$  related to the invariants $k_1$, $k_2$, $k_3$ via:
  $$k_1=x+iy,\quad\quad k_2=\bar{k}_1=x-iy,\quad\quad k_3=iy'.$$ 
\item[\emph{IV:}] $|k_3|\neq 1$, $k_3+\bar{k}_3\neq 0$. In this case there is a 3-real parameter
  family of nonequivalent structures distinguished by real constants
  $x'\neq 0$, $y'$, $x$, $y$ constrained by the equation
$${x'}^2+{y'}^2+2y'(x^2+y^2)-4xy=1.$$
The invariants $k_1,k_2,k_3$ are then
  given by 
$$
k_1=x+iy,\quad\quad k_3=x'+iy',\quad\quad
k_2=\frac{\bar{k}_1(1+k_3^2)-k_1(k_3+\bar{k}_3)}{1-|k_3|^2}.
$$
\end{itemize}  
\end{theorem} 

Among all the structures covered by the above theorem, the simplest have $k_1=k_2=k_3\equiv 0$. This unique structure belongs to the case III above and is the {\it flat case} for the branch $a\neq 0$, $s\neq 0$. We describe it in the following proposition.

\begin{proposition}
A structure of an oriented congruence $(M,[\lambda,\mu])$ with nonvanishing twist, $a\neq 0$, nonvanishing shear $s\neq 0$ and having  $k_1=k_2=k_3\equiv 0$, may be locally represented by forms 
\be
\la=\der u+\frac{\sqrt{2}{\rm e}^{iu}-i\bar{z}}{z\bar{z}-1}\der z+\frac{\sqrt{2}{\rm e}^{-iu}+iz}{z\bar{z}-1}\der \bar{z},\quad\quad\mu=\frac{2{\rm e}^{i u}}{z\bar{z}-1}\der{z}-\sqrt{2}\lambda,
\ee
where $(u,z,\bar{z})$ are coordinates on $M$. This structure has the local symmetry group of Bianchi type VIII, locally isomorphic to the group $\slg(2,\bbR)$. 
\end{proposition}

\begin{remark}
There are more structures with $a\neq 0$, $s\neq 0$, which have a 3-dimensional transitive symmetry group of Bianchi type VIII. It is quite complicated to write them all here. For example, among them, there is a 1-parameter family of nonequivalent structures with $k_1=k_2\equiv 0$. They may be represented by 
\be
\la=\der u+\frac{\kappa{\rm e}^{iu}-i\bar{z}}{z\bar{z}-1}\der z+\frac{\kappa{\rm e}^{-iu}+iz}{z\bar{z}-1}\der \bar{z},\quad\quad\mu=(\kappa^2-1)\frac{2{\rm e}^{i u}}{z\bar{z}-1}\der{z}-\kappa\lambda,\label{flaka1}
\ee
where $\kappa>0,\kappa\neq 1$. The only nonvanishing invariant for this 1-parameter family is $k_3=-i(1-\tfrac{2}{\kappa^2})$. It may be considered as a deformation of the flat case above, which corresponds to $\kappa=\sqrt{2}$.
\end{remark}
\begin{remark}
In a similar way, among all the structures with $a\neq 0$, $s\neq 0$, which have a 3-dimensional transitive symmetry group of Bianchi type IX, we may easily characterize those with $k_1=k_2\equiv 0$. They may be represented by
\be
\la=\der u+\frac{\kappa{\rm e}^{iu}-i\bar{z}}{z\bar{z}+1}\der z+\frac{\kappa{\rm e}^{-iu}+iz}{z\bar{z}+1}\der \bar{z},\quad\quad\mu=(\kappa^2+1)\frac{2{\rm e}^{i u}}{z\bar{z}+1}\der{z}-\kappa\lambda,\label{flaka2}
\ee
where $\kappa>0$. Here the only nonvanishing invariant is $k_3=-i(1+\tfrac{2}{\kappa^2})$.
\end{remark}
\begin{remark}
It is interesting to remark which of the structures (\ref{flaka1}), (\ref{flaka2}) correspond to the flat CR-structure in the sense of Cartan. According to \cite{tafel}, they correspond to $\kappa=0,\sqrt{2}$ in the (\ref{flaka1}) case, and $\kappa=0$ in the (\ref{flaka2}) case. Thus in these cases the corresponding structures of an oriented congruence are locally CR-equivalent to the hyperquadric 
CR structure of Example \ref{hypc}, with a nonstandard splitting, which causes the shear $s\neq 0$.
\end{remark}

It is a rather complicated matter to describe which Bianchi types having a 
3-dimensional transitive symmetry group correspond to a given homogeneous structure with $a\neq 0$, $s\neq 0$. We remark that the groups of 
Bianchi types I and V {\it are excluded} for such structures. We also 
fully describe the situation for Bianchi types II and IV. This is summarized in the following theorem.

\begin{theorem}~\\
There are only two nonequivalent structures of an oriented congruence $(M,[\lambda,\mu])$ with $a\neq 0$, $s\neq 0$, which have a local transitive symmetry group of Bianchi type II. They may be locally represented by 
$$\lambda=\der u+\tfrac{i}{2}(z\der\bar{z}-\bar{z}\der z),\quad\quad \mu=\der z\pm \sqrt{2}(1-i)\lambda,$$
where $(u,z,\bar{z})$ are coordinates on $M$. The constant invariants are
$$k_1=\pm\frac{1-i}{\sqrt{2}},\quad\quad k_2=\pm\frac{1+i}{\sqrt{2}},\quad\quad k_3=-i,$$
and the sign $\pm 1$ distinguishes between the nonequivalent structures.\\

There are also only two 2-parameter families of nonequivalent structures of an oriented congruence $(M,[\lambda,\mu])$ with $a\neq 0$, $s\neq 0$, which have a local transitive symmetry group of Bianchi type IV. They may be 
locally represented by 
$$\lambda=y^{-1}(\der u-\log y \der x),\quad\quad \mu=y^{-1}\der (x+i y)\pm \sqrt{2}(1-i)w\lambda,$$
where $(u,x,y)$ are coordinates on $M$ and $w={\rm Re}(w)+i{\rm Im}(w)\neq 0$ is a complex parameter. The constant invariants are
$$k_1=\pm\frac{1-i}{\sqrt{2}}+\frac{i}{2\bar{w}},\quad\quad k_2=\pm\frac{1+i}{\sqrt{2}}+\frac{i}{2\bar{w}},\quad\quad k_3=-i\pm(\frac{1+i}{\bar{w}}+\frac{1-i}{w}),$$
and the two real parameters ${\rm Re}(w)$ and ${\rm Im}(w)$, together with the sign 
 $\pm 1$ distinguish between the nonequivalent structures.
\end{theorem} 
\begin{remark}
We remark that the structures with a symmetry group of Bianchi type II are in a sense the limiting case of the two families of structures with Bianchi type IV. They correspond to the limit $|w|\to\infty$. 
\end{remark}

%\begin{remark} It is interesting to note that the moduli space of 
%nondegenerate 3-dimensional CR-structures with preferred
%  splitting $(M,[\lambda,\mu])$, described by IV above, 
%is itself a 3-dimensional nodegenerate CR-manifold. If we consider
%  $\bbC^2$ with holomorphic
%  coordinates $k_1$ and $k_3$ we see that this CR-manifold is  
%  embedded as a hypersurface (\ref{hy}) in $\bbC^2$ minus the
%  intersection of the $k_3+\bar{k}_3=0 .  
%\end{remark}

\section{Application 1: Lorentzian metrics in four
  dimensions}\label{cab}
In this section we use our results about oriented congruence
structures to construct Lorentzian metrics in 4-dimensions.

\subsection{Vanishing twist --
  nonvanishing shear case and $pp$-waves}
Since our oriented congruence structures are 3-dimensional objects, we
concentrate only on those structures, which in some \emph{natural} manner 
define an associated 4-dimensional manifold. As we noted in the
sections devoted to the Cartan analysis of the oriented 
congruence structures, in
some cases, such as those described in Section
\ref{vtns}, the Cartan bundle $P$ encoding the basic
invariants of the structures is 4-dimensional. So in this case, 
i.e. when the twist $a\equiv 0$ and the shear $s\neq 0$, we have a
4-dimensional manifold naturally associated with the oriented
congruence structure. Moreover, in such case the Cartan procedure 
provides us also with a rigid coframe of invariant forms
$(\om_1,\bar{\om}_1,\om,\Om)$ on $P$. Using these forms we may define 
\be
g=2\om_1\bar{\om}_1+2\om\Om,\label{tem}\ee
or, as suggested by the form of the associated Cartan connection,
\be
g=2\om_1\bar{\om}_1+2\om(\Om-\om)\label{temp}.\ee
These both are well defined Lorentzian
metrics on $P$, which are built only from the objects naturally and invariantly
associated with the oriented congruence structure.

To be more specific, let us consider the structures with the curvature
of the Cartan connection $R\in\mathfrak{h}_1$, as described in Theorem
\ref{alu}. In this case the bundle $P$ is parametrized by $(z,\bar{z},u,r)$
and the invariant forms are:
\begin{eqnarray*}
&&\Om=\der r,\quad\quad\om=\der u\\
&&\om_1={\rm e}^r \Big(h\der
z-(\bar{h}'+\bar{h}-i\al\bar{h})\der\bar{z}\Big)\\
&&\bar{\om}_1={\rm e}^r \Big(\bar{h}\der
\bar{z}-(h'+h+i\al h)\der
z\Big),
\end{eqnarray*}
with functions $\al=\al(u)$ (real) and $h=h(u)$ (complex) satisfying
the ordinary differential equation \ref{alul}. Inserting these forms in
the formulae (\ref{tem})-(\ref{temp}), we get the respective 
4-dimensional Lorentzian metrics
$$g_0=2{\rm e}^{2r} \Big(h\der
z-(\bar{h}'+\bar{h}-i\al\bar{h})\der\bar{z}\Big)\Big(\bar{h}\der
\bar{z}-(h'+h+i\al h)\der
z\Big)+2\der u\der r,$$
and 
$$g_{-1}=2{\rm e}^{2r} \Big(h\der
z-(\bar{h}'+\bar{h}-i\al\bar{h})\der\bar{z}\Big)\Big(\bar{h}\der
\bar{z}-(h'+h+i\al h)\der
z\Big)+2\der u(\der r-\der u).$$
It turns out that both these metrics have quite nice
properties.

Actually, introducing a still bigger class of metrics 
$$g_{c}=2{\rm e}^{2r} \Big(h\der
z-(\bar{h}'+\bar{h}-i\al\bar{h})\der\bar{z}\Big)\Big(\bar{h}\der
\bar{z}-(h'+h+i\al h)\der
z\Big)+2\der u(\der r-c\der u),$$
with $c={\rm const}\in \bbR$, one checks that they all are of type N in the Petrov
classification of 4-dimensional Lorentzian metrics. This means that
their Weyl tensor is expressed in terms of only one nonvanishing
complex function $\Psi_4$, called the Weyl spin coefficient, which
reads 
$$\Psi_4=2(i \al-c-1).$$
All the other Weyl coefficients $(\Psi_0,\Psi_1,\Psi_2,\Psi_3)$, which
together with $\Psi_4$ totally encode the Weyl tensor of $g_c$, are
identically zero.

Looking at the spin coefficient $\Psi_4$ we see that there is a
distinguished metric in the class $g_c$. This corresponds to
$c=-1$. In such case the Weyl tensor of $g$ is just proportional to
$\Psi_4=2i\al$ and we have a Lorentz-geometric interpretation of the
invariant $\al=\al(u)$ of the corresponding structure of the oriented
congruence. Confronting these considerations with the results of
Section \ref{zaza} we get the following 
\begin{theorem}
Every structure of an oriented congruence $(M,\la,\mu)$ with
vanishing twist, $a\equiv 0$, nonvanishing shear $s\neq 0$, and having the
curvature $R$ of its corresponding Cartan connection in
$\mathfrak{h}_1$, 
defines a Lorentzian metric 
$$g_{-1}=2\om_1\bar{\om}_1+2\om(\Om-\om),$$
which is of Petrov type N or conformally flat. The nonequivalent
metrics correspond to different structures of the oriented congruence,
and the metric is conformally flat if and only if $R\equiv 0$.
\end{theorem}

Interestingly metrics $g_{-1}$ are conformal to Ricci flat
metrics. The Ricci flat metric in the conformal class of $g_{-1}$ is
given by
$$\hat{g}_{-1}=\frac{2{\rm e}^{4u}}{(t+{\rm e}^{2u})^2}\Big(\big(h\der
z-(\bar{h}'+\bar{h}-i\al\bar{h})\der\bar{z}\big)\big(\bar{h}\der
\bar{z}-(h'+h+i\al h)\der
z\big)+{\rm e}^{-2r} \der u(\der r-\der u)\Big),$$
where $t$ is a real constant. For each $\al=\al(u)$ and for each
solution $h=h(u)$ of (\ref{alul}), the corresponding Ricci flat metric
is the so called \emph{linearly polarized $pp$-wave} from General
Relativity Theory (see \cite{exact}, p. 385).

\subsection{Nonvanishing twist -- vanishing shear case and the Bach metrics}\label{bbaa}
Another example of 4-dimensional Lorentzian manifolds naturally
associated with the structures of oriented congruences appears in the
nonvanishing twist -- vanishing shear case, as we explained in Section
\ref{cconse1}. Actually in Section \ref{cconse1} we defined
\emph{conformal} Lorentzian 4-manifolds equipped with the
\emph{conformal} class of Lorentzian metrics $[g_t]$, which are
naturally associated with a congruence structure with twist and
without shear. Here we study the conformal properties of these metrics.

\subsubsection{The Cotton and Bach conditions for conformal metrics}
We recall \cite{gover}
that a Lorentzian metric $g$ on a manifold $M$ is called
{\it conformal to Einstein} iff there exists a real function
$\Upsilon$ on $M$ such that the rescaled metric $\hat{g}={\rm
  e}^{2\Upsilon}g$ satisfies the Einstein equations
$Ric(\hat{g})=\Lambda\hat{g}$. In the case of an oriented $M$ with ${\rm dim} M=4$ there are two {\it necessary} conditions
\cite{Mas,CTPa} for $g$ to be conformal to Einstein (in
algebraically generic cases \cite{gover} these necessary conditions
are sufficient). To describe these conditions we denote by $F$ the
curvature 2-form of the Cartan normal conformal connection
$\omega_{[g]}$ associated with a conformal class $[g]$ (see
\cite{Kobayashi} for definitions). The curvature $F$ is horizontal.
Thus, choosing a representative $g$ of the conformal class $[g]$, we
can calculate its Hodge dual $*F$ and calculate the $6\times 6$
matrix of 3-forms \be D*F=\der*F+\omega_{[g]}\dz*F-*F\dz\omega_{[g]}
\ee for the connection $\omega_{[g]}$. This matrix has a remarkably
simple form
$$
D*F=\begin{pmatrix}
0&*j^\mu&0\\
0&0&*j_\mu\\
0&0&0
\end{pmatrix},
$$
where $*j^\mu$ is a vector-valued 3-form, the Hodge dual of the
so called {\it Yang-Mills current} $j^\mu$ for the conformal
connection $\omega_{[g]}$. Having said this we introduce the
vacuum Yang-Mills equation for the conformal connection $\omega_{[g]}$
\be
D*F=0\label{ym}
\ee
i.e. the condition that the Yang-Mills current $j^\mu$ vanishes.
It turns out that in ${\rm dim} M=4$
equations (\ref{ym}) are {\it conformally invariant}. They are equivalent to
the requirement that the {\it Bach tensor} of $g$ identically
vanishes \cite{Mas,gover}. This condition is known \cite{KNT} to constitute
a first system of
equations which a 4-dimensional metric $g$ must satisfy to be
conformal to Einstein.

Another independent condition can be obtained by decomposing $F$
into $F=F^+\oplus F^-$, where $*F^\pm=\pm i F^\pm$ are its selfdual
and antiselfdual parts (note that $i$ appears here as a consequence
of the assumed Lorentzian signature). Decomposing the curvatures
$F^\pm$ onto a basis of 2-forms $\{\theta^i\dz\theta^j\}$ associated
with a coframe $\{ \theta^i \}$ in which $g$ takes the form
  $g=g_{ij}\theta^i\theta^j$, we recall that the second necessary
  condition for a 4-metric $g$ to be conformal to Einstein is
\be
[F^+_{ij},F^-_{kl}]=0\quad\quad\forall i,j,k,l=1,2,3,4.\label{cotton}
\ee
Here $[,]$ is the commutator of the $6\times 6$ matrices $F^+_{ij}$ and
$F^-_{kl}$. We term (\ref{ym}) the {\it Bach condition} and (\ref{cotton})
the {\it Cotton condition} \cite{gover}.
\subsubsection{Conformal curvature of the associated metrics}
%/home/pawel/notebooks/denny/noshearsystemeinsteinfull.nb
%/noshearsystemeinsteinfull2008.nb
Now we calculate the Cartan normal conformal connection and its
curvature for the conformal metrics (\ref{feog}). We recall the
setting from Sections \ref{cconse}, \ref{cconse1}. The structure of an
oriented
congruence $(M,\la,\mu)$ with vanishing shear and nonvanishing twist 
defines a 5-dimensional principal fiber bundle $H_2\to P\to M$, on
which the invariant forms $(\om_1,\bar{\om}_1,\om,\Om,\bar{\Om})$, 
satisfying the system (\ref{syste}) reside. There is another fiber bundle
associated with such a situation. This is the bundle $P\to N$ with a 
4-dimensional base $N$ and with 1-dimensional fibers. The manifold $N$ is in addition fibered over $M$ also with 1-dimensional fibers. The
forms
$$\{\theta^1,\theta^2,\theta^3,\theta^4\}=\{\om_1,\bar{\om}_1,\om,ti(\bar{\Om}-\Om)\}$$
on $P$ are used to define a bilinear form
$G_t=2(\theta^1\theta^2+\theta^3\theta^4)$ on $P$. Although this is degenerate on $P$,
it projects to a well defined conformal class $[g_t]$ of
\emph{Lorentzian} metrics  
\be
g_t=2(\theta^1\theta^2+\theta^3\theta^4)\label{emu}\ee
on $N$, see (\ref{feog}).

One can try to calculate the Cartan normal conformal connection for
the metrics $g_t$ on $N$ itself, but we prefer to do this on the 5-dimensional bundle
$P$ instead. This is more convenient, since in
such an approach we can directly use the coframe derivatives
(\ref{syste}) of the forms $(\om_1,\bar{\om}_1,\om,\Om,\bar{\Om})$ on $P$,
without the neccessity of projecting them from $P$ to $N$.

Thus, in the following, we associate the
dual set of vector fields $(E_1,\bar{E}_1,E_0,E_2,\bar{E}_2)$ to
$(\om_1,\bar{\om}_1,\om,\Om,\bar{\Om})$, and we will
use them to denote the derivatives of the functions, such as the
invariants $K_1$, $K_2$ and $\bar{K}_2$. The conventions will be as
follows: the symbols $K_{11}=E_1(K_1)$ and $K_{1\bar{1}}=\bar{E}_1(K_1)$
will denote the directional
derivatives of $K_1$ in the respective directions of the vector fields $E_1$ and
$\bar{E}_1$. In particular $K_{2\bar{1}0}$ will denote 
$E_0(\bar{E}_1(K_2))$.

A (rather tedious) calculation gives the following expressions for the
Cartan normal conformal connection $\om_t$ for the metrics $g_t$ on $P$:
\be
\om_t=\begin{pmatrix}
\tfrac12(\Om+\bar{\Om})&\tau^1&\tau^2&\tau^3&\tau^4&0\\
&&&&&\\
\theta^1&-i\Omega_1&0&-\Omega_2&\tfrac{i}{2}\theta^1&\tau^2\\
&&&&&\\
\theta^2&0&i\Omega_1&-\bar{\Omega}_2&-\tfrac{i}{2}\theta^2&\tau^1\\
&&&&&\\
\theta^3&\tfrac{i}{2}\theta^2&-\tfrac{i}{2}\theta^1&-\tfrac{1}{2}(\Om+\bar{\Om})&0&\tau^4\\
&&&&&\\
\theta^4&\bar{\Omega}_2&\Omega_2&0&\tfrac{1}{2}(\Om+\bar{\Om})&\tau^3\\
&&&&&\\
0&\theta^2&\theta^1&\theta^4&\theta^3&-\tfrac{1}{2}(\Om+\bar{\Om})
\end{pmatrix}.
\ee
Here the 1-forms $\Omega_1$ (real) and $\Omega_2$ (complex) are
$$
\Omega_1=t K_1\theta^3+\tfrac{1-t}{2t}\theta^4,\quad
\Omega_2=itK_1\theta^1+it\bar{K}_2\theta^3,\quad\bar{\Omega}_2=-itK_1\theta^2-itK_2\theta^3
$$
and the 1-forms $\{\tau^1,\tau^2,\tau^3,\tau^4\}$ are:
\beq
\tau^1&=&-\tfrac{1}{6}(5t-2)K_1\theta^2 +
    \tfrac{1}{4}(2it K_{11} + K_2(1 - t))\theta^3\nonumber\\
\tau^2&=&\bar{\tau}^1=-\tfrac{1}{6}(5t-2)K_1\theta^1 +
    \tfrac{1}{4}(-2it K_{1\bar{1}} + \bar{K}_2(1 - t))\theta^3\nonumber\\
\tau^3&=&\tfrac{1}{4}(2it K_{11} - K_2(t+1))\theta^1 -
    \tfrac{1}{4}(2it \bar{K}_{1\bar{1}} +\bar{K}_2(t+1))\theta^2 - t^2K_1^2\theta^3 +\tfrac{1}{6}(4t-1)K_1
\theta^4\nonumber\\
\tau^4&=&\tfrac{1}{6}(4t-1)K_1\theta^3 - \tfrac{1}{4}\theta^4.\nonumber
\eeq

The next step, namely the calculation of the curvature
$F_t=\der\om_t+\om_t\dz\om_t$ of $\om_t$, is really tedious, but 
achievable with the help of symbolic calculation programs such as,
e.g. Mathematica. The resulting formulae are too complicated to display 
here, but the $\soa(1,3)$-part of the curvature, which is just the
Weyl tensor of $g_t$, is worth quoting. We present it in terms of the
(lifted to $P$) Weyl spinors $\Psi_0$, $\Psi_1$, $\Psi_2$, $\Psi_3$ and
$\Psi_4$. These read: 
%/noshearsystemeinsteinfull2008.nb
%/notebooks/denny/bestbachexamplesuperbetter.nb
\begin{eqnarray}
&&\Psi_0=0,\quad\quad\quad\quad\quad\Psi_1=0,\nonumber\\
&&\Psi_2=\tfrac16(1-4t) K_1,\nonumber\\
&&\Psi_3=\tfrac14\big(2itK_{1\bar{1}}+(3t-1)\bar{K}_2\big),\label{weylt}\\
&&\Psi_4=-i t\bar{K}_{2\bar{1}}.\nonumber
\end{eqnarray}

We have the following 
\begin{proposition}\label{haha}
Every metric $g_t$ with $K_1\equiv 0$ or $t=\tfrac14$ is of Petrov type $III$ or its specializations. If $t=\tfrac13$ and $K_1\equiv 0$,
then the conformal class $[g_{1/3}]$ of the metric
$g_{1/3}$ is of Petrov type $N$.
\end{proposition}

Calculation of the Yang-Mills current $j=j_\mu\theta^\mu$ for $\omega_t$ is also
possible. Since the covariant derivative
of the Hodge dual of the curvature $F_t$ is 
\emph{horizontal} with repect to the bundle $P\to N$, 
the current components $j_\mu$, as viewed on $P$ or on $N$, differ only by
nonvanishing scales. The result of our calculation on $P$ reads:
\begin{eqnarray*}
j^1&=&\bar{j}^2=\tfrac13(1-4t)[K_{111}\theta^1-2iK_{11}\theta^4]+\tfrac16
j^1_2\theta^2-\tfrac16 j^1_3\theta^3\\
j^3&=&-\tfrac16 j^1_3\theta^1-\tfrac16\bar{j}^1_3\theta^2-\tfrac16
j^3_3\theta^3-\tfrac16 j^1_2\theta^4\\
j^4&=&\tfrac23
(4t-1)[K_1\theta^4+iK_{11}\theta^1-iK_{1\bar{1}}\theta^2]-\tfrac16 j^1_2\theta^3,
\end{eqnarray*}
where
\begin{eqnarray*}
j^1_2&=&(1-4t)(1-12t)K_1^2+(7t-1)(K_{11\bar{1}}+K_{1\bar{1}1})\\
j^1_3&=&16it(4t-1)K_1K_{11}-2(1-2t)(1-4t)K_1K_2+(1-4t)K_{2\bar{1}1}+\\
&&3it(K_{11\bar{1}1}+K_{1\bar{1}11})\\
j^3_3&=&16t^2(1-4t)K_1^3-36t^2K_{11}K_{1\bar{1}}+3(1-t)(1+3t)|K_2|^2+
2(t+2)K_{2\bar{1}3}-\\
&&24t^2K_1(K_{11\bar{1}}+K_{1\bar{1}1})+
2it(4-7t)(K_{1\bar{1}}K_2-K_{11}\bar{K}_2).
\end{eqnarray*}

We have also calculated the Cotton matrices
$[F^+_{tij},F^-_{tij}]$ for each value of the real parameter $t$. We
obtained formulae which are too complicated to write here. However
we observed, that among all the parameter values for $t$, there are
a few preferred ones for which the formulae simplify significantly.
These special parameter values are:
$$t=\pm\frac13,\quad\quad
t=\frac14,\quad\quad t=1.$$ Here we focus on $t=-\tfrac13$ and
$t=1$,  for which we have the following theorem.
\begin{theorem} \label{bact} If $t=-\frac13$ or $t=1$ and the relative invariant
$K_1\equiv 0$, then the
  conformal metrics $[g_t]$ satisfy the Bach condition. If in
addition the relative invariant $K_2\neq 0$, the metrics are \emph{not} conformally flat and do
  \emph{not} satisfy the Cotton condition. If $K_1\equiv K_2\equiv 0$ the
  conformal metrics $g_{-1/3}$ and $g_1$ have $F_t\equiv 0$, i.e. they are conformally flat.
\end{theorem}
The theorem can be verified by using the explicit formulae for the
Yang-Mills current $j^\mu$, the matrices $[F^+_{tij},F^-_{tij}]$,
and the integrability conditions for the system (\ref{syste}) with
$K_1=0$. These integrability conditions, in particular, imply that
$K_{2\bar{1}}=0$.

We shall return to the other two interesting values $t=1/4$ and
$t=1/3$ for $g_t$ below, where we consider examples.
\subsubsection{Examples}
As noted above a particularly interesting class of structures
$(M,\la,\mu)$ corresponds to $K_1\equiv 0$ and $K_2\neq 0$. Looking at
the list of our examples presented in Section \ref{ntvs} we find such
a structure in Section \ref{k2niez}. This corresponds to a special
value of the parameter $\beta_K=-3^{\frac13}$ in the family of
structures described by the invariant system (\ref{b89}), and is locally represented by forms
  $\la$, $\mu$ as in (\ref{cast}) with $\beta_K=-3^{\frac13}$. Actually
it is worthwhile to write the metrics $g_t$ for \emph{all} the
structures covered by (\ref{cast}). These metrics read:
%%../denny/bestbachexamplegeneraldalejeiukoniec.nb 
\begin{eqnarray}
&&g_t=g_t(\beta)=2\der z\der\bar{z}+\nonumber\\
&&t\Big(\der u +\frac{2\beta{\rm
    e}^{-i\beta u}+i\bar{z}}{\beta(z\bar{z}-2\beta^2(2+\beta^3))}\der
z+\frac{2\beta{\rm
    e}^{i\beta u}-iz}{\beta(z\bar{z}-2\beta^2(2+\beta^3))}\der
\bar{z}\Big)\times\nonumber\\
&&\frac{(z\bar{z}-2\beta^2(2+\beta^3))^2}{2\beta^4}\Big(2\der r+\frac{2(\beta{\rm
    e}^{-i\beta u}-i\bar{z})}{z\bar{z}-2\beta^2(2+\beta^3)}\der
z+\frac{2(\beta{\rm
    e}^{i\beta u}+iz)}{z\bar{z}-2\beta^2(2+\beta^3)}\der\bar{z}\Big),\nonumber
\end{eqnarray}
and in addition to the real parameter $t$, they are parametrized by the real parameter 
$\beta\neq 0$ which enumerates nonequivalent structures
$(M,\la,\mu)$.

These are quite interesting conformal Lorentzian metrics for the
following reasons.

First, if $$\beta=\beta_K=-3^{\tfrac13},$$ we have $K_1\equiv 0$, and according 
to Theorem \ref{bact}, the metrics 
\begin{eqnarray}
&&g_{-1/3}(-3^{\frac13})=2\der z\der\bar{z}-\nonumber\\
&&\Big(\der u +\frac{2\cdot3^{\frac13}{\rm
    e}^{3^{\frac13} iu}-i\bar{z}}{3^{\frac13}(z\bar{z}+2\cdot 3^{\frac23})}\der
z+\frac{2\cdot 3^{\frac13}{\rm
    e}^{-3^{\frac13} iu}+iz}{3^{\frac13}(z\bar{z}+2\cdot 3^{\frac23})}\der
\bar{z}\Big)\times\nonumber\\
&&\frac{(z\bar{z}+2\cdot 3^{\frac23})^2}{18\cdot 3^{\frac13}}\Big(2\der r-\frac{2(3^{\frac13}{\rm
    e}^{3^{\frac13} iu}+i\bar{z})}{z\bar{z}+2\cdot 3^{\frac23}}\der
z-\frac{2(3^{\frac13}{\rm
    e}^{-3^{\frac13} iu}-iz)}{z\bar{z}+2\cdot 3^{\frac23}}\der\bar{z}\Big),\nonumber
\end{eqnarray}
and 
\begin{eqnarray}
&&g_{1}(-3^{\frac13})=2\der z\der\bar{z}+\nonumber\\
&&\Big(\der u +\frac{2\cdot3^{\frac13}{\rm
    e}^{3^{\frac13} iu}-i\bar{z}}{3^{\frac13}(z\bar{z}+2\cdot 3^{\frac23})}\der
z+\frac{2\cdot 3^{\frac13}{\rm
    e}^{-3^{\frac13} iu}+iz}{3^{\frac13}(z\bar{z}+2\cdot 3^{\frac23})}\der
\bar{z}\Big)\times\nonumber\\
&&\frac{(z\bar{z}+2\cdot 3^{\frac23})^2}{6\cdot 3^{\frac13}}\Big(2\der r-\frac{2(3^{\frac13}{\rm
    e}^{3^{\frac13} iu}+i\bar{z})}{z\bar{z}+2\cdot 3^{\frac23}}\der
z-\frac{2(3^{\frac13}{\rm
    e}^{-3^{\frac13} iu}-iz)}{z\bar{z}+2\cdot 3^{\frac23}}\der\bar{z}\Big),\nonumber
\end{eqnarray}
are \emph{Bach flat}. Since the invariant $K_2$
of the corresponding structures $(M,\la,\mu)$ is nonvanishing, they are also
\emph{not} conformal to any Einstein metric. Note that, again because of $K_1\equiv 0$ and $K_2\neq 0$, both metrics $g_{1}(-3^{\tfrac13})$ and
$g_{-1/3}(-3^{\tfrac13})$ are of general Petrov type $III$ (see Proposition \ref{haha}). As far as we
know, they both provide the first \emph{explicit} examples of
conformally non Einstein Bach metrics which are of this Petrov type
(compare e.g. with \cite{nurple}).

Second, note also that, since $K_1\equiv 0$
for $\beta_K=-3^{\tfrac13}$, the metric $g_{1/3}(\beta_K)$, with now $t=+1/3$, is also
quite interesting. According to Proposition \ref{haha} this metric is
of Petrov type N. In gravitation theory it would be also termed
\emph{twisting} type N (see \cite{exact}). It is not conformal to any 
Einstein metric, since for all metrics $g_t(\beta_K)$ the Bach tensor $B_t(\beta_k)$,
when expressed in terms of the coframe
$(\theta^1,\theta^2,\theta^3,\theta^4)$, reads
$$B_{t}(-3^{\frac13})=2^5\cdot 3^4~\frac{(t-1)(1+3t)}{(z\bar{z}+2\cdot
  3^{\frac23})^6}~\theta^3\odot\theta^3.$$
This obviously does not vanish, when $t=1/3$, hence the metrics
$g_{1/3}(\beta_K)$ are examples of twisting
  type N metrics, which are not conformally Einstein.

Third, suggested by the structure of the Weyl tensor (\ref{weylt}) 
for all the metrics $g_t$ we specialize the metrics $g_t(\beta)$ to the
case when $t=\tfrac14$. The Yang-Mills current for  this
special case may be read off from the general formulae from the
previous section. Here however we prefer to give the explicit formulae for the
Bach tensor for $g_{1/4}(\beta)$. Here again the Bach tensor $B_{1/4}(\beta)$ for
these metrics has a very simple form
$$B_{1/4}(\beta)=6~\frac{\beta^6(\beta^6+36\beta^3+36)}{(z\bar{z}-2\beta^2(2+\beta^3))^{6}}~\theta^3\odot\theta^3.$$ 
As is readily seen this vanishes for the following two real 
values of $\beta$:  
$$\beta_{S_1}=-\big(6(3+2\sqrt{2})\big)^{\tfrac13},\quad\quad
\beta_{S_2}=-\big(6(3-2\sqrt{2}\big)^{\tfrac13}.$$
Thus the two corresponding metrics
$g_{1/4}(\beta_{S_1})$, and
$g_{1/4}(\beta_{S_2})$ are further examples of
Bach Lorentzian metrics, which are again of Petrov type III. One can
check by direct calculation that they are also not conformal to any
Einstein metric.

Motivated by this last example we calculated the Bach tensor for
\emph{all} the metrics $g_{1/4}$ (not neccessarily those associated with the
$\beta$-parametrized-structures (\ref{cast})). This
calculation leads to the following
%%.../noshearsystemeinsteinfull2008podst14.nb
\begin{theorem}
If $t=\frac14$ and a structure $(M,\la,\mu)$ with nonvanishing twist and
vanishing shear has the relative invariant
$K_1$ satisfying 
$$K_{11\bar{1}}+K_{1\bar{1}1}\equiv 0,$$
then the
Bach tensor $B_t$ of the metrics $g_t$ corresponding to the
structure $(M,\la,\mu)$, as defined in (\ref{emu}),
has a very simple form
$$B_{1/4}=\frac{3}{32}\Big(4K_{11}K_{1\bar{1}}+2i(K_{11}\bar{K}_2-K_{1\bar{1}}K_2)-7K_2\bar{K}_2-4(K_{2\bar{1}0}+\bar{K}_{210})\Big)~\theta^3\odot\theta^3,$$
in which \emph{nine} out of the \emph{apriori} ten components, identically vanish. 
\end{theorem}
Apart from the structures with $\beta_{S_1}$ and $\beta_{S_2}$ we do
not know examples of structures satisfying condition $K_{11\bar{1}}+K_{1\bar{1}1}\equiv 0$.
\section{Application 2: Algebraically special spacetimes}\label{cabb}
All the metrics discussed in Section \ref{cab} are examples of 
\emph{algebraically special spacetimes}. These are 4-dimensional Lorentzian
metrics, whose Weyl tensor is \emph{degenerate} in an open region of
the spacetime. The algebraically special \emph{vacuum} 
(or in other words: \emph{Ricci flat}) metrics have the interesting property
that they define a congruence of shearfree and null geodesics in the 
underlying \emph{spacetime}. At this stage we must emphasize that the
congruence associated with such metrics lives in \emph{four} 
dimensions and the vanishing \emph{shear} and the \emph{geodesic}
condition is a \emph{four} dimensional notion here. Nevertheless we
observe that the 3-dimensional oriented congruences 
in our sense are related, at least at the level of the Lorentzian metrics
discussed so far, to an analogous notion in 3+1 dimensions, where the
metric is of Lorentzian signature. In this section we discuss this
relationship more closely. Note that in \emph{all} the examples of
Section \ref{cab} the \emph{four}-dimensional congruence of shearfree
null 
geodesics was always tangent to the vector field $k=\partial_r$. 

Before passing to the subject proper of this section we remark that
the algebraically special Lorentzian metrics are very important in
physics. To be more specific we consider the metric 
%/pawel/notebooks/denny/kerr.nb
\be
g=2\Big({\mathcal P}^2\mu\bar{\mu}+\lambda(\der r+{\mathcal
  W}\mu+\bar{\mathcal W}\mu+{\mathcal H}\lambda)\Big),\label{kke}\ee
where  
%\newpage
%$$\la=\der u-\frac{i(2M+(a+M)z\bar{z}}{z(1+\tfrac{k}{2}z\bar{z})^2}\der
%z+\frac{i(2M+(a+M)z\bar{z}}{\bar{z}(1+\tfrac{k}{2}z\bar{z})^2}\der
%\bar{z},\quad\quad\mu=\der z,$$
%$$P^2=\frac{r^2}{(1+\tfrac{k}{2}z\bar{z})^2}+\frac{(k
%  M-a+(kM+a)\tfrac{k}{2}z\bar{z})^2}{(1+\tfrac{k}{2}z\bar{z})^4},$$
%$$W=-\frac{ika\bar{z}}{(1+\tfrac{k}{2}z\bar{z})^2},$$
%$$H=-\frac{k}{2}+\frac{-mr+kM^2-aM\frac{1-\tfrac{k}{2}z\bar{z}}{1+\tfrac{k}{2}z\bar{z}}}{r^2+\frac{\big(k
%  M-a+(kM+a)\tfrac{k}{2}z\bar{z}\big)^2}{(1+\tfrac{k}{2}z\bar{z})^2}}$$
%Once again:
%$$g=2\Big(P^2\mu\bar{\mu}+\lambda(\der r+W\mu+\bar{W}\mu+H\lambda\Big),$$
%where  
$$\la=\der
u+\frac{i\big(2M+(a+M)z\bar{z}\big)}{z(1+\tfrac{K}{2}z\bar{z})^2}\der
z-\frac{i\big(2M+(a+M)z\bar{z}\big)}{\bar{z}(1+\tfrac{K}{2}z\bar{z})^2}\der
\bar{z},\quad\quad\mu=\der z,$$
$${\mathcal P}^2=\frac{r^2}{(1+\tfrac{K}{2}z\bar{z})^2}+\frac{\big(K
  M-a+(KM+a)\tfrac{K}{2}z\bar{z}\big)^2}{(1+\tfrac{K}{2}z\bar{z})^4},$$
\be
{\mathcal W}=\frac{iKa\bar{z}}{(1+\tfrac{K}{2}z\bar{z})^2},\label{kke1}\ee
$${\mathcal H}=-\frac{K}{2}+\frac{mr+KM^2-aM~\frac{1-\tfrac{K}{2}z\bar{z}}{1+\tfrac{K}{2}z\bar{z}}}{r^2+\frac{\big(K
  M-a+(KM+a)\tfrac{K}{2}z\bar{z}\big)^2}{(1+\tfrac{K}{2}z\bar{z})^2}},$$
and $m,a,M,K$ are \emph{real} constants.

This scary-looking metric has very interesting properties. First, it
admits a 4-dimensional congruence of null and shearfree geodesics,
which is tangent to the vector field $k=\partial_r$. Second, if $K=1$, it is
\emph{algebraically special}, actually of Petrov type $D$, and more
importantly, it is \emph{Ricci flat}. The parameter values $K-1=M=0$,
correspond to the celebrated \emph{Kerr metric}, describing a
gravitational field outside a \emph{rotating black hole}, with \emph{mass}
$m$ and \emph{angular momentum} parameter $a$. In this case the
angular momentum parameter $a$ measures the \emph{twist} of the
congruence tangent to $k$. If in addition $a=0$, the twist of
the congruence vanishes, and the metric becomes the \emph{Schwarzschild
  metric}. Third, in the $K-1=a=m=0$ case the metric is the \emph{Taub-NUT
  vacuum metric}, which is important in Relativity Theory because its serves as
a `counterexample for almost everything' \cite{misner}. Fourth, 
it should be also noted that if $M=0$ and the other parameters,
including $K$, are arbitrary, the metric is again type $D$ and \emph{Ricci
  flat}. Finally, we should mention that for general values of $K\neq 1$ and $M\neq 0$ the metric is algebraically \emph{general} and \emph{neither} Ricci flat \emph{nor} Einstein.         

From the point of view of our paper the relevance of the metric
(\ref{kke})-(\ref{kke1}) is self evident. The four dimensional
spacetime ${\mathcal M}$ on which the metric is defined, locally parametrized by
$(u,z,\bar{z},r)$, is locally a product ${\mathcal M}=M\times \bbR$, with $M$
being parametrized by $(u,z,\bar{z})$. The 3-dimensional manifold $M$
is then naturally equipped with the oriented congruence structure
$(M,\la,\mu)$, defined in terms of the 1-forms $\la,\mu$ from
(\ref{kke1}). Note that these forms, although defined on $\mathcal{M}$, do not
depend on the $r$ coordinate, and as such project to $M$. Note also
that the oriented congruence structure defined by these forms has
always vanishing shear $s\equiv 0$. It has nonvanishing twist,
with the exception of the Schwarzschild metric $a=M=0$, or the case
when $K=0$ and $M+a=0$. In this last case the metric is of Petrov type $D$, but  is \emph{neither} Ricci flat \emph{nor} Einstein.

Since in the case of \emph{Ricci flat} metrics
(\ref{kke})-(\ref{kke1}) only the Schwarzschild metric has the
corresponding structure of an oriented congruence with vanishing twist,
in the next sections we decided to make a systematic study of the Lorentzian
metrics (\ref{kke}) (not necessarily of the form (\ref{kke1})), with
forms $\la,\mu$ defining an oriented congruence structure in
\emph{three} dimensions which have vanishing shear, but nonvanishing
twist, only. Actually, for the sake of brevity, we only discuss the
case when the structural invariants $K_1$ and $K_2$ of the congruence
structures, as defined in Section \ref{ncrn}, satisfy $K_1\neq 0$,
$K_2\equiv 0$. 

\subsection{Reduction of the Einstein equations}
As we know from Section \ref{K1nK2} every structure
$(M,[\lambda,\mu])$ having $K_1\neq 0$, $K_2\equiv 0$ defines an
invariant coframe $(\omega,\om_1,\bar{\om}_1)$ on $M$ which satisfies
the system (\ref{systek}), (\ref{itg}). Given such a structure we
consider a 4-manifold ${\mathcal M}=\bbR\times M$
with a distinguished class of Lorentzian metrics. These 
metrics can be written using any representative of
a class $[\lambda,\mu]$. Since the invariant forms $(\om,\om_1)$
provide us with such a representative it is natural to use them,
rather than a randomly chosen pair $(\lambda,\mu)$. Thus, given a
structure $(M,[\lambda,\mu])$ having $K_1\neq 0$, $K_2\equiv 0$, we
write a metric on  
\be
{\mathcal M}=\bbR\times M\label{rro1}\ee 
as
\be\label{rro}
g=P^2~[~2\om_1\bar{\om}_1+2\om(\der r+W\om_1+\bar{W}\bar{\om}_1+H\om)~].
\ee 
Here the forms $(\om,\om_1,\bar{\om}_1)$ satisfy the system
(\ref{systek}), (\ref{itg}), $r$ is a coordinate along the $\bbR$ factor
in $\mathcal M$, and $P\neq 0$, $H$ (real) and $W$ (complex) are
arbitrary functions on $\mathcal M$. 

The null vector field $k=\partial_r$ is tangent to a congruence of
twisting and shear-free null geodesics in $\mathcal M$. This is a
distinguished geometric structure on $\mathcal M$.

Now we pass to the question if the metrics (\ref{rro}) may be
Einstein. To discuss this we need to specify what is the interesting
energy momentum tensor that will constitute the r.h.s. of the Einstein
equations. Since the only geometrically distinguished structure on
$\mathcal M$ is the shear-free congruence generated by $k=\partial_r$
it is natural to consider the Einstein equations in the form
\be
Ric(g)=\Phi k\odot k.\label{pur}
\ee   
If the real function $\Phi$ satisfies $\Phi>0$ the above equations
have the physical interpretation of a gravitational field of `pure
radiation' type in which the gravitational energy is propagated with
the speed of light along the congruence $k$. If $\Phi\equiv 0$ we have
just Ricci-flat metrics, which correspond to vacuum gravitational
fields. This last possibility is not excluded by our Einstein
equations. In the following analysis we will not insist on the
condition $\Phi\equiv 0$. 

At this point it is worthwhile to mentioned that a similar problem was
studied by one of us some years ago in \cite{phd}; see also the more
modern treatment in \cite{indiana}. Using the results
of \cite{indiana,phd} and the symbolic calculation program Mathematica, we reduced
the Einstein equations (\ref{pur}) to the following form:

First, it turns out that the Einstein equations (\ref{pur}) can be
fully integrated along $k$, so that the $r$ dependence of
the functions $P$, $H$, $W$ is explicitly determined. Actually we have:
\begin{eqnarray}
P&=&\frac{p}{\cos\frac{r}{2}}\nonumber\\
W&=&i\al{\rm e}^{-ir}+\bet\label{pur1}\\
H&=&-\frac{\bar{m}}{p^4}{\rm e}^{2ir}-\frac{m}{p^4}{\rm
  e}^{-2ir}+\tfrac{1}{2}\bar{\phi}{\rm e}^{ir}+
\tfrac{1}{2}\phi{\rm e}^{-ir}+\tfrac{1}{2}\chi,\nonumber
\end{eqnarray}
where the functions $p,\chi$ (real) and $\al,\bet,m$ (complex) do not
depend on the $r$ coordinate. Thus, using some of the Einstein
equations (\ref{pur}), one quickly reduces the problem from $\mathcal
M$ to a system of equations on the CR-manifold with preferred
splitting $(M,[\lambda,\mu])$.

Now we introduce a preferred set of vector fields
$(\partial_0,\partial,\bar{\partial})$ on $M$ defined as the
respective duals of the preferred forms $(\om,\om_1,\bar{\om}_1)$. Note
that this notation is in agreement with the notation of 
CR-structure theory. In particular $\bar{\partial}$ is the tangential
CR-operator on $M$, so that the equation for a CR-function $\xi$ on
$M$ is $\bar{\partial}\xi=0$.

With this notation the remaining Einstein equations (\ref{pur}) for
$\der s^2$ give first:
\begin{eqnarray}
\al&=&2(\partial\log p-c)\nonumber\\
\bet&=&2i(\partial\log p-2c-A_1)\label{pur2}\\
\phi&=&(\bar{\partial}+A_1+i\bar{B}_1+i\bar{\bet})\al-4\frac{m}{p^4}\nonumber\\
\chi&=&3\al\bar{\al}+2i(\partial+A_1-iB_1)\bar{\bet}-2i(\bar{\partial}+A_1+i\bar{B}_1)\bet\mp 1,\nonumber
\end{eqnarray}
where we have introduced a new unknown complex function $c$ on $M$ and used
the Cartan invariants $A_1>0$, $B_1$ and $\pm1$ of the system (\ref{systek}),
(\ref{itg}).

Finally the differential equations for the unknown functions $c,m$
and $p$ equivalent to the Einstein equations (\ref{pur}) are:
\begin{eqnarray}
(\partial-3A_1+iB_1)c-2c^2+a_{11}-A_1^2+\tfrac{i}{2}A_1(3B_1+\bar{B}_1)=0\label{pur31}\\\nonumber\\
(\bar\partial-6\bar{c})m=0\label{pur32}\\\nonumber\\
(\partial+3A_1-iB_1)\bar{\partial}p+(\bar{\partial}+3A_1+i\bar{B}_1)\partial
p+\nonumber\\
-3[(\partial+3A_1-iB_1)\bar{c}+(\bar{\partial}+3A_1+i\bar{B}_1)c+2c\bar{c}+\nonumber\\
\tfrac{8}{3}A_1^2+\tfrac{4}{3}a_{11}+\tfrac{2i}{3}A_1(\bar{B}_1-B_1)\pm\tfrac{1}{6}]p=\label{pur33}\\
-\frac{m+\bar{m}}{p^3}.\nonumber
\end{eqnarray}

We thus have the following theorem.
\begin{theorem}
Let $(M,[\lambda,\mu])$ be a structure of an oriented
congruence having vanishing shear, nonvanishing
twist and the invariants $K_1\neq 0$, $K_2\equiv 0$.
Then a Lorentzian metric associated with $(M,[\lambda,\mu])$
via (\ref{rro1})-(\ref{rro}) satsifies the Einstein equations
(\ref{pur}) if and only if the metric functions are given by means of
(\ref{pur1})-(\ref{pur2}) with the unknown functions $c,m$ (complex),
$p$ (real) on $M$ satsifying the differential equations (\ref{pur31})-(\ref{pur33}).
\end{theorem}

\begin{remark}
Note that contrary to the invariants $(\om,\om_1,\bar{\om}_1)$ the
coordinate $r$, and in turn the differential $\der r$, has no geometric
meaning. Actually the coordinate freedom in choosing $r$ is $r\to r+f$,
where $f$ is any real function $f$ on $M$. This induces some gauge
transformations on the variables $\bet$ and $\chi$. Nevertherless
the equations (\ref{pur31})-(\ref{pur33}) are not affected by these transformations.
\end{remark}

\begin{remark}\label{solving}
Equations (\ref{pur31})-(\ref{pur33}) should be understood in the following way. Start
with a structure of an oriented congruence $(M,[\lambda,\mu])$ 
having vanishing shear, nonvanishing
twist and the invariants $K_1\neq 0$, $K_2\equiv 0$. Calculate its
invariants $(\om,\om_1,\bar{\om}_1)$,
$(\partial_0,\partial,\bar{\partial})$, $A_1$, $B_1$, $a_{11}$ of
(\ref{systek}), (\ref{itg}). Having this data write down equations
(\ref{pur31})-(\ref{pur33}) for the unknowns $c,m,p$. As a hint for solving these
equations observe that the equation (\ref{pur31}) involves only the 
unknown $c$. Thus solve it first. Once having the general solution for $c$
insert it to the equation (\ref{pur32}). Then this equation
becomes an equation for the unknown $m$. In particular $m=0$ is always
a solution of (\ref{pur32}). Once this equation for $m$ is solved, insert $c$ and $m$
to the equation (\ref{pur33}), which becomes a real, second order
equation for the real unknown $p$. In particular, if it 
happens that you are only interested in solutions
for which $m+\bar{m}=0$, this equation is a linear second order PDE on
$M$. For particular choices of $(M,[\lambda,\mu])$ it can be reduced to
well known equations of mathematical physics, such as for example the
hypergeometric equation \cite{phd}.  
\end{remark}
\begin{remark}
The unknown variable $m$ is related to a notion known to physicists as 
{\it complex mass}. For physically interesting solutions, such as 
for example the Kerr black hole, the imaginary part of $m$ is related to 
the mass of the gravitational source. The real part of $m$ 
is related to the so called NUT parameter. Moreover $m$
is responsible for algebraical specialization of the Weyl tensor of
the metric. If $m\equiv 0$ the metric is of type III, or its
specializations, in the
Cartan-Petrov-Penrose algebraic classification of gravitational fields.
\end{remark}
\subsection{Examples of solutions}
Here we give examples of metrics (\ref{rro}) satisfying the Einstein
equations (\ref{pur}). In all these examples the structures of oriented
congruences $(M,[\lambda,\mu])$ will be isomorphic to the structures 
with a 3-dimensional group of symmetries described by Proposition
\ref{pr3s1}. The invariant forms $(\om,\om_1,\bar{\om}_1)$ for these 
structures are:
\begin{eqnarray}
&&\om=\frac{2\tau^2}{1\mp 4\tau^2}(y^{-2(1\mp 2\tau^2)}\der
u-y^{-1}\der x),\nonumber\\
&&\om_1=\pm i\tau y^{-1}(\der x+i\der
y),\label{is}\\
&&\bar{\om}_1= \mp i\tau y^{-1}(\der x-i\der y).\nonumber
\end{eqnarray}
We recall that the real parameter $\tau$ is related to the invariants
$A_1$, $B_1$ of the structures (\ref{is}) via:
$$A_1=-\frac{\mp 1+2\tau^2}{2 \tau},\quad\quad B_1=i\tau.$$
Since these invariants are \emph{constant}, all the higher order
invariants for these structures, such as for example the $a_{11}$ in (\ref{itg}), are
\emph{identically vanishing}. Although Propsition \ref{pr3s1} excludes
the values $\tau^2=\tfrac{1}{2}$ in the upper sign case, we include it
in the discussion below. This value corresponds to $A_1=0$ and
therefore must describe one of the two nonequivalent 
structures $(M,[\lambda,\mu])$ of Example \ref{epsi}. From the two structures of this example, the one corresponding to $\tau^2=\frac12$ is defined by 
$(\epsilon_1,\epsilon_2)=(0,1)$. In particular, it has a strictly 
4-dimensional symmetry group. 

First we assume that the metric (\ref{rro}) has the same \emph{conformal
symmetries} as the structures (\ref{is}). This assumption, together with 
Einstein's equations (\ref{pur}), which are equivalent to the equations
(\ref{pur1})-(\ref{pur2}), (\ref{pur31})-(\ref{pur33}),
implies that \emph{all the metric functions} $p,m,c$ \emph{must be constant}. Then
the system (\ref{pur31})-(\ref{pur33}) reduces to the following 
algebraic equations for $m,p,c$:
\begin{eqnarray}
&&(-3A_1+iB_1)c-2c^2-A_1^2+\tfrac{i}{2}A_1(3B_1+\bar{B}_1)=0\label{pur313}\\\nonumber\\
&&\bar{c}m=0\label{pur323}\\\nonumber\\
&&3[(3A_1-iB_1)\bar{c}+(3A_1+i\bar{B}_1)c+2c\bar{c}+\label{pur333}\\
&&
\tfrac{8}{3}A_1^2+\tfrac{2i}{3}A_1(\bar{B}_1-B_1)\pm\tfrac{1}{6}]p=
\frac{m+\bar{m}}{p^3}.\nonumber
\end{eqnarray}
Thus we have two cases.
\begin{itemize}
\item Either $c=0$
\item or $m=0$.
\end{itemize}
Strangely enough in both cases equations (\ref{pur313})-(\ref{pur333})
admit solutions \emph{only} for the \emph{upper} sign in (\ref{pur333}).

If $c=0$ then we have only
one solution corresponding to 
$\tau=\pm\frac{1}{\sqrt{2}}$ 
with \emph{arbitrary} constant $p\neq 0$ and $m=\frac{p^4}{4}+i M$,
where $M$ is \emph{real} constant. The
corresponding metric 
$$\der s^2=\frac{p^2}{\cos^2\tfrac{r}{2}}[\frac{\der x^2 + \der y^2}{y^2} + 
      2(\frac{\der x}{y} - \der u )(\der r - 
            2\cos^2\tfrac{r}{2}(\cos r + 
                  4M\sin r)(\frac{\der x}{y} - \der u )]$$
is \emph{vacuum} i.e. it satisfies equations (\ref{pur}) with
      $\Phi\equiv 0$.

If $m=0$ then $p\neq 0$ is an arbitrary constant, and 
we have the following solutions:
\begin{itemize}
\item $\tau=\frac{\epsilon_1}{4}\sqrt{5+\epsilon_2\sqrt{17}}$,
  $c=-\frac{\epsilon_1}{\sqrt{5+\epsilon_2\sqrt{17}}}$,
\item $\tau=\frac{\epsilon_1}{2}\sqrt{\frac{1}{2}(7+\epsilon_2\sqrt{17})}$,
  $c=\frac{\epsilon_1}{4}\sqrt{\frac{1}{2}(7+\epsilon_2\sqrt{17})}(3+\epsilon_2\sqrt{17})$.
\end{itemize}
Here $\epsilon_1^2=\epsilon_2^2=1$. Sadly, irrespectively of the signs of
  $\epsilon_1,\epsilon_2$, all these solutions have
  $\Phi=const<0$, and as such do not correspond to physically
  meaningful sources.

%/home/pawel/notebooks/denny/sprawdzeniasekcja11mostgeneral.nb
As the next example we still consider structures $(M,[\la,\mu])$ with the
invariants (\ref{is}), and assume that the metrics have only \emph{two}
conformal symmetries $\partial_u$ and $\partial_x$. For simplicity we
consider only solutions with $m=0$ in (\ref{pur32}). Under these
assumptions we find that the general solution of
(\ref{pur31})-(\ref{pur33}) includes a free real
parameter $t$ and  is given by  
\be
c=\frac{-2+4\tau^2}{4\tau}+\frac{1-4\tau^2}{4\tau}\frac{1}{1-t y^{(4\tau^2-1)}},\label{ccec}\ee 
with the real function $p=p(y)$ satisfying a linear 2nd order ODE:
\begin{eqnarray}
&&4y(y-ty^{4\tau^2})^2~[~yp''~+~(4\tau^2-2)p'~]+\nonumber\\
&&[(-32 \tau^4+20\tau^2-1)y^2+4t^2(4\tau^4-7\tau^2+2)y^{8\tau^2}-\label{solp}\\
&&16t(8\tau^4-5\tau^2+1)y^{(4\tau^2+1)}]p=0.\nonumber
\end{eqnarray}
If this equation is satisfied, the only \emph{a'priori} nonvanishing
component of the Ricci tensor is 
\begin{eqnarray*}
&&R_{33}=-\tfrac18~\Big(\frac{\cos(\tfrac{r}{2})}{\tau
    (y-ty^{4\tau^2})p}\Big)^{4}\times\\
&&\Big(\big((8\tau^2-3)(128\tau^6-160\tau^4+92\tau^2-21)y^4+\\
&&8t^4\tau^2(32\tau^6+8\tau^4-28\tau^2+9)y^{16\tau^2}+\\
&&4t(8\tau^2-3)(256\tau^6-248\tau^4+58\tau^2+3)y^{3+4\tau^2}+\\
&&36t^2(4\tau^4+\tau^2-1)(32\tau^4-12\tau^2-1)y^{2+8\tau^2}+\\
&&16t^3\tau^2(128\tau^6-184\tau^4+122\tau^2-27)y^{1+12\tau^2}\big)p^2-\\
&&4y(y-ty^{4\tau^2})\big((8\tau^2-3)(16\tau^4-3)y^3+4t^3\tau^2(16\tau^4-3)y^{12\tau^2}+\\
&&6t(8\tau^2-3)y^{2+4\tau^2}+96t^2\tau^2(1-2\tau^2)^2y^{1+8\tau^2}\big)pp'+\\
&&4y^2(y-ty^{4\tau^2})^2\big((8\tau^2-3)y+4t\tau^2y^{4\tau^2}\big)^2{p'}^2
\Big).
\end{eqnarray*}
It follows that this $R_{33}$, with $p$ satisfying (\ref{solp}), \emph{may} identically vanish for some
values of parameter $\tau$. This happens \emph{only} when the parameter
$t=0$. If $$t=0$$ the values of $\tau$ for which $R_{33}$ may be
identically zero and
for which the function $p=p(y)$ satisfies (\ref{solp}) are:
$$\tau=\pm\tfrac12\sqrt{2},\quad\quad\tau=\pm\tfrac12\sqrt{\frac32},\quad\quad\tau=
\pm\tfrac12\sqrt{\frac53},\quad\quad\tau=\pm\tfrac12\sqrt{3},$$
$$\tau_-=\pm\tfrac12\sqrt{\tfrac16(11-\sqrt{13})},\quad\quad\quad\tau_+=\pm\tfrac12\sqrt{\tfrac16(11+\sqrt{13})}.$$
Of these \emph{distinguished} values the most interesting (modulo sign) are 
the last two, $\tau_-$ and $\tau_+$, 
%$$\tau_-=\pm\tfrac12\sqrt{\tfrac16(11-\sqrt{13})},\quad\quad{\rm
%  and}\quad\quad\tau_+=\pm\tfrac12\sqrt{\tfrac16(11+\sqrt{13})},$$
since for them the corresponding metrics (\ref{rro}) may be \emph{vacuum} and \emph{not}
conformally flat. Actually, restricting our attention to the plus
signs above and assuming $t=0$, we have the following possibilities:
\begin{itemize}
%/home/pawel/notebooks/denny/sprawdzeniasekcja11mostgeneralsp4.nb
%/home/pawel/notebooks/denny/sprawdzeniasekcja11mostgeneralsp5.nb
%/home/pawel/notebooks/denny/sprawdzeniasekcja11mostgeneralsp6.nb
\item $\tau_\varepsilon=\tfrac12\sqrt{\tfrac16(11+\varepsilon\sqrt{13})}$,
  $\varepsilon=\pm 1$; for these two values of $\tau$ the general
  solution of (\ref{solp})
  is $$p_\varepsilon=y^{\tfrac{1}{12}(1-\varepsilon\sqrt{13})}(s_2+s_1 y),$$ and the only potentially nonvanishing component of the Ricci tensor is 
$$R_{33}=-\tfrac49 (7+\varepsilon\sqrt{13})~s_2^2~y^{-\tfrac16(1-\varepsilon\sqrt{13})}~\Big(\frac{\cos\tfrac{r}{2}}{s_2+s_1y}\Big)^4.$$
This \emph{vanishes} when $s_2=0$. If $s_2=0$ the corresponding metrics
$g_\varepsilon$, as defined in (\ref{rro}), read
%%sprawdzone w 
%%/denny/einsteingrii.nb
\begin{eqnarray*}
&&g_{\varepsilon}=2{P}^2\Big(\om_1\bar{\om}_1
+\om\big(\der r+{W}\om_1+\bar{W}\bar{\om_1}+\tfrac{3+(9-20\tau_{\varepsilon}^2)\cos r}{12\tau_{\varepsilon}^2}\om\big)\Big),
\end{eqnarray*}
with
$${P}=\tfrac{s_1y^{2(1-\tau_{\varepsilon}^2)}}{\cos\tfrac{r}{2}},\quad\quad{W}=i\tfrac{2(20\tau_{\varepsilon}^2-9)+(8\tau_{\varepsilon}^2-9){\rm e}^{-ir}}{24\tau_{\varepsilon}^3},$$
and $\om,\om_1,\bar{\om}_1$ given by 
(\ref{is}).
For both values of $\varepsilon=\pm1$ the metric is 
\emph{Ricci flat} and of Petrov type $III$. In particular it is \emph{neither} flat, \emph{nor} of type $N$.
\end{itemize}
In all other cases of the \emph{distinguished} $\tau$s the corresponding
\emph{vacuum} metrics are the \emph{flat} Minkowski metrics. In fact,
\begin{itemize}
%/home/pawel/notebooks/denny/sprawdzeniasekcja11mostgeneral7.nb
\item if $\tau=\tfrac12\sqrt{\frac32}$, the general solution to
  (\ref{solp}) is $$p=s_1\sqrt{y}+s_2y,$$
and the corresponding metric (\ref{rro}) is flat.
%/home/pawel/notebooks/denny/sprawdzeniasekcja11mostgeneralsp1.nb:
\item if $\tau=\tfrac12\sqrt{\frac53}$, the general solution
  to (\ref{solp}) is $$p=y^{\tfrac23}(s_1+s_2\log y),$$ and the
  potentially nonvanishing Ricci component $R_{33}$ is
$$R_{33}=-\tfrac{8}{25}s_2(2s_1+s_2+2s_2\log
  y)\Big(\frac{\cos\tfrac{r}{2}}{(s_1+s_2\log
    y)y^{\tfrac13}}\Big)^4.$$
This vanishes when $s_2=0$. In such case 
  the metric is flat.
%/home/pawel/notebooks/denny/sprawdzeniasekcja11mostgeneralsp2.nb
\item if $\tau=\tfrac12\sqrt{2}$, the general
  solution of (\ref{solp}) is $$p=\sqrt{y}(s_1+s_2\log y),$$ and 
$$R_{33}=-\frac{2s_2^2}{y}~\Big(\frac{\cos\tfrac{r}{2}}{s_1+s_2\log
    y}\Big)^4;$$
this vanishes when $s_2=0$; in such case the metric is flat.
%/home/pawel/notebooks/denny/sprawdzeniasekcja11mostgeneralsp3.nb
\item if $\tau=\tfrac12\sqrt{3}$, the general
  solution of (\ref{solp}) is $$p=s_1y+s_2y^{-1},$$ and 
$$R_{33}=-32s_2^2y^2~\Big(\frac{\cos\tfrac{r}{2}}{s_2+s_1y^2}\Big)^4;$$
this vanishes when $s_2=0$; in such case the metric is the flat
Minkowski metric.
\end{itemize}

%/sprawdzeniasekcja11mostgeneral1lambda.nb
We close this section with an example of a metric that goes a bit
beyond the formulation of the Einstein equations presented
here. Remaining with the structures of an oriented congruence with the
upper sign in (\ref{is}), we take $c$ as in (\ref{ccec}) with $t=0$, and 
consider the metric (\ref{rro}), (\ref{pur1}), (\ref{pur2}) with a 
\emph{constant} function $p$ given by 
$$p=\frac{\sqrt{3}}{4
  s\tau}\sqrt{\varepsilon(-1+20\tau^2-32\tau^4)}.$$
Here the $\varepsilon$ is $\pm 1$, and is chosen to be such that the
value $\varepsilon(-1+20\tau^2-32\tau^4)$ is positive; $s$ is a
nonzero constant. A short
calculation shows that the Ricci tensor for this metric has the
following form
$$Ric=(\tau^2-1)(8\tau^2-5)\frac{16\Lambda(4\tau^2+1)\cos^4\tfrac{r}{2}}{3\tau^2(1-20\tau^2+32\tau^4)}k\odot
k+\Lambda g.$$
Thus, this metric is \emph{Einstein}, with cosmological constant equal
to $\Lambda=\varepsilon s^2$, provided that  
$$\tau=\pm1,\quad\quad {\rm or}\quad\quad
\tau=\pm\tfrac12\sqrt{\frac{5}{2}}.$$
It is remarkable that the Einstein metric 
%%sprawdzone 
%%/notebooks/denny/einsteingriilambda.nb
$$g=-\frac{3}{5\Lambda\cos^2\tfrac{r}{2}}\Big(\om_1\bar{\om}_1+\om\big(\der r+
\tfrac{i(2{\rm e}^{-ir}+5)}{\sqrt{10}}\om_1-\tfrac{i(2{\rm
    e}^{ir}+5)}{\sqrt{10}}\bar{\om}_1+\tfrac{7}{10}(3+2\cos r)\om\big)\Big),$$
corresponding to
$\tau=\pm\tfrac12\sqrt{\frac{5}{2}}$, is of \emph{Petrov type} $N$ with the
quadruple principal null direction of the Weyl tensor being \emph{twisting}. It was first obtained by Leroy \cite{leroy} and
recently discussed in \cite{nurtn}. The Einstein metric 
%%sprawdzone 
%%/notebooks/denny/einsteingriilambda.nb
$$g=-\frac{39}{8\Lambda\cos^2\tfrac{r}{2}}\Big(\om_1\bar{\om}_1+\om\big(\der r+
\tfrac{i({\rm e}^{-ir}+4)}{2}\om_1-\tfrac{i({\rm
    e}^{ir}+4)}{2}\bar{\om}_1+\tfrac{5}{8}(3+2\cos r)\om\big)\Big),$$
corresponding
to $\tau=\pm1$ is of Petrov type $III$. 
\subsection{Discussion of the reduced equations} 
Here we discuss the integration procedures for equations
(\ref{pur31})-(\ref{pur33}) 
along the lines indicated in Remark \ref{solving}. We start with 
equation (\ref{pur31}). This is an equation for the unknown
$c$. Remarkably, the existence of a function $c$ satisfying this equation
is equivalent to an existence of a certain CR function $\eta$ on
$M$. To see this we proceed as follows. 
We consider a 1-form $\Pi$ on $M$ given by
\be
\Pi=\om_1+2i(A_1+\bar{c})\om,\label{sso}
\ee
where $c$ is an arbitrary complex function on $M$. Of course 
\be\Pi\dz\bar{\Pi}\neq 0,
\label{sso1}\ee
since otherwise the forms $\om_1$ and $\bar{\om}_1$ would not be independent. 
Now using the differentials $\der\om$, $\der\om_1$, $\der A_1$
given in (\ref{systek}), (\ref{itg}), we easily find that 
$$\der\Pi\dz\Pi=2i~[~(\bar{\partial}-3A_1-i\bar{B}_1)\bar{c}-2\bar{c}^2+a_{11}-A_1^2-\tfrac{i}{2}A_1(3\bar{B}_1+B_1)~]\om_1\dz\bar{\om}_1\dz\om.$$
Thus our equation (\ref{pur31}) is satisfied for $c$ if and only
if $\der\Pi\dz\Pi=0$. Due to 
our Lemma \ref{le}, $\Pi$ satisfying $\der\Pi\dz\Pi=0$ defines a complex valued
function $\eta$ on $M$ such that $\Pi=h\der\eta$. Because of (\ref{sso1})
we have  $h\bar{h}\der\eta\dz\der\bar{\eta}\neq 0$. Furthermore, since 
$\Pi$ is given by (\ref{sso}) then $\Pi\dz\om\dz\om_1=0$, which after
factoring out by $h$ gives $\der\eta\dz\om\dz\om_1=0$. Thus $\eta$ is a
CR-function on $M$. 

Conversely, suppose that we have a CR-function $\eta$ on $M$ such that
\be\der\eta\dz\der\bar{\eta}\neq 0.\label{ndeg}\ee Then the
three one forms $\om_1$, $\om$ and $\der\eta$ are linearly dependent
at each point. Thus there exist complex functions $x$, $y$ on $M$
such that \be\der\eta=x\om_1+y\om.\label{te}\ee Due to the nondegenarcy condition
(\ref{ndeg}) we must have
$x\bar{x}\om_1\dz\bar{\om}_1+x\bar{y}\om_1\dz\om-\bar{x}y\bar{\om_1}\dz\om\neq
0$, so that the complex function $x$ must be nonvanishing. In such
case we may rewrite (\ref{te}) in the more convenient form
$h\der\eta=\om_1+\bar{z}\om$, where $h=1/x$ and $\bar{z}=y/x$. Now,
defining $c$ to be $c=\frac{iz}{2}-A_1$, we see that the trivially satisfied
equation
$(h\der\eta)\dz\der(h\der\eta)=0$ implies that the function $c$ must
satisfy equation (\ref{pur31}). 
Summarizing we have the following proposition.
\begin{proposition}
Every solution $\eta$ of the
tangential CR equation $\bar{\partial}\eta=0$  satisfying
$\der\eta\dz\der\bar{\eta}\neq 0$ defines a solution $c$ of
equation (\ref{pur31}). Given $\eta$, the function $c$
satisfying equation (\ref{pur31}) is
defined 
by \be
c=\frac{i}{2}\frac{\bar{y}}{\bar{x}}-A_1,\label{fc}\ee
where $\der\eta=x\om_1+y\om$. Also the converse is true: every solution 
$c$ of equation (\ref{pur31}) defines a CR function $\eta$ such that 
$\der\eta\dz\der\bar{\eta}\neq 0$. 
\end{proposition}   

\begin{remark}
Recall that the structures
$(M,[\lambda,\mu])$ satisfying the system (\ref{systek}), (\ref{itg})
admit at least one CR-function $\zeta$, since they have zero shear
$s\equiv 0$. Associated to $\zeta$, by the above Proposition, there
should be a solution $c$ of the Einstein equation (\ref{pur31}). One checks by direct calculation that 
$$c=-A_1$$
automatically satisfies (\ref{pur31}). And this is the solution $c$
asociated with $\zeta$. This is consistent with
formula (\ref{fc}), since $y\equiv 0$ means that
$\der\eta\dz\der\zeta\equiv 0$ (compare with (\ref{te})).
\end{remark}

We now pass to the discussion of the second Einstein equation
(\ref{pur32}). 
Equation (\ref{pur32}), the equation for the function $m$,
has a principal part resembling the tangential CR-equation. Remarkably
its solutions $m$ are also expressible in terms of CR-functions. To
see this consider an arbitrary complex valued function $\xi$ and
define $m$ to be
\be
m=[~\partial_0\xi-2i(A_1+\bar{c})\partial\xi+2i(A_1+c)\bar{\partial}\xi~]^3.\label{mm1}
\ee
Here $c$ is supposed to be a solution to the first Einstein equation
(\ref{pur31}). 
Observe, that since the vector field
$\partial_0-2i(A_1+\bar{c})\partial+2i(A_1+c)\bar{\partial}$ is real,
then given $m$ one can always locally solve for $\xi$. Our
goal now is to show that if $\xi$ is a CR-function on $M$, then $m$
given by (\ref{mm1}) satisfies equation (\ref{pur32}). To
prove this one inserts (\ref{mm1}) into 
equation (\ref{pur32}) and commutes the operators
$\bar{\partial}\partial_0$ and $\bar{\partial}\partial$. After this is
performed the equation (\ref{pur32}) for $m$ becomes the
following equation for $\xi$:
$$(\partial_0+2i\bar{\partial}(A_1+c)+2i(A_1+c)\bar{\partial}-2i(A_1+\bar{c})\partial-4i\bar{c}(A_1+c)+A_1-iB_1)\bar{\partial}\xi=0.$$
This, in particular, means that if $\xi$ is a CR-function then this
equation is satisfied automatically. Thus given a CR-function $\xi$, 
via (\ref{mm1}), we constructed $m$ which satisfies equation (\ref{pur32}). To see that all solutions $m$ of (\ref{pur32}) can be constructed in this way is a bit
more subtle (see \cite{indiana}).

\end{document}